\documentclass[10pt,reqno]{article} 
\usepackage[latin1]{inputenc}   
\usepackage{amsfonts,amsmath,layout}
\usepackage{mathrsfs}  
\usepackage{soul}
\usepackage{amssymb,mathabx,amsfonts,amsthm,amscd,stmaryrd,dsfont,esint,upgreek,constants,todonotes}
\DeclareFontFamily{OT1}{pzc}{}
\DeclareFontShape{OT1}{pzc}{m}{it}{<-> s * [1.10] pzcmi7t}{}
\DeclareMathAlphabet{\mathpzc}{OT1}{pzc}{m}{it}

\usepackage{color} 

\definecolor{Red}{cmyk}{0,1,1,0.2}
\usepackage{hyperref}

\newcommand{\N}{\mathbb N}

\newcommand{\Z}{\mathbb Z}

\newcommand{\R}{\mathbb R}

\def\R{\mathbb R}
\def\N{\mathbb N}
\def\Z{\mathbb Z}

\def\E{\mathbb E}
\def\P{\mathbb P}

\def\ep{\epsilon}
\def\rg{\rangle} 
\def\lg{\langle}

\newcommand{\be}{\begin{equation}}
\newcommand{\ee}{\end{equation}}
\def\1{{\bf 1}}

\def\ds{\displaystyle}

\newtheorem{Theorem}{Theorem}[section]

\newtheorem{Lemma}[Theorem]{Lemma}
\newtheorem{Corollary}[Theorem]{Corollary}
\newtheorem{Remark}[Theorem]{Remark}

\begin{document}

\title{Microscopic derivation of a traffic flow model with a bifurcation}
\author{\renewcommand{\thefootnote}{\arabic{footnote}}
  P. Cardaliaguet\footnotemark[1], N. Forcadel\footnotemark[2]}
\footnotetext[1]{CEREMADE, UMR CNRS 7534, Universit\'e Paris Dauphine-PSL,
Place de Lattre de Tassigny, 75775 Paris Cedex 16, France. cardaliaguet@ceremade.dauphine.fr}
\footnotetext[2]{Normandie Univ, INSA de Rouen Normandie, LMI (EA 3226 - FR CNRS 3335), 76000 Rouen, France, 685 Avenue de l'Universit\'e, 76801 St Etienne du Rouvray cedex.
France. nicolas.forcadel@insa-rouen.fr}

\maketitle

\begin{abstract} The goal of the paper is a rigorous derivation of a macroscopic traffic flow model with a bifurcation or a local perturbation from a microscopic one. The microscopic model is a simple follow-the-leader with random parameters. The random parameters are used as a statistical description of the road taken by a vehicle and its law of motion. The limit model is a deterministic and scalar Hamilton-Jacobi on a network with a flux limiter, the flux-limiter describing how much the bifurcation or the local perturbation slows down the vehicles. The proof of the existence of this flux limiter---the first one in the context of stochastic homogenization---relies on a concentration inequality and on a delicate derivation of a superadditive inequality. 
\end{abstract}

\paragraph{AMS Classification:} 35D40, 35R60, 35F20.

\paragraph{Keywords:} traffic flow, stochastic homogenization, Hamilton-Jacobi equations on networks.

\tableofcontents

\bigskip

\section*{Introduction}

In this paper we study traffic flows models with a bifurcation consisting in  a single incoming road which is divided after a junction into several outgoing ones. As a particular case our analysis contains traffic flow models on a single road  with a localized perturbation (a bottleneck for instance). There are two main classes of models to describe these situations: microscopic models, which explain how each vehicle behaves  in function of the vehicles in front; and macroscopic ones, taking the form of a conservation law in which the main unknown is the density of vehicles on the roads. Our aim is to start from simple microscopic models on a bifurcation (or on a perturbation) and derive from these models continuous ones after scaling. The point is to get a better understanding of the continuous traffic flow models arising as the limit of discrete ones. Indeed there exists many different continuous models of traffic flow on a junction or with a local perturbation in the literature \cite{AnRo, BenRos17, CrSa, GaGo11, piccoli, ViGoCh17} and the relation between these models is not completely clear. If the basic continuous model on a single straight road (the so-called LWR model, from Lighthill and Whitham \cite{LW} and Richards \cite{richards}) is well understood and justified by micro-macro limits in several contexts \cite{MR1952895, CFGG, rigorousLWR, GoatinRossi, MR3721873}, there is no consensus for problems with a junction or a bifurcation: the models are only obtained so far by heuristic arguments, with the exception of \cite{FoSa20} discussed below.  In this paper we show that the continuous model suggested in \cite{GaImMo15, IM} pops up as the natural limit of follow-the-leader models. The continuous model in \cite{GaImMo15, IM}  takes the form of a flux limited Hamilton-Jacobi equation: it is a kind of integrated form of the basic LWR model outside the junction  combined with a ``flux limiting condition'' on the junction. Our micro-macro derivation holds for a large class of follow-the-leader models, allowing for a possible heterogeneous behavior of the vehicles. \\

Our starting point is a microscopic model. Before describing it, let us recall that few discrete traffic flow models with a junction or a local perturbation exist in the literature: \cite{CHM20} discusses an interesting leader follower model with a junction including several incoming and outgoing roads: the model we present in the present paper shares similar flavors, but in the much simpler setting of a single incoming road; \cite{AnRo} presents a microscopic model of traffic with a flow limitation at a point and formally justifies the derivation of a conservation law with a discontinuous flux (but leaves the rigorous proof as an open problem); \cite{FSZ18} describes a traffic flow model with (deterministic) traffic lights and derives rigorously the continuous model (in terms of a flux limited Hamilton-Jacobi equation on the line). The only model proving micro-macro derivation in the case of a bifurcation is \cite{FoSa20}: in \cite{FoSa20} there are two outgoing roads and it is assumed (no too realistically) that every second vehicle takes a given road. In this setting the authors show that the convergence of the discrete problem to a flux limited solution of a Hamilton-Jacobi equation on a junction. One of the goals of the present paper is to introduce a more realistic model in which one replaces the deterministic rule of \cite{FoSa20} by a random one (e.g., every second vehicle {\it in average} takes a given outgoing road). The introduction of randomness in traffic flow problems is natural and can be traced back to \cite{CLB09}. The micro-macro derivation of the LWR model from a random one on a single road was established in \cite{CaFo}. Here we prove the corresponding result for a bifurcation.

\paragraph*{Short description of the microscopic model.} In our discrete model there is one incoming road and $K$ outgoing ones, where $K\in \N$, $K\geq 1$. A position on the road is given by a pair $(x,k)$ where $x$ is a real number and $k$ is a label in $\{0, \dots, K\}$. If $x$ is nonpositive, then by convention $k=0$ and the  vehicle is on the incoming road. If $x$ is positive then $k\in \{1, \dots, K\}$ and the vehicle is on the outgoing road $k$. The junction is an interval around $x=0$, say, to fix the ideas, $[-R_0,0]$. The vehicles  are labelled by $i\in \Z$. The position of the vehicle labelled $i$ at time $t$ is denoted by $U_i(t)$. The outgoing road the vehicle chooses is fixed from the beginning (independent of time) and denoted by $T_i\in \{1, \dots, K\}$. The motion of the vehicles is given by a leader-follower model: it satisfies the system of ordinary differential equations
\be\label{eq.dynintro}
 \frac{d}{dt}U_i(t)= V_{Z_i}(U_{i+1}(t)-U_i(t), U_{\ell_i}(t)-U_i(t), U_i(t)), \qquad t\geq 0,\;  i\in \Z.
\ee
We assume that all the vehicle are going or have gone through the junction and were ordered before going through the junction: $i+1$ is the label of the vehicle right in front of the vehicle $i$ before this vehicle has gone through the junction. We denote by $\ell_i$  the label of the first vehicle in front of vehicle $i$ taking the same outgoing road as $i$ (in other words, $\ell_i= \inf \{j>i, T_i=T_j\}$). Each vehicle has a type $Z_i$ encoding, on the one hand, the outgoing road the vehicle is taking or is going to take (namely,  $T_i=T(Z_i)$ for a deterministic map $T:\mathcal Z\to \{1, \dots, K\}$) and,  on the other hand, the ``behavior'' of the vehicle (for instance, if it is a truck or a race car). The velocity law $V=V_z(e_1,e_2,x)$ depends on the type $z\in \mathcal Z$ of the vehicle, the distances $e_1$ or $e_2$ to the next vehicle and the position $x$ of the vehicle. 

In order to obtain a limit model with a few unknowns and as simple as possible, we do not keep track of all the vehicles of a given type (in contrast with \cite{CKM}). Instead we prefer a statistical description and assume that the types  $(Z_i)$ of the vehicles are random, independent and with the same law (i.i.d.); as a consequence the $(T_i)$ are also i.i.d. In addition, we also suppose that  the traffic is homogeneous outside the junction: namely, we assume that, before the junction (i.e., $x\leq -R_0$), $V_z(e_1,e_2,x)$ depends only on $e_1$ and $z$, i.e., $V_z(e_1,e_2,x)=\tilde V^0_z(e_1)$. In the same way, after the junction (i.e., $x\geq 0$) we suppose that $V_z(e_1,e_2,x)=\tilde V^k_z(e_2)$   depends only on $e_2$, $k=T(z)$ and $z$. There are two main reasons to do so: first (and again in contrast with \cite{CKM}), we will see that these assumptions  yield to a relatively simple continuous scalar equations. Second, tracing the type of a vehicle (and even more the road it is going to take later on) seems an impossible task in practice: a statistical description is probably more justified, at least if the structure of the traffic is stable in time. 

For later use we denote by $\pi^k:=\P[T_i=k]$ the proportion of vehicle taking (or planning to take) road~$k$.  

\paragraph*{The convergence result and the continuous model.} For $\ep>0$, we look at the (scaled) traffic density of vehicles on each road:
$$
m^\ep(dx,k,t)= \left\{\begin{array}{ll} 
\ds \ep \sum_{i\in \Z, \ T_i=k} \delta_{\ep U_i(t/\ep)}(dx) & {\rm if }\; x> 0, \; k\in \{1, \dots, K\}\\
\ds \ep \sum_{i\in \Z} \delta_{\ep U_i(t/\ep)}(dx) & {\rm if }\; x\leq 0, \; k=0
\end{array}\right. 
$$
and want to understand the limit, as $\ep\to 0$, of $m^\ep$. For this it is convenient to integrate in space $m^\ep$ and look instead at: 
$$
\nu^\ep(x,k,t) = \left\{\begin{array}{l}
\ds \ep  (\pi^k)^{-1} \left( \sum_{i\in \Z, \ i\leq 0, \  T_i=k}  \delta_{\ep U_{i}(t/\ep)}( (x,+\infty))- \sum_{i\in \Z,\ i>0, \ T_i=k} \delta_{\ep U_{i}(t/\ep)}((-\infty,x]) \right) \\
 \hspace{9cm} {\rm if }\; x> 0, \; k\in \{1, \dots, K\}\\
\ds \ep\left(\sum_{i\in \Z, \ i\leq 0}  \delta_{U_{i}(t)}( (x,+\infty))- \sum_{i\in \Z,\ i>0} \delta_{U_{i}(t)}((-\infty,x])\right) \hspace{1.7cm}  {\rm if }\; x\leq 0, \; k=0 . 
\end{array}\right. 
$$
Note that $\partial_x \nu^\ep=- m^\ep$ if $x\leq 0$ while $\partial_x \nu^\ep=- (\pi^k)^{-1}m^\ep$ if $x\geq 0$ and $k\in \{1,\dots,K\}$. This choice ensures the map $\nu^\ep$ to be ``almost continuous'' at $0$ since the vehicles are split between the $K$ roads after the junction in proportion $\pi^k$ for the road $k$. Our main result (Theorem \ref{thm.main}) roughly  states that, under suitable assumptions on $V$ and if $\nu^\ep(\cdot, \cdot, 0)$ has a locally uniform (deterministic) limit $\nu_0(\cdot, \cdot)$ at time $t=0$, then $\nu^\ep$ has a.s. a locally uniform (deterministic) limit $\nu$ which is the unique viscosity solution to
\be\label{eq.HJintro}
\left\{\begin{array}{ll}
\ds \partial_t \nu(x,k,t) + H^k(\partial_x \nu(x,k,t))=0& {\rm if}\; x\neq 0,\; t>0\\
\ds \partial_t \nu +\max\{ \bar A,  H^{0,+}(\partial_0 \nu), H^{1,-}(\partial_1 \nu), \dots, H^{K,-}(\partial_K \nu))\}=0& {\rm at}\; x=0\\ 
\ds \nu(x,k,0)= \nu_0(x,k) & \mbox{\rm for any}\; x,k.
\end{array}\right.
\ee
The first equation is a Hamilton-Jacobi (HJ) equation in which the homogenized Hamiltonians $H^k(p)$ can be explicitly computed from the $\tilde V^k$. As we explain below it corresponds to an integrated form of the LWR equation. The second equation describes the behavior of the vehicles at the junction (reduced after scaling to $x=0$): we explain below the different terms. It roughly says that $\partial_t \nu + \bar A=0$ at $x=0$ (unless the HJ equation is satisfied at $x=0$). The real number $\bar A$ is the so-called flux limiter. This is the main unknown of the paper. It quantifies how the traffic is slowed down by the junction. We show that 
$$
A_0\leq \bar A\leq  0, \; {\rm where}\; A_0:= \max_{k\in\{0, \dots, K\}} \min_{p\in \R}  H^k(p). 
$$
When $\bar A=A_0$, the flux is not limited at all. If  $\bar A=0$, then the traffic is completely stopped by the junction (this does not happen under our assumptions). The existence of $\bar A$ is the main point of the paper, which  presents the first existence result of a flux limiter in the context of a stochastic homogenization problem. We show that  $\bar A$ can be computed as follows: 
$$
\bar A= -  \lim_{t\to+\infty} \frac{1}{t} \sharp \bigl\{ i\in \Z, \; \exists s\in [0,t], \; U_{e,i}(s)=0\bigr\},
$$
where $\sharp E$ denotes the number of elements of a set $E$, $e=(e^k)_{k=0, \dots, K}$ is such that  $H^k(-1/e^k)= \min_p H^k(p)$ for any $k\in \{0, \dots, K\}$ and  $(U_{e,i})$ is the solution to \eqref{eq.dynintro} with the ``flat'' initial condition $U_{e,i}(0)=e^ki$ (where $k=0$ if $i\leq 0$ and $k=T_i$ if $i\geq 0$). The quantity $\bar A$ can be interpreted as the maximal fraction of vehicles the junction can let pass given an amount of time. 

The introduction of Hamilton-Jacobi equations on a junction or stratified domains can be traced back to \cite{ACCT13, AOT15, AcTc15, BaBrCh13, BrHo07, CaMa13, ImMoZi13, Sh13, ShCa13}; a general theory of flux limited solutions was developed in \cite{IM} (see also \cite{BaBrChIm18}) with, as fundamental result, a comparison theorem; \cite{LiSo16, LiSo17} present different arguments for the comparison while \cite{BaCh18} proposes a general survey on the topic. 

\paragraph*{Short discussion of the problem in terms of scalar conservation law.}  
Hamilton-Jacobi equations on a junction and scalar conservation laws with discontinuous coefficients seem intimately connected, although the rigorous relationships between the two notions has not been discussed so far. We do not intend to investigate this point here but  only develop formal arguments and postpone a more detailed analysis to future works. 

For $k\in \{1, \dots, K\}$, we define the random measures 
$$
\rho^\ep(dx,k,t) = \left\{\begin{array}{ll}
\ds \ep (\pi^k)^{-1}\sum_{i\in \Z, \; T_i=k, \; U^\ep_i(t/\ep)\geq 0} \delta_{\ep U^\ep_i(t/\ep)}(dx) & {\rm if}\; x\geq 0, \; k\in \{1, \dots, K\}\\
\ds  \ep  \sum_{i\in \Z, \; U^\ep_i(t/\ep)<0} \delta_{\ep U^\ep_i(t/\ep)}(dx) & {\rm if}\; x<0, \; k=0
\end{array}\right. 
$$
The quantities are the scaled densities of the traffic on each branch of the junction. A elementary computation shows that we have, in the sense of distribution, $\rho^\ep = -\partial_x \nu^\ep$. 
According to our main result (Theorem \ref{thm.main}) $\rho^\ep$ converges a.s. and in the sense of distribution, to $\rho:=-\partial_x\nu$. As $\nu$ solves \eqref{eq.HJintro} and is Lipschitz continuous,  it is known \cite{Cas92} that $\rho$ is, outside  the junction, an $L^\infty$ entropy solution of the scalar conservation law 
\be\label{conservationlaw}
\partial_t \rho+\partial_x(f(\rho,x,k))=0\qquad {\rm for} \;x\neq 0, 
\ee
where 
$$
f(\rho,x,k)= \left\{\begin{array}{ll}
- H^k(-\rho) & {\rm if}\; x> 0\; {\rm and}\; k\in \{1, \dots, K\},\\
-H^0(-\rho) &  {\rm if}\; x< 0\; {\rm and}\; k=0,
\end{array}\right.
$$
with an initial condition given by $\rho(x,k,0)= -\partial_x \nu_0(x,k)$. 

It is well-known \cite{AMG} that an extra conditions at the junction (depending on the model) is needed to ensure the uniqueness of such a scalar conservation law. The additional equation at $x=0$ for $\nu$ in \eqref{eq.HJintro} should lead to this extra condition. It does not seem however obvious how to interpret it in terms of the limit density $\rho$.

\paragraph*{Method of proof.}

We now describe the method of proof of our main result. As it is quite involve it is convenient for this discussion to reduce  drastically the problem by considering the case of a single road on which the vehicle behave in an identical way, expect on a small zone on which they are subject to a perturbation depending on their type. This situation pops up for instance when  the vehicles are slowed down on a small portion of a road by a speed bump to which they may react in a different way depending on their size. The leader-follower model now reads 
\be \label{eq.simpleintro}
\frac{d}{dt}U_i(t)= V_{Z_i}(U_{i+1}(t)-U_i(t), U_i(t))
\ee
where $Z_i$ is as before the type of the vehicle (supposed to be an i.i.d. random variable) and where $V_z(p,x)= \tilde V(p)$ outside the perturbation $[-R_0,0]$.  Following \cite{IM} (see also \cite{FSZ18, FoSa20}), one expects the limit model to be of the form of a LWR model with a flux limiting condition at the origin. The fundamental diagram outside the perturbation is given by $H(p)= p \tilde V(-1/p)$ and the only issue is to compute the flux limiter. In the case of a deterministic model (for instance time periodic, see \cite{FSZ18}; or periodic in the type, see \cite{FoSa20}) a standard method consists in building a corrector. However in the random setting such a corrector does not necessarily exist: see for instance the discussion in \cite{LS03}. A standard way to overcome this difficult issue is to identify subadditive quantities \cite{ArSo13, LiSo05, So97}. In the case such a quantity is not directly available, a different, more quantitative approach has been developed  in \cite{ArCa18} and later used in different contexts \cite{CaSm20, FeSo17, JiZl19}. As no subadditive quantity seems  adequate in our setting we follow this alternative approach.  

The starting point is to explore what happens for the ``flat'' initial condition $U_i(0)= ei$ ($i\in \Z$), where $e>0$ is  such that $H(-1/e)= \min_p H(p)$. In the absence of a perturbation, this initial condition would be a  steady state of the problem: namely, $U_i(t)= ei+ t\tilde V(e)$ solves \eqref{eq.simpleintro} outside the perturbation. The point is to understand how this steady state solution is modified by the perturbation. For this we introduce the (random) quantity 
$$
\theta_e(t)=\inf\{ i\geq 0, \; U_{-i}(t)\leq 0\},
$$
which corresponds to the number of vehicles having gone through the perturbation at time $t$. If the problem was unperturbed,  one would have simply $\theta_e(t)\simeq t\tilde V(e)/e$. To understand if the macroscopic model is affected by the perturbation, one is therefore led to investigate the behavior  of $\theta_e(t)/t$ as $t$ tends to infinity. The existence of such a limit is the main difficulty of the work. Indeed, $\theta_e$ does not seem to enjoy any obvious sub- or superadditivity property.  Following \cite{ArCa18} the first step of the proof consists in showing that $\theta_e(t)/t$ is almost deterministic. Namely, we prove that there is a constant $C$ (depending on $e$) such that, for all $\ep\in (0,1]$ and all $t\geq C\ep^{-1}$,  
$$
\P\left[ |\theta_e(t)-\bar \theta_e(t)| \ >\ \ep t\right]\leq Ct^2 \exp\left\{ - \ep^2 t/C\right\}. 
$$
where $\bar \theta_e(t)= \E\left[\theta_e(t)\right]$. 
For this the technique developed in \cite{ArCa18} consists  in showing that the martingale
$$
M_n(t):= \E\left[ \theta_e(t)\ |\ \mathcal F_n\right]-\E\left[\theta_e(t)\right], 
$$
(where $(\mathcal F_n)$ is the filtration generated by $\{Z_{-n}, Z_{-n+1}, \dots\}$) has bounded increments, coincides with $\theta_e(t)$ for $n=[Ct]$ (where $C$ is a large constant) and then use Azuma's inequality. Although we won't study this model in detail later, it is indeed possible to show in this case that $(M_n)$ has bounded increments by using three  facts:
\begin{itemize}
\item First $M_0(t)=0$ since the randomness of $\theta_e(t)$ comes only from $Z_i$ with $i\leq 0$ (indeed, $U_i(t)= ei+\tilde V(e)t$ for $i\geq 0$ is deterministic), 
\item Second one can show that two subsequent vehicles remain at a distance not larger than $e$ before the perturbation, 
\item Third one can prove that a vehicle close to the perturbation will cross it in a controlled time. 
\end{itemize}
The next step consists in establishing that $\bar \theta_e(t)/t$ has indeed a limit as $t\to+\infty$. The difficult issue is to understand how the profile of the  solution $(U_i(t))$ at time $t$ looks like the profile of the $(U_i(0)=ei)$ at time $0$.  For this one looks at how much $(U_{i}(t))_{i\in \Z}$ is far from the unperturbed solution $(ei+\tilde V(e)t)$: namely one looks at the quantity 
$$
M_e(t):=\inf_{i\in Z} U_i(t)-e i-\tilde V(e)t. 
$$
If the traffic is slowed down by the perturbation, this quantity is expected to be nonpositive and to decrease in time. An almost finite speed of propagation argument (Lemma \ref{lem.finitespeedPSJunc}) shows that,  far from the perturbation, the solution is almost given by the steady state and therefore the infimum in $M_e(t)$ (if negative) cannot be achieved by large values of $|i|$. So there is a minimum point $i_0$ for $M_e(t)$. By the envelop theorem one expects that 
$$
0> \frac{d}{dt} M_e(t) =  \frac{d}{dt} U_{i_0}(t)- \tilde V(e) = V_{Z_i}(U_{i_0+1}-U_{i_0}(t), U_{i_0}(t))- \tilde V(e).
$$
By minimality of $i_0$, one also has
\be\label{malensfdIntro}
U_{i}(t)-U_{i_0}(t)\geq e(i-i_0)\qquad \forall i\in \Z.
\ee
The two inequalities above imply that $U_{i_0}(t)\in [-R_0,0]$ because otherwise one would have, as $\tilde V$ is nondecreasing and \eqref{malensfdIntro} holds, 
$$
0> \frac{d}{dt} M_e(t)  =\tilde V(U_{i_0+1}-U_{i_0}(t))- \tilde V(e)\geq \tilde V(e)- \tilde V(e)=0.
$$
The fact that  $U_{i_0}(t)\in [-R_0,0]$ then implies that $i_0$ is close to $-\theta_e(t)$ and thus that 
$$
M_e(t)= U_{i_0}(t)- ei_0-\tilde V(e)t = e\theta_e(t)-\tilde V(e)t +O(1).
$$
On the other hand, by \eqref{malensfdIntro}, one has
$$
U_i(t)\geq U_{i_0}(t)+ e(i-i_0) \geq e(i-i_0-R_0/e)\qquad \forall i\in \Z.
$$
Setting  $i_1= i_0+R_0/e$, we obtain by comparison that the solution $U_i$ at time $t+s$ is above the solution starting from $e(i-i_1)$: 
$$
U^\omega_i(t+s)\geq U^{\tau_{i_1}\omega}_{i-i_1}(s) \qquad \forall i\in \Z,
$$
(the shift in the $\omega$ is due to the fact that one has to shift also the types of the vehicles). By the concentration inequality and the fact that $i_0\approx -\theta_e(t)\approx -\bar \theta_e(t)$, this implies that 
$$
\bar \theta_e(t+s)\geq \bar \theta_e(s)+ \bar \theta_e(t)-C.
$$
Fekete's Lemma then implies that $\bar \theta_e(t)/t$ has a limit and therefore that $\theta_e(t)/t$ has a limit. One can also prove that this limit  gives the value $\bar A$ of the flux limiter. \\ 

The proof in the general case (a bifurcation with several outgoing roads) follows the same lines but is much more involved. Many arguments described above are no longer valid. For instance, it is no longer true that $M_0(t)$ vanishes, because, as the distribution of the vehicles at initial time on the outgoing roads is random, $\theta_e(t)$ actually depends on the behavior of all the vehicles. We overcome this issue by using the approximate finite speed of propagation. Second, the distance between two subsequent vehicle can be arbitrarily large: this is already true at the initial time on the outgoing roads. In addition, because the vehicles have different types, the maximal speed of a leader can be larger than the maximal speed of its follower.  We show however that this distance is controlled by the distance to the first ``slow'' vehicle in front of the leader (Lemma \ref{key.estiUlli-Ui}). The main consequence of this is that $M_n(t)$ cannot have bounded increments (in contrast with \cite{ArCa18} for instance; see however \cite{CaSm20}); one has to rely on more refined concentration inequalities. Finally the presence of a bifurcation (instead of a perturbation) makes the proof of the superadditive inequality much trickier: it actually relies on the delicate construction of a corrector outside the junction (Subsection \ref{subsec:correctoroutside}). 

\subsection*{Organization of the paper} 

In the first part we explain the problem and the notation, introduce the standing assumption and state the main result (Theorem \ref{thm.main}). In the second part we give several facts which are valid for any solution of the system: an estimate of the distance to the next vehicle (Lemma \ref{key.estiUlli-Ui}) and an approximation finite speed of propagation (Lemma \ref{lem.finitespeedPSJunc}). In the third part we study the time function $\theta_e(t)/t$, show its concentration (Theorem \ref{thm:concentrationTOT}) and prove its convergence (Theorem \ref{thm.kebarvartheta}). In the last part, we derive from this the behavior of the solution starting from the flat initial datum (Lemma \ref{lem.limNTOT}), infer definition of the flux limiter $\bar A$ and finally show the homogenization result. 

Throughout the paper, the letter $C$ denote a deterministic constant which may change from line to line and which depends on the data but not on time.

\section{The main result} 

\subsection{Statement of the problem}

We consider the system
\be\label{eq.SystJuncTOT}
\frac{d}{dt} U^\omega_i(t)= V_{Z^\omega_i}(U^\omega_{i+1}(t)-U^\omega_i(t), U^\omega_{\ell^\omega_i}(t)-U^\omega_i(t), U^\omega_i(t))\qquad i\in \Z, \; t\geq 0, 
\ee
where $V:\mathcal Z\times \R^3\to \R_+$ is Lipschitz continuous in the three last variables (uniformly in the $z-$variable), nondecreasing with respect to the two middle ones and bounded by $\|V\|_\infty$. The type of the vehicle $i\in \Z$ is the random variable $Z_i$ in $\mathcal Z$. We assume that $\mathcal Z$ is a finite set and that the $(Z_i)_{i\in \Z}$ are i.i.d. 

There is a single incoming road and $K$ outgoing roads (where $K\in \N\backslash \{0\}$). The junction $\mathcal R$ is given by 
$$
\mathcal R= \bigcup_{k=0}^K \mathcal R^k, \qquad \mathcal R^0= (-\infty,0]\times \{0\}, \; \mathcal R^k= [0, +\infty)\times \{k\} \; \text{for $k\in \{1, \dots, K\}$}. 
$$
We also denote by $ \stackrel{o}{{\mathcal R}}$ the interior of the roads:
$$ \stackrel{o}{{\mathcal R}}= \bigcup_{k=0}^K  \stackrel{o}{ \mathcal R^k}, \qquad  \stackrel{o}{\mathcal R^0}= (-\infty,0)\times \{0\}, \;  \stackrel{o}{\mathcal R^k}= (0, +\infty)\times \{k\} \; \text{for $k\in \{1, \dots, K\}$}. 
$$
The outgoing road chosen by a vehicle is determined by its type $z$ and  is given by the map  $T: \mathcal Z\to \{1,\dots, K\}$. We  set   
$$
T_i^\omega= T(Z_i^\omega)\;{\rm and}\; \ell^\omega_i=\inf\{ j>i, \; T_j^\omega=T_i^\omega\} \qquad \forall i\in \Z.
$$
The vehicle $\ell_i$ is the first vehicle in front of $i$ which takes the same outgoing road as $i$. As the vehicles with the same outgoing road remain ordered, $\ell_i$ does not depend on time. 
For $k\in \{1, \dots, K\}$, let $\pi^k= \P[T_0=k]$ be the probability for a vehicle to take the outgoing road $k$. By convention, we set $\pi^0=1$. {\it Without loss of generality, we assume throughout the paper that $\pi^k\in (0,1]$ for any $k\in \{1, \dots, K\}$}. In this case $\ell_i$ is well-defined a.s. since, $\P-$a.s., $\{j>i, \; T_j^\omega=T_i^\omega\}$ is nonempty. \\

The bifurcation is supposed to be at $x=0$. We assume that the equation is homogeneous outside a transition zone $[-R_0,0]$ near the bifurcation: namely we suppose the existence of $R_0>0$ and of $\tilde V^0,\dots, \tilde V^K:[0,+\infty)\to [0,+\infty)$ such that 
$$
V_z(e_1,e_2, x) =\left\{ \begin{array}{ll}
\tilde V^0_z(e_1) & {\rm if }\; x\leq -R_0\\
\tilde V^{k}_z(e_2) & {\rm if}\; x\geq 0 \;{\rm and} \; T(z)= k. 
\end{array}\right.
$$
The meaning of this assumption is that, if the position $U_i(t)$ of a vehicle $i$ at time $t$ is not in the interval $(-R_0,0)$, the velocity of this vehicle is determined by its type and by the distance to the vehicle in front of it (which has label $i+1$ if $U_i(t)\leq -R_0$ and $\ell_i$ if $U_i(t)\geq 0$). It is only when the vehicle is in the transition zone $[-R_0,0]$ that its velocity also depends possibly on its position and on the vehicles in front;  for instance it may slow down to prepare the change of road.\\

The problem as stated above contains the following particular cases: 
\begin{itemize}
\item Problem on a single road with a perturbation: in this case there is a single outgoing road and  the vehicles solve the simpler system 
$$
\frac{d}{dt} U^\omega_i(t)= V_{Z^\omega_i}(U^\omega_{i+1}(t)-U^\omega_i(t), U^\omega_i(t))\qquad i\in \Z, \; t\geq 0,
$$
where $V_z(e_1,x)=\tilde V_z(e_1)$ does not depend on $x$ if $x\notin [-R_0,0]$.  

\item Problem in which the type is only the choice of the outgoing road: in this case  the system is still of the form \eqref{eq.SystJuncTOT} but one has $\mathcal Z= \{1, \dots, K\}$, $T(z)=z$ and the $\tilde V^k_z$ do not depend on $z$. There is still a transition zone $[-R_0,0]$ on which the velocity of the vehicle $i$ passes from a dependence to the distance to the vehicle right in front (with label $i+1$) to the distance to the vehicle going on the same outgoing road (with label $\ell_i$). 
\end{itemize} 

If the proof of homogenization would be somewhat simpler in the first case (as described in the introduction), the second case contains already (almost) all the difficulties we will meet below. \\

The goal of the paper is to understand the behavior of the solution on large scale of time and space: namely, the behavior of $(x,t)\to \ep U_{[x/\ep]}(t/\ep)$, (where $[y]$ is the integer part of the real number $y$). \\

{\bf Notation:} Throughout the paper, $\Omega:= \mathcal Z^{\Z}$ is endowed with the product $\sigma-$field $\mathcal F$ and with the product probability measure $\P$. We denote by $\tau:\Z\times \Omega\to \Omega$ the shift map defined by
$$
(\tau_n\omega)_i= \omega_{i+n}, \qquad \forall \omega =(\omega_i)_{i\in \Z}\in \Omega, \; \forall n\in \Z.
$$
We set $Z^\omega_i= \omega_i$ for $\omega=(\omega_i) \in \Omega$ and $i\in \Z$. As $\P$ is the product measure on $\Omega$, this means  that the $(Z_i)_i\in \Z$ are i.i.d. We note for later use that $Z^{\tau_n\omega}_{i-n}= Z^\omega_i$ while $\ell^{\tau_n\omega}_i = \ell^\omega_{i+n}-n$ for any $n, i\in \Z$. 

For $x,y\in \R$, we denote by $[x]$ the integer part of $x$, set $(x)_+= \max\{0, x\}$, $(x)_-=\max\{0, -x\}$, $x\wedge y=\min\{x,y\}$. If $E\in \mathcal F$, then $E^c= \Omega\backslash E$. 

\subsection{Assumptions} 

Let us state our standing assumptions on  $V_z$: 
\begin{itemize}
\item[$(H_1)$] For any $z\in {\mathcal Z}$, the map $(e_1,e_2,x)\to V_z(e_1,e_2,x)$ is Lipschitz continuous from $\R_+^2\times \R$ to $\R_+$ and nondecreasing with respect to the first two variables;
\item[$(H_2)$] There exists $e_{\max}>\Delta_{\min}>0$ and $0<R_2<R_1<R_0$, with $R_0>e_{\max}$, such that for any $z\in \mathcal Z$, for any $(e_1,e_2,x)\in \R_+^2\times \R$, 
\begin{itemize}
\item[(i)] $V_z(e_1,e_2,x)=0$ if ($e_1\leq \Delta_{\min}$ and $x\leq -R_2$) or if ($e_2\leq \Delta_{\min}$ and $x\geq -R_1$), 
\item[(ii)]  $V_z(e,e_2,x)= V_z(e_{\max},e_2,x)$ and $V_z(e_1,e,x)= V_z(e_1,e_{\max},x)$ if $e\geq e_{\max}$;
\end{itemize}
\item[$(H_3)$] There exists $\tilde V^0,\dots, \tilde V^K:[0,+\infty)\to [0,+\infty)$ such that 
$$
V_z(e_1,e_2, x) =\left\{ \begin{array}{ll}
\tilde V^0_z(e_1) & {\rm if }\; x\leq -R_0 \\
\tilde V^{k}_z(e_2) & {\rm if}\; x\geq 0 \;{\rm and} \; T(z)= k. 
\end{array}\right.
$$
\item[$(H_4)$] For any $z\in {\mathcal Z}$ and any $k\in \{0, \dots, K\}$,  there exists $h^k_{\max,z}\in (\Delta_{\min}, e_{\max}]$ such that  $p\to \tilde V^k_z(p)$ is increasing  and concave
in $[\Delta_{\min},h^k_{\max,z}]$ and constant on $[h^k_{\max,z},+\infty)$; 
\item[$(H_5)$] There exists $\kappa>0$ such that,  for any $z\in \mathcal Z$,
\begin{itemize}
\item[(i)] $V_z(e_1,e_2,x)= \tilde V^0_z(e_1)$ if $e_1\leq e_2$, $x\leq -R_2$ and $V_z(e_1,e_2,x)\leq \kappa$, 
\item[(ii)] $\partial_x V_z(e_1,e_2,x) \geq 0$ if $x\in [-R_1,0]$ and $V_z(e_1,e_2,x)\leq \kappa$, 
\item[(iii)] $V_z(e_1,e_2,x)  >0$ if $e_1\wedge e_2 >\Delta_{\min}$.
\end{itemize}
%
\end{itemize}
Note that, by assumption ($H_2$), we have $\tilde V^k_z(e)= 0$ if $e\leq \Delta_{\min}$ and $\tilde V^k_z(e)=\tilde V^k_z(e_{\max})$ if $e\geq e_{\max}$. \\
 
Some comment on the assumption are now in order. The assumption that $\mathcal Z$ is finite is useful throughout the proofs but could be relaxed; as this would introduce an extra layer of technicalities, we prefer to keep this condition for simplicity. Assumption (${ H_2}$) is standard in the analysis of leader-follower models. The existence of $\Delta_{\min}$ prevents  vehicles to collide (and could correspond to the size of the smallest vehicle for instance). The existence of $e_{\max}$ just says that the vehicles do not take into account the vehicles too far ahead. Assumption $R_0>e_{\max}$ can be made without loss of generality. Assumption ($H_3$) means that the roads are homogeneous outside the bifurcation. This formalizes the fact that we concentrate here on a single bifurcation. Assumption ($H_4$) is also standard in the analysis of  leader-follower models. There is one restriction though: the minimal distance such that the velocity has to be positive (i.e., here $\Delta_{\min}$) has to be the same for all vehicle and is not allowed to depend on the type of the vehicle; this restriction is related to the last (and technical) assumption ($H_5$). Assumption ($H_5$) has to do with the behavior of vehicles with slow velocity on the junction and ensures that  the vehicles starting with a flat initial condition $(U_i(0):=e^ki)_{i\in \Z}$ (where $k=0$ if $i\leq 0$ and $k=T_i$ if $i\geq 0$ and $e^k$ is such that  $H^k(-1/e^k)= \min_p H^k(p)$) have a velocity bounded below by a positive constant independent of time and position (Lemma \ref{lem.minvke}). This last property is instrumental throughout the proofs. Assumption (${H_5}$), without being unrealistic, is a little restrictive, but we do not know if it is possible to relax it. We illustrate these assumptions by an example. 
\bigskip
 
\noindent {\bf An example.} Let $0< r_3<r_2<r_1<r_0$.
Fix three smooth and nonincreasing maps $\xi_i:\R\to [0,1]$ such that $\xi_i(x)=1$ for $x\leq -r_{i-1}$, $\xi_i(x)= 0$ for $x\geq -r_i$ ($i=1, 2, 3$). Fix also $\bar s>0$ and, for any $z\in \mathcal Z$ and $k\in\{0, \dots, K\}$, $W^k_z:[0, +\infty)\to [0,+\infty)$  a smooth and nondecreasing map with $W_z(0)=0$, $W_z(s)=W_z(\bar s)$ for any $s\geq \bar s$; we also assume that $W^k$ is increasing on $[0,\bar s]$. Then the map defined by $e_1,e_2\geq 0$, $x\in \R$, $z\in \mathcal Z$ and $k=T(z)$ by
\begin{align*}
V_z(e_1,e_2,x)= & \xi_1(x) W^0_z\bigl((e_1-\Delta_{\min})_+\bigr)
+(1-\xi_1(x)) \xi_2(x)W^0_z\bigl( (e_1\wedge e_2-\Delta_{\min})_+\bigr) \\
& + (1-\xi_2(x)) 
W^k_z \bigl( \xi_3(x)(e_1\wedge e_2-\Delta_{\min})_++(1-\xi_3(x))(e_2-\Delta_{\min})_+\bigr)
\end{align*}
satisfies the required conditions with $R_0=r_0$, $R_1=r_2$, $R_2=r_3$ and $\tilde V^k_z(p)= W^k_z((p-\Delta_{\min})_+)$. The complicated expression of $V$ expresses the transition between a configuration in which the vehicle drives at speed $\tilde V^0_z$ and considers only the vehicle in front on road $0$ to a configuration in which it drives at speed $\tilde V^k_z$ and considers only the vehicle in front on road $k$. Namely, before $-r_0$, the vehicle, driving at speed $\tilde V^0_z$, takes into account only the next vehicle on road $0$. Between $-r_0$ and $-r_1$, the vehicle, still driving at speed $\tilde V^0_z$, slows down in order to take also into account the next vehicle on the road $k$. Between $-r_1$ and $-r_2$, the vehicle adapts its speed to road $k$ (passes from velocity $\tilde V^0_z$ to $\tilde V^k_z$). Between $-r_2$ and $-r_3$, the vehicle, driving at speed $\tilde V^k_z$, looses track of the vehicle which was in front on road $0$ and only considers the vehicle in front on road $k$ after $0$.

\subsection{The homogenized velocities and Hamiltonians.}

Let $V^k_{\max,z}:=  \tilde V^k_z(h^k_{\max,z})$. Under assumptions (H1)---(H4), the map $ \tilde V^k_z:  [\Delta_{\min},h^k_{\max,z}]\to [0, V^k_{\max,z}]$ is increasing and continuous for any $z\in {\mathcal Z}$  and any $k\in \{0, \dots, K\}$. We denote by $( \tilde V^k_z)^{-1}$ its inverse. \\

 Let
\be\label{def.barvk}
\bar v^0:= \inf_{z\in \mathcal Z} \tilde V^0_z(e_{\max}), \qquad \bar v^k:= \inf_{z\in \mathcal Z, \; T(z)= k} \tilde V^k_z(e_{\max}).
\ee
We recall from \cite{CaFo} the definition of the homogenized velocities $\bar V^k$ and homogenized Hamiltonians:  $\bar  V^0$ is the inverse of the continuous increasing map defined on $(0,\bar v^0)$ by $v\to  \E\left[ (\tilde V^0_{Z_0})^{-1}(v)\right]$. We note that  $\bar  V^0$ is defined on $(\Delta_{\min}, \E\left[ (\tilde V^0_{Z_0})^{-1}(\bar v^0)\right])$.  We extend  it for any $e\in [0,\Delta_{\min}]$ by $\bar  V^0(e)=0$ and for $e\geq \E\left[ (\tilde V^0_{Z_0})^{-1}(\bar v^0)\right]$ by $\bar  V^0(e)=\bar v^0$. In the same way we define $\bar  V^k$ as the inverse of the continuous increasing map defined on $(0,\bar v^k)$ by $v\to  \E\left[ (\tilde V^k_{Z_0})^{-1}(v)\ |\ T_0=k\right]$. It defines  $\bar  V^k$ on $(\Delta_{\min}, \E\left[ (\tilde V^k_{Z_0})^{-1}(\bar v^0)\; |\; T_0=k\right])$.  We extend it  for any $e\in [0,\Delta_{\min}]$ by $\bar  V^k(e)=0$ and for any $e\geq  \E\left[ (\tilde V^k_{Z_0})^{-1}(\bar v^0)\; |\; T_0=k\right]$ by $\bar  V^k(e)=\bar v^k$. The maps $\bar  V^k$ (for $k\in \{0, \dots, K\}$) are continuous and bounded on $[0,+\infty)$. 

We set, for any $k\in \{1, \dots K\}$,  
$$
H^0(p)= p \bar V^0(-1/p), \; H^k(p)= p \bar V^k(-1/(\pi^k p)), \; p\in (-\infty,0), \qquad H^0(p)= H^k(p)=0, \; \forall p\geq0
$$
and
\be\label{defA0}
A_0=\max_{k\in\{0, \dots, K\}}   \min_{p\in \R}  H^k(p).  
\ee
By Assumption $(H4)$, for $i\in \{0, \dots, K\}$,  $H^k$ is convex in $(-1/(\pi^k\Delta_{\min}),0)$ (see Lemma \ref{lem:convexity}).\\

A last set of notation will be needed in order to define the condition at the junction: for $k\in \{0, \dots, K\}$, we denote by 
$H^{k,+}$ (resp. $H^{k,-}$) the largest nondecreasing (resp. nonincreasing) map below $H^k$.  

\subsection{The main result} 

The main result of the paper states that the system  homogenizes: let $(U^{0,\ep}_i)_{i\in \Z}$  be a deterministic family of initial conditions satisfying the compatibility condition: for any $i\in \Z$, 
\be\label{compatibilitycond}
U^{0,\ep}_{i+1}\geq U^{0,\ep}_i+\Delta_{\min} \; {\rm if} \; U^{0,\ep}_{i+1}\leq -R_2 \qquad {\rm and}\qquad U^{0,\ep}_{\ell_i}\geq U^{0,\ep}_i +\Delta_{\min}\; \mbox{\rm for any $i\in \Z$}. 
\ee 
Up to relabel the indices, we also assume that $U^\ep_{i,0}\le0$  iff $i\le0$. 
Let $U^\ep$ be the  solution of \eqref{eq.SystJuncTOT} with initial condition $(U^{0,\ep}_i)_{i\in \Z}$. Let us define, for $k\in \{1, \dots, K\}$ and $(x,t)\in \R\times [0,+\infty)$,  
\be\label{def.NJuncTOTbis}
N^{\ep,\omega}(x,k,t)=  \sum_{i\in \Z, \ i\leq 0, \  T_i=k}  \delta_{U^\ep_{i}(t)}( (x,+\infty))- \sum_{i\in \Z,\ i>0, \ T_i=k} \delta_{U^\ep_{i}(t)}((-\infty,x]). 
\ee
and set for $x\leq0$
$$
N^{\ep,\omega}(x,0,t)=  \sum_{i\in \Z,\ i\le 0}  \delta_{U^\ep_{i}(t)}( (x,+\infty)).
$$
Then we introduce the scaled quantities 
\be\label{defnuep}
\nu^{\ep,\omega}(x,k,t)= \left\{ \begin{array}{ll} \ep (\pi^k)^{-1} N^{\ep, \omega}(x/\ep,k,t/\ep)& \forall (x,k,t)\in \R\times \{1, \dots, K\}\times [0,+\infty)\\
\ep N^{\ep,\omega}(x/\ep,0,t/\ep)& \forall (x,t)\in  (-\infty,0]\times [0,+\infty)\end{array}\right. 
\ee

\begin{Theorem}\label{thm.main} There is a set $\Omega_0$ of full probability and a constant $\bar A<0$ (the flux limiter) such that, if $(U^{0,\ep}_i)_{i\in \Z}$ is a family of initial conditions such that the associated scaled function $\nu^\ep(\cdot, \cdot, 0)$ defined by \eqref{defnuep} (with $t=0$) converges locally uniformly in $\mathcal R$ to a Lipschitz continuous map  $\nu_0:\mathcal R\to \R$, then, for any $\omega\in \Omega_0$, $\nu^\ep$ converges locally uniformly in $\mathcal R\times [0,+\infty)$ to the unique continuous viscosity solution of the Hamilton-Jacobi equation with flux limiter $\bar A$: 
\be\label{eq.limitsystem}
\left\{\begin{array}{l}
\ds \partial_t \nu + H(\partial_x \nu)=0\qquad {\rm in }\;  \stackrel{o}{{\mathcal R}}\times (0,+\infty) \\
\ds \partial_t \nu +\max\{ \bar A,  H^{0,+}(\partial_0 \nu), H^{1,-}(\partial_1 \nu), \dots, H^{K,-}(\partial_K \nu))\}=0\; {\rm at}\; x=0\\ 
\ds \nu(x,k,0)= \nu_0(x,k) \qquad {\rm in}\; \mathcal R.
\end{array}\right.
\ee
\end{Theorem}

Let us recall the notion of viscosity solution of \eqref{eq.limitsystem}. For this we define the set of test functions $C^1(\mathcal R\times (0,+\infty))$ as the set of continuous maps $\phi:\mathcal R\times (0,+\infty)\to \R$ such that the restriction to each branch of $\mathcal R$ is of class $C^1$ on this branch and $\partial_t \phi$ exists and is continuous everywhere. We denote by $\partial_k\phi(0,t)$ its derivative at $x=0$ on the branch $k$ (namely, $\partial_k\phi(0,t)= \partial_x\phi(0,k,t)$, which is well-defined by continuity). 

We say that a  map $\nu$ is a viscosity solution of \eqref{eq.limitsystem} if $\nu:\mathcal R\times [0,+\infty)\to \R$ is  uniformly continuous, and if, for any test function $\phi\in C^1(\mathcal R\times (0,+\infty))$ such that $\nu-\phi$ has a local maximum (respectively minimum) at $(\bar x,\bar k, \bar t)\in \mathcal R\times (0,+\infty)$  one has \begin{align*}
& \partial_t \phi(\bar x,\bar k, \bar t) + H^{\bar k}(\partial_x \phi(\bar x, \bar k, \bar t))\leq 0 & {\rm if }\; \bar x\neq 0 & \qquad \mbox{\rm (resp. $\geq 0$)}\\ 
& \partial_t \phi(0, \bar t) +\max\{ \bar A,  H^{0,+}(\partial_0 \phi(0,\bar t)), H^{1,-}(\partial_1  \phi(0,\bar t)), \dots, H^{K,-}(\partial_K  \phi(0,\bar t)))\}\leq 0& {\rm if }\; \bar x=0  &\qquad  \mbox{\rm (resp. $\geq 0$)}
\end{align*}

\section{Properties of the solution} 

In this section we investigate two important properties of the solution: the distance to the next vehicle and the finite speed of propagation. \\

Throughout the paper we need to define the solution $(U_i)$ of \eqref{eq.SystJuncTOT} for a finite number of indices, namely for $i\in \{i_1, \dots, i_0\}$ where $i_0,i_1\in \Z$, $i_1<i_0$.  We say that $(U_i)_{i\in \{i_1, \dots, i_0\}}$ is a subsolution (respectively a supersolution) of \eqref{eq.SystJuncTOT} on a time interval $[0,T]$ if, for any $i\in \{i_1, \dots, i_0\}$, the map $t\to U_i(t)$ is nondecreasing, Lipschitz continuous with a Lipschitz constant not larger than $\|V\|_\infty$ and if, for any $i\in \{i_1, \dots, i_0\}$ with $\ell_i\leq i_0$,
$$
\frac{d}{dt} U^\omega_i(t)\leq V_{Z^\omega_i}(U^\omega_{i+1}(t)-U^\omega_i(t), U^\omega_{\ell^\omega_i}(t)-U^\omega_i(t), U^\omega_i(t)) \qquad \forall t\in[0,T]
$$
(resp. $$
\frac{d}{dt} U^\omega_i(t)\geq V_{Z^\omega_i}(U^\omega_{i+1}(t)-U^\omega_i(t), U^\omega_{\ell^\omega_i}(t)-U^\omega_i(t), U^\omega_i(t)) \qquad \forall t\in[0,T].)
$$

\subsection{Basic properties}

Given an initial condition $(U^0_i)_{i\in \Z}$ satisfying the compatibility condition \eqref{compatibilitycond}, there exists a unique solution $U=(U_i)_{i\in \Z}$ to \eqref{eq.SystJuncTOT}. Moreover we have the following basic comparison  principle: if  $(U_i)_{i\in \Z}$ and $(\tilde U_i)_{i\in \Z}$ are two solutions of \eqref{eq.SystJuncTOT} such that $U_i(0)\leq \tilde U_i(0)$ for any $i\in \Z$, then $U_i(t)\leq \tilde U_i(t)$ for any $i\in \Z$ and any $t\geq 0$. These results are standard and are easy consequences of Lemma \ref{lem.finitespeedPSJunc} below. 

\begin{Lemma}[Basic ordering] \label{lem.CompDeBase} Let $i_0, i_1\in \Z$ with $i_1<i_0$,  $T>0$ and $(U_i)_{i\in \{i_1, \dots,  i_0\}}$ be a solution of \eqref{eq.SystJuncTOT} on the time interval $[0,T]$ with $(U_i(0))$ satisfying the compatibility condition \eqref{compatibilitycond}. We have $U_{i}(t)\leq U_{j}(t)-\Delta_{\min}$ for any $i_1\leq i< j\leq i_0$  such that $\ell_i\leq i_0$ and $\ell_j\leq i_0$ and any $t\in [0,T]$ with $U_{j}(t)\leq -R_2$ or $U_i(t)\leq -R_2$. In addition, for any $i\in \{i_1,\dots, i_0\}$ with $\ell_i\leq i_0$ and any $t\in [0,T]$ , $U_{i}(t)\leq U_{\ell_i}(t)-\Delta_{\min}$. \\
\end{Lemma}

Recall that $R_2$ and $\Delta_{\min}$ are defined in Assumption ($H_2$). Note that, after the junction, the order is not necessarily preserved among the vehicles if they have not the same type. 

\begin{Remark}\label{rem.lem.distminnext}{\rm 
We note for later use that, if $U_j(t)\leq -R_2$ and $i<j$, then $U_{i}(t)\leq U_j(t)-(j-i)\Delta_{\min}$. 
}
\end{Remark}

\begin{proof} To prove the first claim, it is enough to check that $U_{i-1}(t)\leq U_i(t)-\Delta_{\min}$ for $t\in [0,T]$ if $U_{i-1}(t)\leq -R_2$ or $U_i(t)\leq-R_2$. We start with the first case. Assume by contradiction that there exists $i\in \{i_1+1,\dots, i_0\}$ and a time $s\in[0,T]$ such that $\delta:=\Delta_{\min}-(U_{i}(s)-U_{i-1}(s))>0$ and $U_{i-1}(s)\leq -R_2$. Let $\tau>0$ be the largest time such that $U_{i}(t)-U_{i-1}(t)\geq \Delta_{\min}-\delta$ on $[0,\tau)$. Note that $\tau\leq s$, $U_{i}(\tau)-U_{i-1}(\tau )= \Delta_{\min}-\delta$ and  $U_{i-1}(\tau)\leq -R_2$. Then, for $\ep>0$ small enough, the map $t\to U_{i}(t)-U_{i-1}(t)+\ep/(\tau-t)$ has a minimum on $[0, \tau)$, which is less than $\Delta_{\min}$ and reached at a time $\bar t\in(0,\tau)$. By optimality condition we have 
\begin{align*}
& V_{Z_{i}}(U_{i+1}(\bar t)-U_{i}(\bar t),U_{\ell_{i}}(\bar t)-U_{i}(\bar t), U_{i}(\bar t))\\
& \qquad - V_{Z_{i-1}}(U_{i}(\bar t)-U_{i-1}(\bar t), U_{\ell_{i-1}}(\bar t)-U_{i-1}(\bar t), U_{i-1}(\bar t)) + \ep/(\tau-\bar t)^2=0, 
\end{align*}
where by Assumption ($H_2$) the second term vanishes because $U_{i-1}(\bar t)\leq -R_2$ and $U_{i}(\bar t)-U_{i-1}(\bar t)\leq \Delta_{\min}$.  So there is a contradiction and  we have proved that $U_{i-1}(t)\leq U_i(t)$ if $U_{i-1}(t)\leq -R_2$. \\

Let us now check that  $U_{i-1}(t)+\Delta_{\min}\leq U_i(t)$ if  $U_i(t)\leq-R_2$. Let $\tau$ be the first time (if any) such that $U_{i-1}(\tau)= -R_2$. We have just proved that $U_{i-1}(s)+\Delta_{\min}\leq U_i(s)$ if $s\in[0,\tau]$. Thus $U_i(\tau)>-R_2$, which implies that $t<\tau$ and proves the claim. \\ 

The proof of the second statement is analogous: if there is $i\in \{i_1,\dots, i_0\}$ with $\ell_i\leq i_0$, a time $s\in [0,T]$ and $\delta>0$ such that  $U_{i}(s)> U_{\ell_i}(s)-\Delta_{\min}+\delta$, then we look at the largest time $\tau$ such that $U_{i}(t)< U_{\ell_i}(t)-\Delta_{\min}+\delta$ on $[0,\tau]$. We have $U_{i}(\tau)= U_{\ell_i}(\tau)-\Delta_{\min}+\delta$ and, by the previous step, we know that $U_{i}(\tau)>-R_2$. Let $\bar t\in (0,\tau)$ be a minimum point of the map $t\to U_{\ell_i}(t)-U_i(t)+\ep/(\tau-t)$ on $[0,\tau)$. For $\ep>0$ small enough, this minimum is less than $\Delta_{\min}$. Then by optimality we have 
\begin{align*}
& V_{Z_{\ell_i}}(U_{\ell_i+1}(\bar t)-U_{\ell_i}(\bar t),U_{\ell_{\ell_i}}(\bar t)-U_{\ell_i}(\bar t), U_{\ell_i}(\bar t))\\
& \qquad - V_{Z_i}(U_{i+1}(\bar t)-U_{i}(\bar t), U_{\ell_{i}}(\bar t)-U_{i}(\bar t), U_{i}(\bar t)) + \ep/(\tau-\bar t)^2=0, 
\end{align*}
where the second term vanishes by assumption (H2) because $U_{i}(\bar t)\geq -R_2>-R_1$ and $U_{\ell_i}(\bar t)-U_{i}(\bar t)\leq \Delta_{\min}$.  So there is a contradiction and  we have proved that $U_{i}(t)\leq U_{\ell_i}(t)-\Delta_{\min}$.
\end{proof}

\subsection{The maximal distance to the next vehicle}\label{subsec.distVehicle}

In this part we investigate the maximal distance to the next vehicle. This question will play an important role in the proof of Theorem \ref{thm:concentrationTOT}. Here we work in a deterministic setting: we fix a deterministic sequence $(z_i)_{i\in \Z}$ in $\mathcal Z$ such that $\ell_i:=\inf\{j\geq i+1, \; T(z_j)=T(z_i)\}$ is finite for any $i\in \Z$. We consider the system of ODEs
\be\label{eq.SystJuncTOTmaxdist}
\frac{d}{dt} U_i(t)= V_{z_i}(U_{i+1}(t)-U_i(t), U_{\ell_i}(t)-U_i(t), U_i(t))\qquad i\in \Z, \; t\geq 0, 
\ee
Let us introduce some notation for the slow vehicles. Recall that 
$$
\bar v^0:= \inf_{z\in \mathcal Z} \tilde V^0_z(e_{\max}), \qquad \bar v^k:= \inf_{z\in \mathcal Z, \; T(z)= k} \tilde  V^k_z(e_{\max})
$$
and let $z^k_{\min}$ (where $k\in \{0, \dots,K\}$) be  elements of $\mathcal Z$ such that $\tilde V^0_{z^0_{\min}}(e_{\max})=\bar v^0$ and, for $k\geq 1$, $T(z^k_{\min})=k$ and $\tilde  V^k_{z^k_{\min}}(e_{\max})=\bar v^k$.  The types  $z^k_{\min}$ correspond to ``slow vehicles'', in the sense that their maximal velocity is the smallest. 

\begin{Lemma} \label{lem.EstiDansR0} Let $\bar t>0$, $i_1\in \Z$. Assume that $T(z_{i_1})= z^0_{\min}$ and that $(U_i)_{i\in\Z}$ is a solution to \eqref{eq.SystJuncTOTmaxdist} on $[0,\bar t]$ with an initial condition satisfying the compatibility condition \eqref{compatibilitycond}. If $U_{i_1}(\bar t)\leq -R_0$, then for any $i\leq i_1$ we have 
$$
U_{ i_1}(\bar t) \leq U_{i}(\bar t) +U_{i_1}(0)-U_{i}(0) +e_{\max}(i_1-i). 
$$
\end{Lemma}

\begin{proof}  We prove by induction on $m\in \N$ that $U_{ i_1}(\bar t) \leq U_{i_1-m}(\bar t) +U_{ i_1}(0)-U_{i_1-m}(0)+e_{\max}m$. The result is obvious for $m=0$. Let us assume that it holds for some $m-1$ with $m\geq 1$. We set $i:=i_1-m$. We argue by contradiction and assume that there exists $t\in [0,\bar t]$ such that $U_{i_1}(t) -U_{i}(t) >U_{i_1}(0) -U_{i}(0)+e_{\max}m$. Then for $\ep>0$ small, the maximum of $t\to U_{i_1}(t) -U_{i}(t) -\ep/(\bar t-t)$ exists on $[0,\bar t)$ and is larger than $U_{i_1}(0) -U_{i}(0)+e_{\max}m$. We denote by $t_\ep$ the point of maximum and we remark that $t_\ep$ is positive by definition. By optimality condition, we have 
\be\label{aejkzsdf}
\tilde V_{z_{i_1}}^0(U_{i_1+1}(t_\ep)-U_{i_1}(t_\ep)) -\tilde V_{z_i}^0(U_{i+1}(t_\ep)-U_i(t_\ep))-\ep/(\bar t-t_\ep)^2=0.
\ee
On the other hand, by induction assumption, we have $U_{i_1}(t_\ep)-U_{i+1}(t_\ep)\leq U_{i_1}(0)-U_{i+1}(0) + e_{\max}(m-1)$. So 
 \begin{align*}
 U_{i+1}(t_\ep)-U_i(t_\ep) & = U_{i_1}(t_\ep)-U_i(t_\ep)-(U_{i_1}(t_\ep)-U_{i+1}(t_\ep)) \\
 & >  U_{i_1}(0) -U_{i}(0)+e_{\max}m -(U_{i_1}(0)-U_{i+1}(0) + e_{\max}(m-1))\geq e_{\max},
 \end{align*}
 so that, by assumption (H2)
 $$
\tilde V_{z_i}^0(U_{i+1}(t_\ep)-U_i(t_\ep)) = \tilde V_{z_i}^0(e_{\max})
$$
while, as $z_{i_1}=z^0_{\min}$, 
 $$
 \tilde V_{z_{i_1}}^0(U_{i_1+1}(t_\ep)-U_{i_1}(t_\ep)) \leq \tilde V_{z_{i_1}}^0(e_{\max})= \bar v^0 \leq \tilde V_{z_i}^0(e_{\max}). 
 $$
This contradicts \eqref{aejkzsdf} and proves the result. 
\end{proof}

\begin{Lemma}\label{lem.EstiDansRk} Let $(U_i)_{i\in \Z}$ be a solution of \eqref{eq.SystJuncTOTmaxdist} with an initial condition satisfying the compatibility condition \eqref{compatibilitycond}. Let $i,i_0\in \Z$ and $k\in \{1, \dots, K\}$ with $i_0> i$, $T(z_i)=T(z_{i_0})=k$ and $z_{i_0}=z^k_{\min}$. If $U_i(s)\geq 0$ for some $s\geq 0$, then 
$$
U_{i_0}(t) \leq U_i(t)+  U_{i_0}(s) - U_i(s) + e_{\max}\sharp\{ j\in \{i+1, \dots , i_0\}, \; T(z_j)=T(z_i)\}\qquad \forall t\geq s.
$$
\end{Lemma}

\begin{proof} We proceed  by induction on the value of $n:= \sharp\{ j\in \{i+1, \dots , i_0\}, \; T(z_j)=T(z_i)\}$. Let us fist assume that $n=1$. Then $i_0=\ell_i$. We argue by contradiction and assume that there exists $t\ge  s$ such that $U_{\ell_i}(t) -U_{i}(t) >e_{\max}+ U_{i_0}(s) - U_i(s)$. Then for $\ep>0$ small, the maximum $t_\ep$ of $t\to U_{\ell_i}(t) -U_{i}(t) -\ep t^2$ exists on $[s,+\infty)$ and is larger than $e_{\max}+ U_{\ell_i}(s) - U_i(s)$. Hence $t_\ep$ is larger than $s$. By optimality condition, we have (setting $k=T(z_i)$)
\be\label{aejkzsdfTOT}
\tilde V_{z_{\ell_i}}^k(U_{\ell_{\ell_i}}(t_\ep)-U_{\ell_i}(t_\ep))-\tilde V_{z_i}^k(U_{\ell_i}(t_\ep)-U_i(t_\ep))-2\ep t_\ep=0.
\ee
As  $U_{\ell_i}(t_\ep)-U_i(t_\ep)  \geq e_{\max}$, we get, by the definition of $i_0=\ell_i$: 
$$
\tilde V_{z_i}^k(U_{\ell_i}(t_\ep)-U_i(t_\ep))=\tilde  V_{z_i}^k(e_{\max}) =  -2\ep t_\ep +\tilde V_{z^k_{\min}}^k(U_{\ell_{\ell_i}}(t_\ep)-U_{\ell_i}(t_\ep))  <  \bar v^k, 
$$
which contradicts the definition of $\bar v^k$. So the result holds for $n=1$. 

Let us now assume that the result holds for $n-1$ (where $n\geq 2$) and let us prove it for $n$. We argue by contradiction in the same way and suppose that 
there exists $t\ge  s$ such that $U_{i_0}(t) -U_{i}(t) > U_{i_0}(s) - U_i(s)+n e_{\max}$. 
As above,  for $\ep>0$ small, the maximum $t_\ep$ of $t\to U_{i_0}(t) -U_{i}(t) -\ep t^2$ exists on $[s,+\infty)$ and is larger than $U_{i_0}(s) - U_i(s)+n e_{\max}$. By the induction assumption we have
$$
U_{i_0}(\tau) - U_{\ell_i}(\tau)\leq U_{i_0}(s) - U_{\ell_i}(s) + (n-1)e_{\max} \qquad \forall \tau\geq s.
$$
Hence 
\begin{align*}
& U_{\ell_i}(t_\ep)-U_i(t_\ep) = U_{i_0}(t_\ep)-U_i(t_\ep)-(U_{i_0}(t_\ep) - U_{\ell_i}(t_\ep)) \\
& \qquad \geq U_{i_0}(s) - U_i(s)+n e_{\max}-(U_{i_0}(s) - U_{\ell_i}(s) + (n-1)e_{\max}) \ge e_{\max}.
\end{align*}
Moreover, by optimality condition, we have 
$$
\tilde V_{z_{i_0}}^k(U_{\ell_{i_0}}(t_\ep)-U_{i_0}(t_\ep))-\tilde V_{z_i}^k(U_{\ell_i}(t_\ep)-U_i(t_\ep))-2\ep t_\ep=0,
$$
which leads to a contradiction as above. 
\end{proof}

\begin{Lemma} \label{key.estiUlli-Ui} Let $(U_i)_{i\in \Z}$ be a solution of \eqref{eq.SystJuncTOTmaxdist} with an initial condition satisfying the compatibility condition \eqref{compatibilitycond}. Assume in addition that there exists $\delta>0$ such that $(d/dt) U_i(t)\geq \delta$ for all $t\geq0$ and all $i\in \Z$. Then, there exists a constant $C_0$ depending on $V$ and $\delta$ only such that, for any $i,i_0\in \Z$ and $k\in \{1, \dots, K\}$ with $i_0> i$, $T(z_i)=T(z_{i_0})=k$ and $z_{i_0}=z^k_{\min}$, we have
$$
U_{i_0}(t)\leq U_i(t)+ C_0(U_i(0))_-+ U_{i_0}(0)-U_i(0)+C_0\sharp \{j\in \{i+1, \dots, i_0\}, \; T(z_j)=k\}\qquad \forall t\geq 0.
$$
If in addition there exists $i_1\geq i_0$ such that $T(z_{i_1})=z^0_{\min}$ and $U_{i_1}(0)\leq -R_0$, then 
$$
U_{i_0}(t)\leq U_{i}(t) + C_0(1+U_{i_1}(0)-U_i(0)+ i_1-i) \qquad \forall t\geq0.
$$
\end{Lemma}

\begin{proof} 
If $U_i(0)\geq 0$, the first result holds by Lemma \ref{lem.EstiDansRk}. Let us now assume that $U_i(0)<0$. We set $n:= \sharp \{j\in \{i+1, \dots, i_0\}, \; T(z_j)=k\}$. Let $t_i=\inf\{t\geq 0, \; U_i(t)=0\}$. Then, as $(d/dt) U_{i}(t)\geq \delta$, we have $t_i\leq \delta^{-1} (-U_i(0))$. By Lemma \ref{lem.EstiDansRk}, we have 
 $$
U_{i_0}(t)-U_i(t) \leq U_{i_0}(t_i)-U_i(t_i)+e_{\max}n \qquad \forall t \geq  t_i.
$$
On the other hand, for $t\in [0,t_i]$, we get
$$
U_{i_0}(t)-U_i(t) \leq U_{i_0}(0)-U_i(0)+\|V\|_\infty t_i \leq U_{i_0}(0)-U_i(0)+C(U_i(0))_- . 
$$
This proves the first part of the claim. \\

Assume now that $U_{i_1}(0)\leq -R_0$. Let $\bar t=\inf \{t\geq 0, \; U_{i_1}(t)=-R_0\}$. We know from Lemma \ref{lem.EstiDansR0} that 
$$
U_{i_1}(t) \leq U_{i}(t)+U_{i_1}(0)-U_i(0)+e_{\max}(i_1-i)\qquad \forall t\in [0, \bar t].
$$
This implies by Lemma \ref{lem.CompDeBase} that 
\be\label{A1kjenf:cg}
U_{i_0}(t) \leq U_{i}(t)+U_{i_1}(0)-U_i(0)+e_{\max}(i_1-i)\qquad \forall t\in [0, \bar t].
\ee
Let $t_i=\inf\{t\geq 0, \; U_i(t)=0\}$. Then, as 
$$
U_i(\bar t) \geq U_{i_1}(\bar t)-(U_{i_1}(0)-U_i(0)+e_{\max}(i_1-i))\geq -R_0-(U_{i_1}(0)-U_i(0)+e_{\max}(i_1-i)),
$$
while $U_i(t_i)=0$, we have, since $(d/dt) U_{i}(t)\geq \delta$, 
\begin{align*}
t_i-\bar t & \leq  \delta^{-1}(U_i(t_i)- U_i(\bar t))\leq \delta^{-1}(R_0+U_{i_1}(0)-U_i(0)+e_{\max}(i_1-i))\\
& \leq C(1+U_{i_1}(0)-U_i(0)+i_1-i). 
\end{align*}
Note that $U_{i_0}(\bar t) \leq U_{i_1}(\bar t)= -R_0\leq 0$. So, as $U_i(t_i)=0$,  
$$
U_{i_0}(t_i)-U_i(t_i) \leq U_{i_0}(t_i)-U_{i_0}(\bar t) \leq \|V\|_\infty(t_i-\bar t)\leq C(1+U_{i_1}(0)-U_i(0)+i_1-i).
$$
By Lemma \ref{lem.EstiDansRk}, we obtain 
\be\label{A2kjenf:cg}
U_{i_0}(t)-U_i(t) \leq U_{i_0}(t_i)-U_i(t_i)+e_{\max}(i_0-i)\leq C(1+U_{i_1}(0)-U_i(0)+ i_1-i) \qquad \forall t \geq  t_i. 
\ee
Finally, for $t\in [\bar t, t_i]$, we have 
\be\label{A3kjenf:cg}
U_{i_0}(t)-U_i(t) \leq U_{i_0}(t_i)-U_i(t_i)+ \|V\|_\infty(t_i-t) \leq C(1+U_{i_1}(0)-U_i(0)+ i_1-i).
\ee
Combining \eqref{A1kjenf:cg}, \eqref{A2kjenf:cg} and \eqref{A3kjenf:cg} proves the second part of the claim. 
\end{proof}

\subsection{Approximate speed of propagation} 

The approximate speed of propagation says that the behavior of a vehicle mostly depends on a finite number of vehicles in front of it. To describe this result we need to introduce a few notation. Given $T\in \Z$, we define by induction 
$$
J^\omega_0(T)=T, \qquad J^\omega_n(T) = \inf_{k\in \{1, \dots, K\}}\sup\Bigl\{ i\in \Z, \; T^\omega_i= k, \; \ell^\omega_i\leq J^\omega_{n-1}(T)\Bigr\}.
$$
We note that the $J_n(T)$ are random and decreasing in $n$. By construction, if $i\leq J_n(T)$, then $i+1\leq J_{n-1}(T)$ and $\ell_i\leq J_{n-1}(T)$.

\begin{Lemma} [Approximate finite speed of propagation on the junction] \label{lem.finitespeedPSJunc} Fix $i_0\in \Z$, $L\in \N$, $T\geq 0$ and $E\in \mathcal F$ an event with a positive probability. Assume that, in $E$, $(U_i)_{i\in \{i_0, \dots, i_0+L\}}$ is a (non decreasing) subsolution while $(\tilde U_i)_{i\in \{i_0, \dots, i_0+L\}}$ is a supersolution of the system
$$
\frac{d}{dt} U_i(t)=  V_{Z_i}(U_{i+1}(t)-U_i(t), U_{\ell_i}(t)-U_i(t), U_i(t))
$$
for $i=i_0, \dots ,J_1(i_0+L)$  and $t\in [0,T]$. Suppose in addition that, in $E$,  $U_i(0)\leq \tilde U_{i}(0)$ for $i\in \{i_0, \dots, i_0+L\}$. Then, for all $n\in \Z$, $n\geq 1$, for all $\omega\in E$  and for all $i\in \{i_0, \dots,  J_n(i_0+L)\}$, 
 $$
 U_{i}(t)\leq \tilde U_{i}(t) +C\ 2^{-n}e^{\beta t}  \qquad \forall t\in  [0,T], 
$$
where, $\beta=\gamma+2C_1$, with $\gamma:=\sup_{z\in \mathcal Z}( \|\partial_1  V_z\|_\infty+\|\partial_2  V_z\|_\infty)$ and $C_1:=\sup_{z\in \mathcal Z} \|\partial_x  V_z\|_\infty$, and where $C$ depend on $\beta$ and on $  \|V\|_{\infty}$ only.
\end{Lemma}

\begin{proof} We work in $E$ all along the proof.  Let us set $W_i(t)= \tilde U_i(t)-U_i(t)$ for $i=i_0,\dots,J_1( i_0+L)$. Since we work with sub and super-solution, we extend the velocity $V_z(p,x)$ by $0$ if $p\le 0$. Let $n_0$ be the largest integer such that $J_{n_0}(i_0+L)\geq i_0$. 
We define, for $n\in \{1,\dots, n_0\}$, 
$$
M_n= \sup_{i\in\{i_0, \dots, J_n(i_0+L)\}}  \sup_{s\in [0,T]}e^{-\beta s} [W_i(s)]_-.
$$
Note that, as $W_i(0)\geq 0$ and $(d/dt) W_i(s)\leq \|V\|_\infty$, we have 
$$
M_1\le \sup_{i\in\{i_0, \dots, i_0+L\}}  \sup_{s\in [0,T]}e^{-\beta s} \|V\|_\infty s \leq C, 
$$
where $C$ depends  on $\beta$ and $\|V\|_\infty$ only. The main step consists in showing that, for all $n\in\{2, \dots, n_0\}$,  
\be\label{MnMn-1}
M_{n}\leq \frac12 M_{n-1}. 
\ee
Fix $n\in\{2, \dots, n_0\}$ and $i\in \{i_0, \dots, J_n(i_0+L)\}$. We have, for $t\in [0,T]$, 
$$
\frac{d}{dt} W_i(t)  \geq  V_{Z_i}(\tilde U_{i+1}(t)- \tilde U_i(t), \tilde U_{\ell_i}(t)- \tilde U_i(t), \tilde U_i(t))- V_{Z_i}(U_{i+1}(t)-U_i(t), U_{\ell_i}(t)-U_i(t), U_i(t)). 
$$
So
\begin{align*}
\frac{d}{dt} W_i(t) &\geq A_{i,1}(t)(W_{i+1}(t)-W_i(t)) +A_{i,2}(t)(W_{\ell_i}(t)-W_i(t)) +B_i(t) W_i(t)
\end{align*}
where 
\begin{align*}
& A_{i,1}(t):= \int_0^1 \partial_1 V_{Z_i}(w_{i}(\tau)) d\tau,
\; A_{i,2}(t):= \int_0^1 \partial_2 V_{Z_i}(w_{i}(\tau)) d\tau,\; 
B_i(t):= \int_0^1 \partial_3 V_{Z_i}(w_i(t)) d\tau,
\end{align*}
with $w_i(t)= (w_{i,1}(\tau), w_{2,i}(t), w_{3,i}(t))$, 
\begin{align*}
&\qquad \qquad w_{i,1}(\tau)= (1-\tau)(\tilde U_{i+1}(t)-\tilde U_i(t))+\tau (U_{i+1}(t)-U_i(t)), 
\\
&w_{2,i}(t) = (1-\tau)(\tilde U_{\ell_i}(t)-\tilde U_i(t))+\tau (U_{\ell_i}(t)-U_i(t))
\; {\rm and }\;  w_{3,i}(t)= (1-\tau) \tilde U_i(t)+\tau U_i(t).
\end{align*}
We note for later use that $0\leq A_{i,1}\leq \gamma$, $0\leq A_{i,2}\leq \gamma$ and $|B_i|\leq C_1$. Setting $A_i=A_{i,1}+A_{i,2}$, we find
\begin{align*}
W_i(t)\geq & W_i(0) \exp\left\{-\int_0^t(A_i-B_i)(s)ds\right\} + \int_0^t A_{i,1}(s) \exp\left\{-\int_s^t (A_i-B_i)(\tau)d\tau\right\} W_{i+1}(s)ds\\
& + \int_0^t A_{i,2}(s) \exp\left\{-\int_s^t (A_i-B_i)(\tau)d\tau\right\} W_{\ell_i}(s)ds.
\end{align*}
As $W_i(0)\geq 0$ and $A_i\geq 0$, we infer that 
\begin{align*}
[W_i(t)]_- & \leq \int_0^t A_{i,1}(s) \exp\left\{-\int_s^t (A_i-B_i)(\tau)d\tau\right\} [W_{i+1}(s)]_- ds\\
& \qquad + \int_0^t A_{i,2}(s) \exp\left\{-\int_s^t (A_i-B_i)(\tau)d\tau\right\} [W_{\ell_i}(s)]_- ds\\
&\le \int_0^t A_i(s) \exp\left\{-\int_s^t (A_i-B_i)(\tau)d\tau+\beta s\right\} e^{-\beta s}([W_{i+1}(s)]_-\vee [W_{\ell_i}(s)]_-) ds\\
&  \leq \sup_{s\in [0,t]}(e^{-\beta s}([W_{i+1}(s)]_-\vee [W_{\ell_i}(s)]_-)\;\int_0^t A_i(s) \exp\left\{-\int_s^t (A_i-|B_i|)(\tau)d\tau+\beta s\right\} ds.
\end{align*}
We now estimate the last term in the inequality above. After an integration by part, we have, since $\frac \beta 2 \geq C_1\geq |B_i|$ and $0\leq A_i\leq \gamma$,  
\begin{align*}
& \int_0^t A_i(s) \exp\left\{-\int_s^t (A_i-|B_i|)(\tau)d\tau+\beta s\right\}ds \\
& =\int_0^t A_i(s) \exp\left\{-\int_s^t A_i(\tau)d\tau\right\}\exp\left\{\int_s^t |B_i|(\tau)d\tau+\beta s\right\}ds \\
&  = \left[ \exp\left\{-\int_s^t (A_i-|B_i|)(\tau)d\tau+\beta s\right\}\right]_0^t -
\int_0^t \exp\left\{-\int_s^t (A_i-|B_i|)(\tau)d\tau+\beta s\right\}(-|B_i(s)|+\beta)ds  \\
&  \leq e^{\beta t}- \exp\left\{-\int_0^t (A_i-|B_i|)(\tau)d\tau\right\}  -
\int_0^t \exp\left\{\gamma (s-t)+\beta s \right\}(-C_1+\beta)ds \\
&  \leq e^{\beta t}-e^{-\gamma t}  -e^{-\gamma t}
\int_0^t \exp\left\{(\gamma+\beta)s \right\}(-C_1+\beta)ds .
\end{align*}
So 
\begin{align*}
& \int_0^t A_i(s) \exp\left\{-\int_s^t (A_i-|B_i|)(\tau)d\tau+\beta s\right\}ds\\
& \qquad  \leq e^{\beta t}- e^{-\gamma t}  - \frac{\beta-C_1}{\gamma+\beta} (e^{\beta t}-e^{-\gamma t})
 \leq  \frac{\gamma+C_1}{\gamma+\beta}e^{\beta t}= \frac{e^{\beta t}}{2},
\end{align*}
by the choice of $\beta=\gamma+2C_1$. This shows that 
$$
 \sup_{s\in [0,T]} e^{-\beta s}  [W_{i}(s)]_-  \leq \frac12  \sup_{s\in [0,T]} e^{-\beta s} ( [W_{i+1}(s)]_-\vee  [W_{\ell_i}(s)]_-).
$$
As $i\leq J_n(i_0+L)$, $i+1$ and $\ell_i$ belong to $\{i_0,\dots, J_{n-1}(i_0+L)\}$. So the right-hand side is less that $M_{n-1}/2$. Taking the supremum over all $i\in \{i_0, \dots, J_n(i_0+L)\}$ gives \eqref{MnMn-1}. 

By induction, we obtain that, for all $n\leq n_0$,  
$$
M_n\leq 2^{-(n-1)}M_1\leq C2^{-n},
$$
from which we derive the result. 
\end{proof}

Next we investigate the behavior of $J_n$ for large values of $n$: 

\begin{Lemma}\label{lem.Jn} There exists a constant $\alpha>0$ such that, for any $T\in \Z$, any $\ep\in (0,1]$ and any $n\in \N$, 
$$
\P\left[ |J_n(T)-T+\alpha n| \geq \ep n \right]\leq 2\exp\{ -\ep^2n /C\}. 
$$
\end{Lemma}

\begin{proof} Let $T\in \Z$ and  $X_n:= J^\omega_{n}(T)$, $n\in \N$. Then $(X_{n+1}-X_n)$ is a family of i.i.d. nonpositive random variables with law given by (for all $m\in \N$)
\begin{align*}
\P[X_{1}-X_0< -m]& \leq  \P\Bigl[\exists k\in \{1, \dots, K\}, \sharp\{ i\in \{T-m, \dots, T\}, \; T_i= k\}\le 1 \; \Bigr] \leq \sum_{k=1}^K m(1-\pi^k)^{m-1}\\
 & \leq Km\pi^{m-1}\leq C(\frac{1+\pi}{2})^m,
\end{align*}
where $\pi =\max_k (1-\pi^k)<1$ and $C$ depends on $\pi$ and $K$ only. Therefore $X_{1}-X_0$ satisfies Bernstein's conditions: there exist positive numbers $\nu$ and $c$ (depending on $\pi$ only) such that $\E\left[|X_1-X_0|^2\right] \leq \nu$ and $\E\left[ |X_1-X_0|^q\right]\leq q! \nu c^{q-2}$ for any integer $q\geq 3$ (see Subsection \ref{subsec:lemmaJn} in the Appendix). Let us set $\alpha= \E\left[X_{0}-X_1\right]>0$. From Bernstein's Inequality (Corollary 2.11 in \cite{BLM13})
$$
\P\left[ |J_n(T)-T+\alpha n| \geq x \right]\leq 2\exp\{ -x^2/(C(n+x))\}
$$
for some constant $C$ depending on $\pi$ only. This implies the result for $\ep\in (0,1]$.
\end{proof}

\section{The time function}\label{section.timefunction}

The goal of Sections \ref{section.timefunction} is to build the flux limiter at the junction. 
For doing so we consider the solution of our system  starting with a ``flat initial condition'' and look at the time it takes to reach 0 from a position (far) on the ingoing road. 

Let us fix from now on $e=(e^k)_{k=0,\dots K}$ such that  $H^k(-1/e^k)= \min_p H^k(p)$. We note for later use that $e^k>\pi^k \Delta_{\min}$, where $\Delta_{\min}$ is defined in Assumption (H2). We define $(U_{e,i}^\omega)_{i\in \Z}$  as the solution of \eqref{eq.SystJuncTOT} with initial condition defined  for any $i\in \Z$ by
$$U^{\omega}_{e,i}=\left\{\begin{array}{ll}
e^0i &{\rm if}\;  i\leq 0\\
 e^ki &{\rm if}\; i\geq 0\;  {\rm and}\; T_i=k.
\end{array}\right.$$ Then we set 
$$
\theta^\omega_e(t)=\inf\{ i\geq 0, \; U^\omega_{e,-i}(t)\leq 0\}.
$$
The quantity $\theta^\omega_e(t)$ is the number of vehicles having gone through $0$ at time $t$.  The goal of the section is to show, by using a concentration inequality, that $\theta_e(t)/t$ has a.s. a deterministic limit as $t\to +\infty$.

\subsection{Preliminaries}

Let us collect some basic facts on the $(U_{e,i}^\omega)_{i\in \Z}$ and on $\theta_e$. Let us set $v_e^k=\bar V^k(e^k/\pi^k)$ for $k\in \{0, \dots, K\}$ (where $\pi^0=1$ by convention).

\begin{Lemma}\label{lem.minvke}  We have $\min_{k\in \{0, \dots, K\}} v_e^k>0$. 
\end{Lemma}

\begin{proof} As $e>\pi\Delta_{\min}$ and $\mathcal Z$ is finite, assumption ($H_5$) implies the existence of $C>1$ such that $\tilde V^0_z(e^0)\geq C^{-1}$ for any $z\in \mathcal Z$. Thus $(\bar V^0)^{-1}(C^{-1})=\E[(\tilde V^0_{Z_0})^{-1}(C^{-1})] \leq e^0$, which shows that $v^0_e=\bar V^0(e^0)\geq C^{-1}$. The proof for $v^k_e$ (for $k=1, \dots , K$) works in the same way. 
\end{proof}

\begin{Lemma}\label{lem.delta} 
There exists $\delta>0$ such that
$$
\frac{d}{dt} U_{e,i}(t)\geq \delta\qquad \forall t\geq0, \; \forall i\in \Z.
$$
\end{Lemma} 

\begin{proof} Recall the definition of $\kappa$ in assumption  ($H_5$). Let us set $e^{\min}=\min_{k=0,\dots,K}e^k$ and 
$$
\delta := \min\{\kappa\ ,\  \min_{k=0, \dots, K} v^k_e\ ,\ \min_{x\in \R, \ z\in \mathcal Z} V_z(e^{\min},e^{\min},x)\ , \ \min_{x\in \R, \ z\in \mathcal Z} V_z(R_1-R_2+\Delta_{\min},R_1-R_2+\Delta_{\min},x) \}. 
$$
By Lemma \ref{lem.minvke}, the fact that $e^{\min}>\Delta_{\min}$ and assumption ($H_5$-(iii)) (combined with the fact that $\mathcal Z$ is finite and that $V$ is independent of $x$ for $x\notin [-R_0,0]$), we have that $\delta>0$. Fix $n\in \N$ large (say, $n\geq R_0/e^0$) and let $U^n$ be the solution to \eqref{eq.SystJuncTOT} with initial condition defined by 
  for any $i\in \Z$, $|i|\le n$ by
$$U^n_i(0)=\left\{\begin{array}{ll}
e^0i &{\rm if}\;  -n\le i\leq 0\\
 e^ki &{\rm if}\; 0\le i  \le n\;  {\rm and}\; k=T_i.
\end{array}\right.$$
If  $|i|>n$ we define $U^n_i(0)$, by induction by setting, if $i<-n$, 
$$
\tilde V^0_{Z_{i-1}}(U^n_{i}(0)-U^n_{i-1}(0))= v_e^0
$$
and, if $\ell_i> n$ and $k=T_i$, 
$$
\tilde V^k_{Z_{i}}(U^n_{\ell_i}(0)-U^n_{i}(0))= v_e^k.
$$
Then $U^n$ converges locally uniformly on $\Z\times [0,+\infty)$ to $U_e$ as $n\to +\infty$. We are going to show that the claim holds for $U^n$, which implies the claim for $U_e$. 

Let us first note that the claim holds for $t=0$. Indeed,  by definition of $U^n_{i}(0)$, we have, if $-n\leq i< 0$, 
$$
\frac{d}{dt} U^n_{i}(0)=V_{Z_i}(e^0,U_{\ell_i}^n(0)-U^n_i(0), e^0i) \geq V_{Z_i}(e^{\min},e^{\min}, ei) \geq \delta,
$$
while, if $i\ge 0$ and $k=T_i$, then 
$$
\frac{d}{dt} U^n_{i}(0)=\tilde V^k_{Z_i}(U_{\ell_i}^n(0)-U^n_i(0)) = \left\{\begin{array}{ll}
\tilde V^k_{Z_i}((\ell_i-i)e^k)\geq \delta & {\rm if }\; \ell_i\leq n ,\\
v^k_e \geq \delta & {\rm otherwise}.
\end{array}\right. 
$$
Finally, if $i<-n$,  then
$$
\frac{d}{dt} U^n_{i}(0)= \tilde V^0_{Z_i}(U^n_{i+1}(0)-U^n_i(0))=v_e^0\geq \delta.
$$
So we have proved that, in any case, $\frac{d}{dt} U^n_{i}(0)\geq \delta$. \\

By Lipschitz continuity in time of $I(t):=\inf_{i\in \Z} \frac{d}{dt} U^n_i(t)$, we have that $I(t)\geq \delta/2$ for $t\geq 0$ small. Let $[0, T]$ be an interval on which this inequality holds. We are going to show that actually $I(t)\geq \delta$ on $[0,T]$, which is enough to prove the claim. For this we argue by contradiction and assume that there is $(\bar t,j)\in [0,T]\times \Z$ such that $\frac{d}{dt} U^n_{j}(\bar t)< \delta$. Then, for $\ep>0$ small enough, 
$$
I:=\inf_{t\in [0, \bar t), \; i\in \Z}\;  \frac{d}{dt} U^n_{i}(t)+\frac{\ep}{\bar t-t}
$$
is less than $\delta$. In the next step we show that the infimum is actually a minimum. For this we first note that $\bar U^n_i(t):= U^n_i(0)+v^0_e t$ for $i\leq -N$ and $t\in [0,T]$ is a solution to \eqref{eq.SystJuncTOT} for $N$ large enough depending on $T$ and $n$. So, by  the approximate finite speed of propagation (Lemma \ref{lem.finitespeedPSJunc}) applied to $U^n_i$ and $\bar U^n_i$ for $i\leq -N$ and $t\in [0,T]$, we have for any $t\in [0, T]$
$$
\lim_{i\to -\infty} U^n_{i}(t)-U^n_{i}(0)= tv_e^0, 
$$
so that
$$
\lim_{i\to -\infty} \frac{d}{dt} U^n_{i}(t) = \lim_{i\to -\infty} \tilde V^0_{Z_i}(U^n_{i+1}(t)-U^n_i(t))= \lim_{i\to -\infty} \tilde V^0_{Z_i}(U^n_{i+1}(0)-U^n_i(0))=v_e^0\geq \delta.
$$
On the other hand, by the  construction of $U^n$, we have for $i>n$ and $T_i=k$ that $U^n_i(t)= U^n_i(0)+tv_e^k$, so that 
$$
\lim_{i\to +\infty, \ T_i=k} \frac{d}{dt} U^n_{i}(t) =v_e^k\geq  \delta. 
$$
This shows that the infimum in the definition of $I$ is a minimum: let $(t_0,i_0)$ be a minimum point such that $i_0$ is maximal. Recalling that $I(0)\geq \delta$, we have $t_0>0$. By the optimality of $(t_0,i_0)$ and the maximality of $i_0$,  we have
\be\label{liaeuzhkrdjfg}
 \frac{d}{dt} U^n_{i_0+1}(t_0)>  \frac{d}{dt} U^n_{i_0}(t_0)\; {\rm and}\;  \frac{d}{dt} U^n_{\ell_{i_0}}(t_0)>  \frac{d}{dt} U^n_{i_0}(t_0).
\ee
By optimality of $t_0>0$, we also have  (omitting the dependence of $V_{Z_{i_0}}$ with respect to its parameters to simplify the notation)
$$
0= \partial_{e_1} V_{Z_{i_0}}\left( \frac{d}{dt}U^n_{i_0+1}(t_0)-  \frac{d}{dt} U^n_{i_0}(t_0)\right) 
+ \partial_{e_2} V_{Z_{i_0}}\left( \frac{d}{dt}U^n_{\ell_{i_0}}(t_0)-  \frac{d}{dt} U^n_{i_0}(t_0)\right) 
+ \partial_x V_{Z_{i_0}} \frac{d}{dt} U^n_{i_0}(t_0) +\frac{\ep}{(\bar t-t_0)^2}.
$$
By \eqref{liaeuzhkrdjfg}, all the terms are nonnegative except perhaps $\partial_x V_{Z_{i_0}}$. By assumption ($H_5$-(ii)), we cannot have $U^n_{i_0}(t_0) \geq -R_1$ since in this case $\partial_x V_{Z_{i_0}} \geq 0$. So $U^n_{i_0}(t_0) < -R_1$. Now, if $U^n_{\ell_{i_0}}(t_0)\geq U^n_{i_0+1}(t_0)$, then by assumption ($H_5$-(i)), 
$V_{Z_{i_0}}$ does not depend on $x$ and therefore $\partial_x V_{Z_{i_0}} = 0$. Thus $U^n_{\ell_{i_0}}(t_0)< U^n_{i_0+1}(t_0)$. By Lemma \ref{lem.CompDeBase}, with $i=i_0+1$ and $j=\ell_{i_0}>i_0+1$, this implies that $U^n_{\ell_{i_0}}(t_0)\ge -R_2+\Delta_{\min}$ and, hence, $U^n_{i_0+1}(t_0)\ge-R_2+\Delta_{\min}$. Thus 
$$
\frac{d}{dt} U^n_{i_0}(t_0) \geq  V_{Z_{i_0}}(R_1-R_2+\Delta_{\min},R_1-R_2+\Delta_{\min},U^n_{i_0}(t_0))\geq \delta, 
$$
which is again impossible. 
\end{proof}

A immediate consequence of Lemma \ref{lem.delta} is that $U_{e,i}(t)\to +\infty$ as $t\to +\infty$ and therefore $\theta_e(t)\to +\infty$ as $t\to +\infty$. Next we show a bound from above for $\theta_e$.

\begin{Lemma}\label{lem.estitheta} There exists a constant $C_\theta>0$ such that 
$$
 0 \leq \theta_e(t)-\theta_e(s) \leq C_\theta(t-s+1) \qquad \forall 0\leq s\leq t. 
$$
Moreover, for $t\geq 0$,
$$
\theta_e(t)\leq C_\theta t. 
$$
\end{Lemma}

\begin{proof} Let us first prove the second statement. For this we note that, for any $i\geq \|V\|_\infty t/e^0$, we have 
$$
U_{e,-i}(t) \leq -e^0i+ \|V\|_\infty t \leq 0. 
$$
So $\theta_e(t) \leq  \|V\|_\infty t/e^0$, which proves the claim. 

We now prove the first statement. Let $i_0:= -\theta_e(s)$. Then $U_{e,i_0}(s)\leq 0$. Let $\delta$ be the constant given by Lemma \ref{lem.delta}. Assume first that $s\geq R_2/\delta$. Then, by Lemma \ref{lem.delta}, $U_{e,i_0}(s-R_2/\delta)\leq U_{e,i_0}(s)-R_2\leq - R_2$. Recalling Remark \ref{rem.lem.distminnext} after Lemma \ref{lem.CompDeBase} we have, for any $i\leq i_0-\|V\|_\infty(t-s+R_2/\delta)/\Delta_{\min}$, 
$$
U_{e,i}(t)\leq U_{e,i}(s-R_2/\delta)+ \|V\|_\infty(t-s+R_2/\delta)\leq U_{e,i_0}(s-R_2/\delta)- \Delta_{\min} (i_0-i) + \|V\|_\infty(t-s+R_2/\delta)\leq -R_2< 0. 
$$
So, for any $i\leq  i_0 -\|V\|_\infty(t-s+R_2/\delta)/\Delta_{\min}$, we obtain $U_{e,i}(t) < 0$. This shows that
$$
\theta_e(t) \leq -i_0+ \|V\|_\infty(t-s+R_2/\delta)/\Delta_{\min}\leq  \theta_e(s) + C(t-s+1). 
$$
If $t\leq R_2/\delta$, the conclusion obviously holds. Finally, if $s\leq R_2/\delta$ and $t> R_2/\delta$, then, by the previous inequality and the first part of the proof, we have
\begin{align*}
\theta_e(t)  \leq & \theta_e(R_2/\delta) + C(t-R_2/\delta+1) \\
\leq&  \|V\|_\infty R_2/(\delta e^0)+  C(t-s+1)\\
\leq& \theta_e(s) + \|V\|_\infty R_2/(\delta e^0)+  C(t-s+1)\\
\leq &\theta_e(s) + C'(t-s+1). 
\end{align*}
\end{proof}

\subsection{A concentration inequality}

In this section, we prove a concentration inequality for 
$$
\theta_e(t)= \inf\{i\geq 0, \; U_{e, -i}(t)\leq 0\}. 
$$

\begin{Theorem}\label{thm:concentrationTOT} There is a constant $C>0$ such that for any $\ep \in (0,C^{-1}]$ and any $t\ge C \ep^{-1}$,  
$$
\P\left[ | \theta_e (t)-\E[\theta_e(t)]| \ \geq \ep t\right] \leq C\exp\{ -\ep^2 t/C\}.
$$
\end{Theorem}

The proof requires several steps. The first issue is that $\theta_e(t)$ depends a priori on all the $Z_i$, even for any $i\geq 0$ large. In order to reduce this dependence, we  introduce the auxiliary quantity $\theta^m_e(t)$ defined, for any $m\in \N$ large (say $m\geq 2$), by 
$$
\theta^m_e(t)=  \inf\{i\geq 0, \; U^m_{e, -i}(t)\leq 0\}
$$
where the $(U^m_{e,i})_{i\in\Z}$ is the solution of \eqref{eq.SystJuncTOT} where the sequences $Z_i$ and the initial condition $U_e(0)$  are replaced into $Z_i^m$ and $U^m_e(0)$ defined as follows: 
$$
Z^m_i:= \left\{\begin{array}{ll}
Z_i & {\rm if}\; i\leq m-1\\
z^k_{\min} & {\rm if}\; i\geq m\; {\rm and}\; T_i=k
\end{array}\right.
$$
and 
$$
U^m_{e,i}(0)=  \left\{\begin{array}{ll}
 U_{e,i}(0) & {\rm if}\; i\leq m-1\\
e^{T_i}(m-1)+e^{T_i}\sharp\{ j\in \{ m, \dots, i\}, \; T_j=T_i\} & {\rm if}\; i\geq m.
\end{array}\right.
$$
Let us recall that the $z^k_{\min}$ are introduced at the beginning of Subsection \ref{subsec.distVehicle}. 
Before proceeding, let us collect several important properties of the $U^m_{e,i}$. 

\begin{Lemma}\label{lem.Um}  \begin{enumerate}
\item For each $i\in \Z$ with $i\leq m-1$, $U^m_{e,i}$ is $\sigma\{Z_j, \; j\in \{i, \dots, m-1\}\}-$measurable. 
\item There exists $\delta>0$ such that $(d/dt) U^m_{e,i}(t)\geq \delta$ for any $i\in \Z$ and any $t\geq 0$. 
\item There exists a constant $C>0$ such that 
$$
 0 \leq \theta^m_e(t)-\theta^m_e(s) \leq C(t-s+1) \qquad {\rm and}\qquad \theta^m_e(t)\leq C t \qquad \forall 0\leq s\leq t.
$$
\item Setting $\sigma_i:= \inf\{j\geq i+1, \; Z_j= z^{T_i}_{\min} \}$ and $\sigma^0_i=\inf\{j\geq i+1, \; Z_j=z^0_{\min}\}$, we have 
\be\label{distnextUm}
U^m_{e,\ell_i}(t)\leq U^m_{e,i}(t)+ C (1+\sigma^0_{\sigma_i}\wedge m- i) \qquad \forall i\leq m-1, \; \forall t\geq 0.
\ee
\end{enumerate}
\end{Lemma}

\begin{proof} 1) Let us first check that $U^m_{e,i}(t)= U^m_{e,i}(0)+\tilde V^k_{z^k_{\min}}(e^k)t$ for any $i\geq m$ with $T_i= k$. Indeed, for such any $i\geq m$, we have 
$$
\frac{d}{dt}U^m_{e,i}(t)= \tilde V^k_{z^k_{\min}}(e^k)
$$
while 
\begin{align*}
V_{Z^m_i}(U^m_{e,i+1}(t)-U^m_{e,i}(t),U^m_{e,\ell_i}(t)-U^m_{e,i}(t),U^m_{e,i}(t))& = \tilde V^k_{z^k_{\min}}(U^m_{e,\ell_i}(t)-U^m_{e,i}(t))\\
& = 
 \tilde V^k_{z^k_{\min}}(U^m_{e,\ell_i}(0)-U^m_{e,i}(0))= \tilde V^k_{z^k_{\min}}(e^k).  
\end{align*}
The measurability of $U^m_{e,i}$ for $i\leq m-1$ can then be  checked by backward induction. For $i=m-1$, $U^m_{e,m-1}$ solves (if we set $T_i=k$), 
$$
\frac{d}{dt} U^m_{e,m-1}(t)= \tilde V^{k}_{Z_i}(U^m_{\ell_i}(t)-U^m_{e,i}(t)), \; t\geq 0, \qquad U^m_{e,m-1}(0)=e^k(m-1).
$$
Given $Z_i$, the above equation has deterministic coefficients since $U^m_{\ell_i}(t)= e^k(m-1)+\tilde V^k_{z^k_{\min}}(e^k)t$. For $i\leq m-2$, it can be proved by induction in the same way that $U^m_{e,i}$ satisfies an ODE with coefficients which are $\sigma\{Z_j, \; j\in \{i, \dots, m-1\}\}-$measurable.

2) \& 3) The existence of $\delta>0$ such that $(d/dt) U^m_{e,i}(t)\geq \delta$  is a consequence of Lemma \ref{lem.delta} (which is a deterministic statement). In the same way, the estimate on $\theta^m_e$ is an application of Lemma \ref{lem.estitheta}. 

4) Let us finally check \eqref{distnextUm}. For this we use Lemma \ref{key.estiUlli-Ui}. Fix $i\leq m-1$ and let $k=T_i$, $i_0:= \inf\{j\geq i+1, \; Z^m_j= z^k_{\min}\}$ and $i_1=\inf\{j\geq i_0, \; Z^m_j= z^0_{\min}\}$ (if this exists).  Note that $i_0= \sigma_i\wedge (\inf\{j\ge m, T_j=k\})$ and $i_1= \sigma^0_{\sigma_i}$ if $\sigma^0_{\sigma_i}<m$. If $e^0i_1\leq -R_0$, then by the second part of Lemma \ref{key.estiUlli-Ui} we have 
$$
U^m_{e,\ell_i}(t)\leq U^m_{e,i_0}(t)\leq  U^m_{e,i}(t)+  C_0(1+U^m_{i_1}(0)-U^m_i(0)+ i_1-i)\le U^m_{e,i}(t)+ C (1+i_1- i) \qquad \forall t\geq 0.
$$
As $i_1\leq 0< m-1$ we also have $i_1=\sigma^0_{\sigma_i}< m$, which proves the inequality in this case.

 Let us now assume that $e^0i_1> -R_0$. According to the first part of Lemma \ref{key.estiUlli-Ui} we have 
$$
U^m_{e,\ell_i}(t)\leq  U^m_{e,i_0}(t)\leq U^m_{e,i}(t)+ C_0(e^0i)_-+ U^m_{e,i_0}(0)-U^m_{e,i}(0)+C_0\sharp \{j\in \{i+1, \dots, i_0\}, \; T(Z^m_j)=k\}.
$$
As $e^0i_1> -R_0$, we have 
$$
(e^0i)_- \leq e^0(i_1\wedge m- i)+ R_0\leq e^0(\sigma^0_{\sigma_i}\wedge m-i) +R_0.
$$
On the other hand, by the construction of the $U^m_{e,j}(0)$, we have $U^m_{e,i_0}(0)= (m\wedge \sigma_i)e^k$. Finally, as by definition of the $ Z^m_j$ and of $i_0$ we also have (with $k:=T_i$) 
$$
\sharp \{j\in \{i+1, \dots, i_0\}, \; T(Z^m_j)=k\} \leq \sharp \{j\in \{i+1, \dots, m-1\}, \; T(Z^m_j)=k\}+1 \leq \sigma_i\wedge m-i+1. 
$$
This shows that 
$$
U^m_{e,\ell_i}(t)\leq U^m_{e,i}(t)+ C (\sigma^0_{\sigma_i}\wedge m- i+1) \qquad \forall t\geq 0.
$$

\end{proof}

We now note that $\theta_e$ and $\theta^m_e$ are close.

\begin{Lemma}\label{thetathetam}  There exists a constant $C>0$ such that, for any $ t\geq C$ and if  $m=[Ct]$,   
$$
\P\left[|\theta_e(t)- \theta^{m}_e(t) |> C \right]\leq C\exp\{- t/C\}.
$$
\end{Lemma}

\begin{proof} Note that $(U_{e,i})$ and $(U^m_{e,i})$ solve the same equation for $i\leq m-1$ with the same initial condition.  Lemma \ref{lem.finitespeedPSJunc} on the approximate speed of propagation then states that there exists constants $C>0$ and $\beta >0$ such that, for all $n\in \N$, $n\ge 1$ and $i\leq  J_n(m-1)$, 
 $$
 |U^{m}_{e,i}(s) -U_{e,i}(s)| \leq C\ 2^{-n}e^{\beta s}  \qquad \forall s  \geq 0.
$$
Fix $t>0$ and let us choose $n = [\beta t/\ln(2)]+1$ and $m = [(1+\alpha) n]$ where $\alpha>0$ is defined in Lemma \ref{lem.Jn}. In the event $\{J_n(m-1)\geq 0\}$, we have
 $$
 |U^{m}_{e,i}(s) -U_{e,i}(s)| \leq C \qquad \forall s  \in[ 0,t], \; \forall i\leq 0.
$$
Let us check that this inequality implies in the event $\{J_n(m-1)\geq 0\}$ that
\be\label{zalejfkdngc}
| \theta^{m}_e(t) -\theta_e(t)| \leq C.
\ee
Indeed, if $U^m_{e,i}(t)\leq 0$, then $U_{e,i}(t)\leq C$. Assume $t\geq C\delta^{-1}$ where $\delta$ is defined in Lemma \ref{lem.delta}. Then Lemma \ref{lem.delta} implies that $U_{e,i}(t-C\delta^{-1})\leq 0$. This shows that 
$$
\theta_e(t-C\delta^{-1})\leq \theta^m_e(t),
$$
and thus, by Lemma \ref{lem.estitheta}, that 
$$
\theta_e(t)-C' \leq \theta^m_e(t),
$$
for some new constant $C'$. If $t< C\delta^{-1}$, then, by Lemma  \ref{lem.estitheta}, 
$$
\theta_e(t)\leq C'  \leq \theta^m_e(t)+C'. 
$$
Therefore $\theta_e(t)-\theta^m_e(t)\leq C'$ in any case. The inequality 
$$
\theta^m_e(t)-C' \leq \theta_e(t)
$$
can be checked in the same way, by using points 2) and 3) of Lemma \ref{lem.Um}. This proves \eqref{zalejfkdngc}. 

As \eqref{zalejfkdngc} holds in the event  $\{J_n(m-1)\geq 0\}$, we get
$$
\P\left[|\theta_e(t)- \theta^{m}_e(t)| > C\right]\leq \P\left[ J_n(m-1)< 0\right]\leq \P\left[ J_n(m)< 0\right]. 
$$
Recalling the choice of $m$ and Lemma \ref{lem.Jn} (with $\ep=1$) we have
$$
 \P\left[ J_n(m)< 0\right]\leq  \P\left[ J_n(m) -m+\alpha n < -m+\alpha n+1\right]\leq \P\left[ J_n(m) -m+\alpha n< -n \right] \leq C\exp\{-t/C\}.
 $$
This gives the result. 
\end{proof}

The key step of the proof of Theorem \ref{thm:concentrationTOT} consists in establishing a concentration inequality for $\theta^{m}_e(t)$. 
To do so, let us set, for $n\in\N$,  $\mathcal F_{m,n}=\sigma\{ Z_{i}, \; i\in \{ m-n, \dots, m-1\}\ \}$ if $n\geq 1$ and $\mathcal F_{m,0}=\{\varnothing, \Omega\}$. We also set
$$
M_n(t)= \E\left[ \theta^m_e(t)\ |\ \mathcal F_{m,n}\right]- \E\left[ \theta^m_e(t)\right].
$$
Note that $(M_n(t))$ is a martingale with $M_0(t)=0$. \\

As, by Lemma \ref{lem.Um}, $\theta^m_e(t)\leq Ct$ and $\{\theta^m_e(t)\leq n\}= \{U^m_{-n}(t)\leq0\}$ is $\mathcal F_{m,m+n}-$measurable for any $n\in \N$, we have that $M_{n}(t)=\theta^m_e(t)-\E[\theta^m_e(t)]$  for $n\geq  \bar n:=[Ct]$ for $C$ large enough.\\

The next step is instrumental and consists in estimating $\left| M_{n+1}(t)-M_n(t)\right|$. 

\begin{Lemma} \label{lem.mol=jkzndcTOT} For any $n\in \N$, 
\be\label{mol=jkzndcTOT}
\left| M_{n+1}(t)-M_n(t)\right|\leq  C(1+\sigma^0_{\sigma_{m-n}}\wedge m-(m-n)), 
\ee
where $\sigma^0$ and $\sigma$ are defined in Lemma \ref{lem.Um}. 
\end{Lemma}

\begin{proof}
Let us first remark that, for any $n\geq 0$, $\theta^m_e(t){\bf 1}_{\{\theta^m_e(t)\leq (m-n-1)_-\}}$ is $\mathcal F_{m,n}-$measurable. Hence 
$$
 \left| M_{n+1}(t)-M_n(t)\right|= \left| \E\left[ \theta^m_e(t)\ |\ \mathcal F_{m,n+1}\right] -\E\left[ \theta^m_e(t)\ |\ \mathcal F_{m,n}\right]\right| {\bf 1}_{\{\theta^m_e(t)> (m-n-1)_-\}}.
$$
In the next steps, we  work in $\{U^m_{m-n}(t)> 0\}$. Let us introduce some notation. Given $n\in \N$ and a continuous componentwise nondecreasing map $x=(x^1, \dots, x^K):[0,+\infty)\to \R^{K}$,  we denote by $\hat U^{n,x}=(\hat U^{n,x,\omega}_i)_{i\leq m-n-1}$ the solution to 
\begin{align*}
&\frac{d}{dt} \hat U^{n,x,\omega}_i(\tau) \\
&=
\left\{\begin{array}{ll}
 V_{Z_i^\omega}(\hat U^{n,x,\omega}_{i+1}(\tau)-\hat U^{n,x,\omega}_{i}(\tau), \hat U^{n,x,\omega}_{l_i^\omega}(\tau)-\hat U^{n,x,\omega}_{i}(\tau),\hat U^{n,x,\omega}_{i}(\tau))  & {\rm if}\; i\leq m-n-2,\; \ell_i\leq m-n-1,\\
V_{Z_i^\omega}(\hat U^{n,x,\omega}_{i+1}(\tau)-\hat U^{n,x,\omega}_{i}(\tau), x^{T_i}(\tau)-\hat U^{n,x,\omega}_{i}(\tau),\hat U^{n,x,\omega}_{i}(\tau))
& {\rm if}\; i\leq m-n-2,\; \ell_i\geq m-n,\\
 V_{Z_i^\omega}(x^{k_0}(\tau)-\hat U^{n,x,\omega}_{i}(\tau), x^{T_i}(\tau)-\hat U^{n,x,\omega}_{i}(\tau),\hat U^{n,x,\omega}_{i}(\tau))
& {\rm if}\; i= m-n-1,
  \end{array}\right. 
\end{align*}
with for $ i\leq m-n-1$
$$
\hat U^{n,x,\omega}_{i}(0)=\left\{\begin{array}{ll}
e^0i &{\rm if}\;  i\leq 0\\
 e^ki &{\rm if}\; i\geq 0\;  {\rm and}\; k=T_i
\end{array}\right.$$ 
where $k_0\in \{1, \dots, K\}$ is such that $x^{k_0}(0)\leq x^k(0)$ for any $k\in \{1, \dots, K\}$ (if there are several minimizers of $x^k(0)$ we choose the smallest one). An important property of the $(\hat U^{n,x,\omega}_i)_{i\leq m-n-1}$ is that they depend on $\{Z_i, \; i\leq m-n-1\}$ only. We also define 
$$
\hat \theta^{n,x,\omega}(t)= \inf \left\{ i\ge  (m-n-1)_-, \; \hat U^{n,x,\omega}_{-i}(t)\leq 0\right\}. 
$$
We note that $U^{m,\omega}_{e,i}(\tau)= \hat U^{n,\left(U^{m,\omega}_{e,l^{m-n,\omega}(k)}\right)_{k=1,\dots,K},\omega}_i(\tau)$ for any $i\leq m -n-1$ and $\tau\geq 0$ where
\be\label{defellik}
l^{i,\omega}(k)=\inf\{j\ge i,\; T_j^\omega=k\}.
\ee
 Moreover, $\theta^{m,\omega}_e(t)= \hat \theta^{n, \left(U^{m,\omega}_{e,l^{m-n,\omega}(k)}\right)_{k=1,\dots,K},\omega}_e(t)$ in $\{U_{e,m-n}^{m,\omega}(t)>0\}$. As $\hat \theta^{n,x}(t)$ depends only on $\{Z_i, \; i\leq m- n-1\}$ while the $U^{m,\omega}_{e,l^{m-n,\omega}(k)}$ are $\mathcal F_{m,n}-$measurable, we have, in $\{U_{e,m-n}^{m,\omega}(t)>0\}$,  
$$
\E\left[ \theta^m_e(t)\ |\ \mathcal F_{m,n}\right]= 
\E\left[ \hat \theta^{n,x}(t)\right]_{x=\left(\hat U^{m,\omega}_{e,l^{m-n,\omega}(k)}\right)_{k=1,\dots,K}}.
$$
In the same way, we have
$$
\E\left[ \theta^m_e(t)\ |\ \mathcal F_{m,n+1}\right]\le 1+ \E\left[ \hat \theta^{n+1,\tilde x}(t)\right]_{\tilde x=\left(\hat U^{m,\omega}_{l^{{m-n-1},\omega}(k)}\right)_{k=1,\dots,K}}.
$$
So 
\begin{align}\label{deuxJUNCTOT}
&\left| M_{n+1}(t)-M_n(t)\right| \leq\\
& \qquad 1 +\Bigl|  \E\left[ \hat \theta^{n+1,\tilde x}(t)\right]_{\tilde x=\left(U^{m,\omega}_{l^{{m-n-1},\omega}(k)}\right)_{k=1,\dots,K}}- \E\left[ \hat \theta^{n,x}(t)\right]_{x=\left(U^{m,\omega}_{l^{m-n,\omega}(k)}\right)_{k=1,\dots,K}}\Bigr| {\bf 1}_{\{\theta^m_e(t)> (m-n-1)_-\}}.\notag
\end{align}

Next we estimate the difference between $\E\left[ \hat \theta^{n+1,\tilde x}(t)\right]$ and $\E\left[ \hat \theta^{n,x}(t)\right]$ when $\tilde x=\left(U^{m,\omega}_{l^{{m-n-1},\omega}(k)}\right)_{k=1,\dots,K}$ and $x=\left(U^{m,\omega}_{l^{m-n,\omega}(k)}\right)_{k=1,\dots,K}$ (recall that we work in $\{U^m_{m-n}(t)> 0\}$).  For this we fix two $C^1$  maps $x,\tilde x:[0,\infty)\to \R^K$ such that there exists $k_0\in \{1, \dots, K\}$ and $\gamma\geq 1$ with, for any $k\in \{1, \dots, K\}$,  
\begin{align}\label{hypsurxtildex}
& (d/d\tau)x^k(\tau)\geq \delta\; \text{ and} \;(d/d\tau)\tilde x^k(\tau)\geq  \delta, \qquad 
 x^k=\tilde x^k \; {\rm if }\; k\neq k_0\notag\\
&\text{and}\; -\gamma+x^{k_0}(\tau) \leq \tilde x^{k_0}(\tau)\leq x^{k_0}(\tau)\qquad \forall \tau\geq 0.
\end{align}
Note that the conditions above are satisfied by $\tilde x=(U^{m,\omega}_{l^{{m-n-1},\omega}(k)})_{k=1,\dots,K}$ and $x=(U^{m,\omega}_{l^{m-n,\omega}(k)})_{k=1,\dots,K}$ with $\gamma = C(1+\sigma^0_{\sigma_{m-n}}\wedge m-(m-n))$ and $k_0= T_{m-n}$ thanks to Lemma \ref{lem.Um}. Note also that the $\hat U^{n,x,\omega}_i$ and $\hat \theta^{n,x,\omega}$ satisfy the same conclusion as $U_e$ and $\theta_e$ in Lemmas \ref{lem.delta} (with the same constant $\delta$) and \ref{lem.estitheta}, the proof being  the same. 

The main part of the proof consists in showing that, under \eqref{hypsurxtildex}, 
\be\label{mol=jkzndc}
\left| \E\left[ \hat \theta^{n+1,\tilde x}(t)\right]-\E\left[ \hat \theta^{n,x}(t)\right]\right| \leq C(1+\gamma). 
\ee
For this we first check that 
\be\label{sixTOT}
\E\left[ \hat \theta^{n+1,\tilde x}(t)\right]\leq \E\left[ \hat \theta^{n+1,x}(t)\right]
\leq \E\left[ \hat \theta^{n,x}(t)\right]+1. 
\ee
The first inequality holds by comparison, which implies from assumption \eqref{hypsurxtildex} that $\hat U^{n+1,\tilde x}_i(s)\leq \hat U^{n+1, x}_i(s)$ for any $s\geq0$ and any $i\leq m-n-2$.  For the second inequality, let us set $W_i^\omega(s)= \hat U^{n+1, x, \tau_1\omega}_{i-1}(s)$ for $i\leq m-n$ and $s\geq0$. 
Then $W_i^\omega$ solves (since $Z_{i-1}^{\tau_1\omega}= Z_i^\omega$, $T^{\tau_1\omega}_{i-1}= T^\omega_i$ and $\ell^{\tau_1\omega}_{i-1}= \ell^\omega_i-1$)
\begin{align*}
&\frac{d}{dt} W^{\omega}_i(\tau) \\
&=
\left\{\begin{array}{ll}
 V_{Z_i^\omega}(W^\omega_{i+1}(\tau)-W^\omega_{i}(\tau), W^\omega_{l_i^\omega}(\tau)-W^\omega_{i}(\tau),W^\omega_{i}(\tau))  & {\rm if}\; i\leq m-n-2,\; \ell_i\leq m-n-1,\\
V_{Z_i^\omega}(W^\omega_{i+1}(\tau)-W^\omega_{i}(\tau), x^{T_i}(\tau)-W^\omega_{i}(\tau),W^\omega_{i}(\tau))
& {\rm if}\; i\leq m-n-2,\; \ell_i\geq m-n,\\
 V_{Z_i^\omega}(x^{k_0}(\tau)-W^\omega_{i}(\tau), x^{T_i}(\tau)-W^\omega_{i}(\tau),W^\omega_{i}(\tau))
& {\rm if}\; i= m-n-1,
  \end{array}\right. 
\end{align*}
$$
W^{\omega}_{i}(0)=\left\{\begin{array}{ll}
 e^0 (i-1) &{\rm if}\; i\le1\\
 e^k(i-1)&{\rm if} \; i\ge 1 \; {\rm and}\; T^{\tau_1\omega}_{i-1}=T_i^\omega=k
 \end{array}
 \right.
 \leq \hat U^{n,x,\omega}_i(0)\; {\rm for} \; i\leq m-n-1,
$$
Therefore by comparison 
$$
\hat U^{n+1, x, \tau_1\omega}_{i-1}(\tau)=W^\omega_i(\tau)\leq \hat U^{n,x,\omega}_i(\tau)\qquad \forall \tau\geq0, \; \forall i\leq m-n-1,
$$
which implies that 
$$
\hat \theta^{n+1,x,\tau_1\omega}(t)\leq \hat \theta^{n,x,\omega}(t)+1
$$
and gives the second inequality in \eqref{sixTOT} after taking expectation. \\

Using \eqref{hypsurxtildex}, the definitions of
 $\hat U^{n+1,\tilde x,\omega}_i(0)$ and $\hat U^{n, x,\omega}_{i}(0)$ and the fact that  $(d/dt) \hat U^{n+1,\tilde x,\omega}_i(t)\geq \delta>0$ (since the $\hat U^{n,x}$ satisfy the conclusion of Lemma \ref{lem.delta}), we get that there exists $C_1>0$ large (but which does not dependent on $m$, $n$, $\omega$ and $i$) such that
\be\label{kuqebsdlmk}
\hat U^{n+1,\tilde x,\tau_1\omega}_{i-1}(C_1\gamma) \geq \hat U^{n, x,\omega}_{i}(0)\qquad \forall i\leq m-n-1,
\ee
and 
$$
\tilde x^{k}(t+C_1\gamma)\geq x^{k}(t)\qquad \forall t\geq 0, \; \forall k\in \{1, \dots, K\}.
$$
We claim that, for some constant $C_2>0$,  
\be\label{totTOT}
\E\left[ \hat \theta^{n,x}(t)\right] \leq  \E\left[ \hat \theta^{n+1,\tilde x}(t)\right]+ C_2\gamma.
\ee
To prove \eqref{totTOT} let $W^\omega_i(t)= \hat U^{n+1,\tilde x,\tau_1\omega}_{i-1}(t+C_1\gamma)$ for $i\leq m-n-1$. For $i\leq m-n-1$, $W_i$ solves, exactly as above,  
\begin{align*}
&\frac{d}{dt} W^{\omega}_i(\tau) \\
&=
\left\{\begin{array}{ll}
 V_{Z_i^\omega}(W^\omega_{i+1}(\tau)-W^\omega_{i}(\tau), W^\omega_{l_i^\omega}(\tau)-W^\omega_{i}(\tau),W^\omega_{i}(\tau))  & {\rm if}\; i\leq m-n-2,\; \ell_i\leq m-n-1,\\
V_{Z_i^\omega}(W^\omega_{i+1}(\tau)-W^\omega_{i}(\tau), \tilde x^{T_i}(\tau+C_1\gamma)-W^\omega_{i}(\tau),W^\omega_{i}(\tau))
& {\rm if}\; i\leq m-n-2,\; \ell_i\geq m-n,\\
 V_{Z_i^\omega}(\tilde x^{k_0}(\tau+C_1\gamma)-W^\omega_{i}(\tau), \tilde x^{T_i}(\tau+C_1\gamma)-W^\omega_{i}(\tau),W^\omega_{i}(\tau))
& {\rm if}\; i= m-n-1,
  \end{array}\right. 
\end{align*}
$$
W^{\omega}_{i}(0)= \hat U^{n+1,\tilde x,\tau_1\omega}_{i-1}(C_1\gamma) \geq \hat U^{n,x,\omega}_i(0)\; {\rm for} \; i\leq m-n-1.
$$
As $\tilde x^{k}(\tau+C_1\gamma)\geq  x^{k}(\tau)$ and as $V_z$ is increasing in the first two variables, we obtain by comparison that 
$$
\hat U^{n+1,\tilde x,\tau_1\omega}_{i-1}(t+C_1\gamma) = W^\omega_i(t) \geq \hat U^{n,x,\omega}_i(t)\qquad \forall t\geq0, \; \forall i\leq m-n-1.
$$
Therefore 
$$
1+\hat \theta^{n+1,\tilde x,\tau_1\omega}(t+C_1\gamma)\geq \hat\theta^{n,x,\omega}(t).
$$
Using Lemma \ref{lem.estitheta} (which holds for $\hat \theta$ as explained above), we obtain 
$$
 C_2\gamma+ \hat \theta^{n+1,\tilde x,\tau_1\omega}(t)\geq \hat\theta^{n,\omega}(t),
$$
which implies  \eqref{totTOT} after taking expectation.\\

Combining \eqref{sixTOT} and \eqref{totTOT} gives \eqref{mol=jkzndc}. Then recalling \eqref{deuxJUNCTOT},  we obtain \eqref{mol=jkzndcTOT}.
\end{proof}

Let us set 
$$
\xi_{i}(t) = |M_{i+1}(t)-M_{i}(t)|.  
$$
and 
$$
[M]_n= \sum_{i=1}^{n} (\xi_{i}(t) )^2, \qquad \lg M\rg_n= \sum_{i=1}^{n} \E\left[ (\xi_{i}(t) )^2\ |\ \mathcal F_{m,i}\right].
$$
Following \cite[Theorem 2.1]{BeTo08}, the following concentration inequality holds: 
$$
\P\left[ |M_n|\geq x, \; [M]_n+  \lg M\rg_n\leq y\right] \leq 2 \exp\{ -x^2/(2y)\}. 
$$
This implies that 
\be\label{concentrationBercuTouatiTOT}
\P\left[ |M_n|\geq x\right] \leq 2 \exp\{ -x^2/(2y)\}+ \P\left[[M]_n+  \lg M\rg_n> y\right] . 
\ee
\begin{Lemma} \label{lem.estilgMrgTOT} There exists a constant $C>0$ such that 
$$
\P\left[[M]_n+  \lg M\rg_n> Cn\right] \leq  \exp\{ -\frac{n}{C}\}.
$$
\end{Lemma}

\begin{proof}
In view of Lemma \ref{lem.mol=jkzndcTOT} we have 
$$
\left| \xi_i(t)\right|\leq    C(1+\sigma^0_{\sigma_{m-i}}\wedge m-(m-i)).
$$
We first replace the right-hand side by a more suitable random variable. For $k\in \{1, \dots, K\}$, let $\sigma^k_i:= \inf\{j\geq i+1, \; T_j=k, \; Z_j=z^k_{\min}\}$. Note that $\sigma^k_i$ is independent of $\{Z_j, \; j\leq i\}$ and that $\sup_k \sigma^k_i\geq \sigma_i$. 
Then 
$$
\left| \xi_i(t)\right|\leq    C\sum_{k=1}^K (1+\sigma^0_{\sigma^k_{m-i}}\wedge m-(m-i)). 
$$
As $\sigma^0_{\sigma^k_{m-i}}\wedge m$ is $\mathcal F_{m, i}-$measurable, we have 
\begin{align*}
[M]_n+  \lg M\rg_n & \leq Cn+C\sum_{k=1}^K \sum_{i=1}^{n}(\sigma^0_{\sigma^k_{m-i}}\wedge m-(m-i))^2\\
& \leq Cn+C\sum_{k=1}^K \sum_{i=1}^{n}(\sigma^0_{\sigma^k_{m-i}}\wedge m-\sigma^k_{m-i}\wedge m)^2+ C\sum_{k=1}^K \sum_{i=1}^{n}(\sigma^k_{m-i}\wedge m-(m-i))^2.
\end{align*}
For $k\in \{0, \dots, K\}$, let us define by induction
$$
s^k_0=\inf\{j\geq  m, \; T_j=k, \;  Z_j=z^k_{\min}\}, \; s^k_{i+1}=\sup\{ j<s^k_i, \; T_j=k,\; Z_j=z^k_{\min}\}. 
$$
Note that the $(s^k_{i}-s^k_{i+1})_{i\geq 0}$ are i.i.d. and that, by definition, for any $j\in \{s^k_{i+1}, \dots, s^k_{i}-1\}$, one has $\sigma^k_j= s^k_i$. 
Therefore
$$
\sum_{r= s^k_{i+1}}^{s^k_i-1} (\sigma^k_r-r)^2 = \sum_{r= s^k_{i+1}}^{s^k_i-1} (s^k_i-r)^2 \leq C(s^k_i-s^k_{i+1})^3.
$$
As $s^k_0\geq m$ while $s^k_{n}\leq m-n$, this shows that 
$$
\sum_{i=1}^{n} (\sigma^k_{m-i}\wedge m-(m-i))^2\leq \sum_{i=m-n}^{m-1} (\sigma^k_{i}-i)^2\leq \sum_{j=0}^{n-1} \sum_{r=s^k_{j+1}}^{s^k_j-1} (\sigma^k_{r}-r)^2   \leq C \sum_{j=0}^{n-1} (s^k_j-s^k_{j+1})^3. 
$$
On the other hand,  
\begin{align*}
 & \sum_{i=1}^{n}(\sigma^0_{\sigma^k_{m-i}}\wedge m-\sigma^k_{m-i}\wedge m)^2 \le
 \sum_{i=1}^{n}(\sigma^0_{\sigma^k_{m-i}\wedge m}-\sigma^k_{m-i}\wedge m)^2 =
   \sum_{i=1}^{n}\sum_{j, \; s^k_j\le m} (\sigma^0_{s^k_j}-s^k_j)^2{\bf 1}_{s^k_j= \sigma^k_{m-i}} \\
 &  \leq  \sum_{j=1}^{n} (\sigma^0_{s^k_j}-s^k_j)^2(s^k_j-s^k_{j-1})
 \leq \sum_{j=1}^{n} (\sigma^0_{s^k_j}-s^k_j)^4+ \sum_{j=1}^{n} (s^k_j-s^k_{j-1})^2\\
&   \leq  \sum_{i=s^k_n}^{m-1} (\sigma^0_{i}-i)^4 + \sum_{j=1}^{n-1} (s^k_j-s^k_{j-1})^3.
 \end{align*}
 The first term in the right-hand side can be treated as above and we obtain: 
$$
[M]_n+  \lg M\rg_n\leq Cn+C \sum_{k=1}^K\left(\sum_{j=0}^{n-1} (s^k_j-s^k_{j+1})^3+\sum_{j=0}^{m-s^k_n-1}(s^0_j-s^0_{j+1})^5\right). 
$$
Therefore
\begin{align}\label{qlkjesdfngc}
& \P\left[[M]_n+  \lg M\rg_n> y\right] \leq \sum_{k=1}^K  \P\left[\sum_{j=0}^{n-1} (s^k_j-s^k_{j+1})^3+\sum_{j=0}^{m-s^k_n-1}(s^0_j-s^0_{j+1})^5> (KC)^{-1}(y-Cn)\right]\notag\\
 &\qquad \leq  \sum_{k=1}^K  \P\left[\sum_{j=0}^{n-1} (s^k_j-s^k_{j+1})^3> (2KC)^{-1}(y-Cn)\right]\notag \\
& \qquad  +
 \sum_{k=1}^K  \P\left[\sum_{j=0}^{n-1} (s^k_j-s^k_{j+1})^3\leq (2KC)^{-1}(y-Cn)\;, \; \sum_{j=0}^{m-s^k_n-1}(s^0_j-s^0_{j+1})^5> (2KC)^{-1}(y-Cn)\right]. 
\end{align}
Let $X_j^k= s^k_j-s^k_{j+1}$ (for $k\in \{0, \dots, K\}$). Then the $(X^k_j)_{j=0, \dots, n}$ are i.i.d. and $X_0^k$ follows a geometric law of parameter $p^k:= \P[Z_0= z^k_{\min}]$ which has exponential moments.  In particular $(X^k_0)^3$ satisfies Bernstein's condition: there exists $c^k>0$ such that, for any $k\in \{0, \dots, K\}$ and $p\geq 2$,  
$$
\E\left[ \left| |X^k_0|-\E\left[|X^k_0|\right]\right|^p\right] \leq \frac{ p! (c^k)^{p-2}}{2} v^k \qquad {\rm where}\; v^k:=Var(|X^k_0|^3). 
$$
From Bernstein's Theorem (Corollary 2.11 in \cite{BLM13}) we have 
$$
\P\left[  \sum_{j=0}^{n} (s^k_j-s^k_{j+1})^3 >  \E\left[ |X^k_0|^3\right]n +x \right] \leq \exp\left\{ -\frac{x^2}{2(nv^k +xc^k)}\right\}. 
$$
This allows to handle the first terms in the right-hand side of \eqref{qlkjesdfngc}. As for the second term, we note that, in the event $\{\sum_{j=0}^{n-1} (s^k_j-s^k_{j+1})^3\leq (2KC)^{-1}(y-Cn)\}$, we have by Hölder's inequality:
$$
s^k_n = s^k_0+ \sum_{i=0}^{n-1} (s^k_{i+1}-s^k_i) \geq m -n^{2/3} \left( \sum_{i=0}^{n-1} |X^k_i|^3\right)^{1/3} \geq m- C_1 n^{2/3}(y-Cn)^{1/3}.  
$$
Hence 
\begin{align*}
& \P\left[\sum_{j=0}^{n-1} (s^k_j-s^k_{j+1})^3\leq (2KC)^{-1}(y-Cn)\;, \; \sum_{j=0}^{m-s^k_n-1}(s^0_j-s^0_{j+1})^5> (2KC)^{-1}(y-Cn)\right] \\ 
& \qquad \leq \P\left[\sum_{j=0}^{[C_1 n^{2/3}(y-Cn)^{1/3}]+1}(s^0_j-s^0_{j+1})^5> (2KC)^{-1}(y-Cn)\right].
\end{align*}
We can again use Bernstein's inequality to handle this later term: we have, for some constant $c^0>0$ and $v^0:=  Var(|X^0_0|^5)$, 
$$
\P\left[  \sum_{j=0}^{n} (s^0_j-s^0_{j+1})^5 >  \E\left[ |X^k_0|^5\right]n +x \right] \leq \exp\left\{ -\frac{x^2}{2(nv^0 +xc^0)}\right\}. 
$$
So choosing $y= C_2n$ for a sufficiently large constant $C_2$ in \eqref{qlkjesdfngc}, we obtain that 
$$
\P\left[[M]_n+  \lg M\rg_n> C_2n\right]\leq 2K \exp\left\{ -\frac{n}{C_2}\right\}.
$$
\end{proof}

\begin{proof}[Proof of Theorem \ref{thm:concentrationTOT}]  Coming back to \eqref{concentrationBercuTouatiTOT} and using Lemma \ref{lem.estilgMrgTOT}, we  find 
$$
\P\left[ |M_n|\geq x\right] \leq 2 \exp\{ -x^2/(Cn)\}+  \exp\{ -\frac{n}{C}\}. 
$$
Recalling that $M_{\bar n}=\theta^m_e(t)-\E[\theta^m_e(t)]$ for $\bar n=Ct$, we obtain therefore (for $n=\bar n$, $x=\ep t$ and $\ep\in (0,1]$)
$$
\P\left[ \left|\theta^m_e(t)-\E\left[\theta^m_e(t)\right]\right|\geq \ep t\right] \leq C\exp\left\{ -\frac{\ep^2t}{C }\right\}.
$$
Then we use Lemma \ref{thetathetam} to get the concentration inequality for $\theta_e(t)$. 
\end{proof}

\subsection{A corrector outside  the junction}\label{subsec:correctoroutside}

In this part we build a random sequence $(W_i)_{i\in \Z}$ which plays the role of a corrector for large values of $|i|$. We first use $(W_i)_{i\in \Z}$ in this section to investigate the behavior of $U_{e,i}(t)$ for large $|i|$. The main role of the $(W_i)_{i\in \Z}$  will be however in the next section where the property of being a kind of corrector for large values of $|i|$ will be used in a crucial way. 

We recall that $e=(e^k)_{k=0,\dots,K}$ is such that $H^k(-1/e^k)=\min_pH^k(p)$. This implies in particular that $v^k_e=\bar V^k(e^k/\pi^k)<\bar v^k$ for all $k\in\{0,\dots,K\}$, where the $\bar v^k$ are defined in \eqref{def.barvk}. We define $(W_i)_{i\in \Z}$ as follows: we set $W^\omega_0=0$ and define $W_i^\omega$ for $i\leq -1$ by backward induction by setting:
$$
\tilde V^0_{Z_i^\omega}(W^\omega_{i+1}-W^\omega_i)= v_e^0\qquad {\rm for}\; i<0. 
$$
For $i\geq 1$, define $W^\omega_i$ by forward induction by setting, if  $\ell^{\omega,-1}_i :=\sup\{j <i, \; T^\omega_j=T^\omega_i\}$, 
$$
W_i^\omega= 0 \; {\rm if}\; i\geq1 \;{\rm and} \; \ell^{\omega,-1}_i <0
$$
and 
$$
\tilde V^k_{Z_{\ell^{\omega,-1}_i}^\omega} (W^\omega_{i}-W^\omega_{\ell^{\omega,-1}_i})=v_e^k, \qquad {\rm where} \; T^\omega_i=k
, \qquad {\rm if}\; i\geq 1\;{\rm and} \; \ell^{\omega,-1}_i \geq0.
$$
By the definition of $e_{\max}$, we have $0< W^\omega_{i+1}-W^\omega_i\leq e_{\max}$ if $i+1\leq 0$ while $0< W^\omega_{\ell_i}-W^\omega_i\leq e_{\max}$ if $\ell_i\geq 1$. We now collect several properties of the sequence $(W_i)$.

\begin{Lemma}\label{lem.boundWi-Wj} We have 
$$
|W_i-W_j|\leq e_{\max} |i-j|
$$
if $T_i=T_j$ or if $i\wedge j \leq 0$. 
\end{Lemma}

\begin{proof}  If $i\leq -1$, then by the definition of $W_i$ one has $0\leq W_{i+1}-W_i\leq e_{\max}$. So the claim holds if $i\leq 0$ and $j\leq 0$. Let us now assume that $j\geq 0$ and let us set $k=T_j$. If $\ell^{-1}_j< 0$, then $W_j=0$ and
$$
|W_j-W_{\ell^{-1}_j}|= |W_{\ell^{-1}_j}|\leq - e_{\max}\ell^{-1}_j\leq e_{\max} (j-\ell^{-1}_j).
$$
If $\ell^{-1}_j\geq 0$, then by the definition of $W_j$ one has $|W_j-W_{\ell^{-1}_j}|\leq e_{\max}$. By induction, this  implies that $|W_j-W_i|\leq e_{\max}|j-i|$ if $T_i=T_j$ and $i\geq 0$, $j\geq 0$. The claim for $i\geq 0$ and $j\leq 0$ follows easily. 
\end{proof}

\begin{Lemma}\label{lem.estiWi} There exists $C>0$ such that, for any $\ep\in (0,1]$,   
for any $i_0, i\in \Z$, one has 
$$
\P\left[ |W^{\tau_{i_0}\cdot}_{i-i_0}-e^0(i-i_0)|> \ep |i-i_0|\right] \leq C\exp\{-\ep^2|i-i_0|/C\}\qquad {\rm if}\; i\leq i_0- C\ep^{-1}
$$
and
$$
\P\left[ |W^{\tau_{i_0}\cdot}_{i-i_0}-e^{T^{\tau_{i_0}\cdot}_i}(i-i_0)|> \ep |i-i_0|\right] \leq C\exp\{-\ep^2|i-i_0|/C\}\qquad {\rm if}\; i\geq i_0+C\ep^{-1}.
$$
\end{Lemma} 

\begin{proof} Fix first $i_0=0$. For $i\leq 0$, the proof is a straightforward application of  Hoeffding's inequality (\cite[Theorem 2.8]{BLM13}) combined with the property that, by the definition of $\bar V^0$, $\E\left[ (\tilde V^0_{Z_i})^{-1}(v_e^0)\right]=e^0$. 

Let us now investigate the case $i\geq 0$. For $k\in \{1, \dots, K\}$, let $s^k_0= \inf\{i\geq 0, \; T_i=k\}$ and let us define by induction $s^k_{i+1}= \inf\{i\geq s^k_i+1, \; T_i=k\}$. Then 
$$
W_{s^k_i}= \sum_{j=0}^{i} (\tilde V^k_{Z_{s^k_j}})^{-1} (v_e^k),
$$
where the $(\tilde V^k_{Z_{s^k_j}})^{-1} (v_e^k)$ are i.i.d. with the same law as $(\tilde V^k_{Z_{0}})^{-1} (v_e^k)$ given $T_0=k$, which is bounded by $e_{\max}$. Recall that  $\E\left[ (\tilde V^k_{Z_{0}})^{-1} (v_e^k)\ |\ T_0=k\right]= e^k/\pi^k$. So, by Hoeffding's  inequality,  
$$
\P\left[ |W_{s^k_i}-ie^k/\pi^k | >x \right] \leq 2 \exp\{ -x^2/(2ie_{\max})\}.
$$
Since we also have by Bernstein's inequality (\cite[Corollary 2.11]{BLM13}): 
$$
\P\left[ |s^k_i - (\pi^k)^{-1}i |> \ep i\right]\leq 2 \exp\{-\ep^2i/C\}, 
$$
we can infer that, for any $i\geq C\ep^{-1}$ and setting $j_k= [\pi^k i]$,  
\begin{align*}
&\P\left[|W_i-e^{T_i}i|> \ep i\right]  \leq  \sum_{k=1}^K \P\left[|W_i-e^ki|> \ep i, \; T_i=k,\; | i-s^k_{j_k}| \leq \ep i / (2e_{\max})\right] + \P\left[ | i-s^k_{j_k}| > \ep i / (2e_{\max})\right] \\
&\qquad  \leq \sum_{k=1}^K \P\left[|W_{s^k_{j_k}}-e^kj_k/\pi^k|> \ep i/2-|j_k/\pi^k-i| \right]+ \P\left[ | j_k/\pi^k-s^k_{j_k}| > \ep i / (2e_{\max})-|j_k/\pi^k-i|  \right] \\
& \qquad \leq C\exp\{-\ep^2 i/C\}.
\end{align*}

 We now address the case $i_0\neq 0$. We note that $i\to W^{\tau_{i_0}\omega}_{i-i_0}$ can be built exactly as $W_i^\omega$ except that the origin is $i_0$ and $\omega$ is shifted by $\tau_{i_0}$. Thus we have in the same way 
$$
\P\left[|W^{\tau_{i_0}\cdot}_{i-i_0}-e^0(i-i_0)|> \ep |i-i_0|\right] \leq C\exp\{-\ep^2|i-i_0|/C\}\qquad \forall i\in \Z, \;i\leq i_0-C\ep^{-1} 
$$
and
$$
\P\left[|W^{\tau_{i_0}\cdot}_{i-i_0}-e^{T^{\tau_{i_0}\cdot}_i}(i-i_0)|> \ep |i-i_0|\right] \leq C\exp\{-\ep^2|i-i_0|/C\}\qquad \forall i\in \Z, \;i\geq i_0+C\ep^{-1} .
$$
\end{proof}

Fix $\bar C$ large and to be chosen below, $\ep>0$, $T\ge 1$ with $\ep T\geq C$, for $C$ large enough. We define the event 
\begin{align}\label{defEept}
E_{\ep,T}:=&  \Bigl\{ \sup_{|i|\leq \bar CT,\; |i_0|\leq C_\theta T} |W^{\tau_{i_0}\cdot}_{i-i_0} -e^0(i-i_0){\bf 1}_{i\leq i_0}- e^{T^{\tau_{i_0}\cdot}_i}(i-i_0){\bf 1}_{i>i_0} |\leq \ep T, \notag
\\ &
\qquad  \qquad \qquad   J_{[\bar CT/(2\alpha)] }([\bar C T])\geq \bar CT/ 2, \; J_{[\bar CT/(8\alpha)]}([-\bar CT/4])\geq -\bar CT/2  \Bigr\},
\end{align}
where $J_n(T)$ and $\alpha$ are defined in Lemma \ref{lem.Jn}.
By Lemma \ref{lem.Jn}, Lemma \ref{lem.boundWi-Wj} and Lemma \ref{lem.estiWi} we have 
\be\label{PEept}
\P\left[ E_{\ep,T}^c\right] \leq CT^2 \exp\{-\ep^2T /C\}, 
\ee
where $C= C(\bar C)$. We assume that $\bar C$ is so large that
\be\label{hypbarC}
C_\theta\leq \bar C /10\; {\rm and}\; \bar C\geq 16\beta\alpha/\ln(2),
\ee
where $\beta$ is given in the approximate finite speed of propagation (Lemma \ref{lem.finitespeedPSJunc}) and $C_\theta$ is defined in Lemma \ref{lem.estitheta}. 

\begin{Lemma}\label{lem.Wicorrector} If $\bar C$ is large enough we have,  for $T\geq \bar C\ep^{-1}$ and in $E_{\ep,T}$, 
\be\label{WivsUei+}
\left|e^{T_i}i+v_e^{T_i}s - U_{e,i}(s) \right|\leq 3\ep T \qquad \forall s\in [0,2T],   \; \forall i\in [2(\min_k e^k)^{-1}\ep T, \bar CT/2]\cap \Z
\ee
and 
\be\label{WivsUei-}
\left|e^0i+v_e^0s - U_{e,i}(s) \right|\leq 3\ep T  \qquad \forall s\in[ 0, 2T],\; \forall i\in [-\bar CT, -\bar CT/2)]\cap \Z.
\ee
\end{Lemma}

\begin{proof} Let us first note that for $i\in[2(\min_k e^k)^{-1}\ep T, \bar CT]\cap \Z$, since we are in $E_{\ep,T}$, $W_i\ge e^{T_i}i-\ep T\ge\ep T>0$. Then, the maps $s\to U_{e,i}(s)$ and $s\to W^1_i(s):=W_i-\ep T+v_e^{T_i} s$ solve \eqref{eq.SystJuncTOT} with an initial condition which satisfies, since we are in $E_{\ep,T}$, $0\leq W^1_i(0)\leq U_{e,i}(0)$. So, by the approximate finite speed of propagation (Lemma \ref{lem.finitespeedPSJunc}) we have 
$$
W^1_i(s)\leq U_{e,i}(s)+ 2^{-n}e^{\beta s} \qquad \forall s\geq 0, \;\forall n\in \N, \; \forall i\in [2(\min_k e^k)^{-1}\ep T, J_n([\bar C T])]\cap \Z. 
$$
Choosing $n= [\bar CT/(2\alpha)] $, we obtain, if $\bar C$ is large enough (depending on $\beta$ only) and since we are in $E_{\ep,T}$, 
$$
e^{T_i}i - 2\ep T+v_e^{T_i} s\leq  W^1_i(s)\leq U_{e,i}(s)+1 \qquad \forall s\in [0,2T],   \; \forall i\in [2(\min_k e^k)^{-1}\ep T, \bar CT/2]\cap \Z.
$$
Replacing $W^1$ by $W^2_i(s):= W_i+\ep T+v_e^k s$ gives the opposite inequality. Thus  \eqref{WivsUei+} holds. 

To obtain \eqref{WivsUei-}, we note that, for $i\in [-\bar CT, -\bar CT/2]\cap \Z$, the maps $s\to U_{e,i}(s)$ and $s\to W^1_i(s):=W_i-\ep T+v_e^0 s$ solve \eqref{eq.SystJuncTOT} on the time interval $[0, \bar C T/(4\|V\|_\infty)]$ with an initial condition which satisfies, since we are in $E_{\ep,T}$, $W^1_i(0)\leq U_{e,i}(0)$. So we have as above 
$$
W^1_i(s)\leq U_{e,i}(s)+ 2^{-n}e^{\beta s} \qquad \forall s\in[ 0, \bar C T/(4\|V\|_\infty)],\; \forall i\in \{-\bar CT, \dots, J_n([-\bar CT/4])\}. 
$$
We choose $n= [\bar CT/(8\alpha)]$ and get, for $\bar C$ large enough (depending on $\beta$ and $\|V\|_\infty$), since we are in $E_{\ep,T}$, 
$$
W^1_i(s)\leq U_{e,i}(s)+ 1 \qquad \forall s\in[ 0, 2T],\; \forall i\in [-\bar CT, \dots, -\bar CT/2)]\cap \Z. 
$$
Arguing as above we get \eqref{WivsUei-}. 
\end{proof}

\subsection{A superadditive quantity}

The aim of this section is to investigate the existence of a limit for  $ \theta_e(t)/t$ as $t\to +\infty$. For doing so, we introduce new notation. Fix $\tilde h>0$ such that 
\be\label{cond.tildev}
\tilde h< \min_{k\in\{0,\dots, K\}} \frac{v_e^k}{e^k}= \min_{k\in\{0,\dots, K\}} \frac{\bar V^k(e^k/\pi^k)}{e^k} = - \max_{k\in \{0, \dots, K\}} H^k(-1/e^k)=-A_0,
\ee
where $A_0$ is defined in \eqref{defA0}. 
Let us define
$$
\bar \theta_e(s)=\E\left[\theta_e(s)\right], \qquad \bar M_{e,\tilde h}(t) = \inf_{s\in [0,t]} \bar \theta_e(s)- \tilde h s.
$$
We note that the quantity $\bar M_{e,\tilde h}$ is nonpositive, nonincreasing in $t$ and in $\tilde h$. The main result of this section is the following: 

\begin{Theorem}\label{thm.kebarvartheta} The limit $k_{e,\tilde h}$ of $\bar M_{e,\tilde h}(t)/t$, as $t\to +\infty$ exists, is nonpositive and nonincreasing with respect to $\tilde h$. Let us set 
$$
k_e:= \inf_{0<\tilde h< -A_0} k_{e, \tilde h}= \lim_{\tilde h\to - A_0} k_{e, \tilde h}.
$$
If $k_{e}<0$, then the limit $\bar \vartheta_e$ of $\theta_e(t)/t$, as $t\to +\infty$, exists almost surely, is deterministic and satisfies $\bar \vartheta_e<  -A_0$.

If $k_{e}=0$, then the limit $\liminf_{t\to+\infty} \theta_e(t)/t$ is deterministic and is not smaller than $-A_0$. 
\end{Theorem} 

To prove the result, we are going to show that  $\bar M_{e,\tilde h}$ is almost superadditive (Lemma \ref{lem:superadditivityJuncTOT}), which implies that $\bar M_{e,\tilde h}(t)/t$ has a limit $k_{e,\tilde h}$ as $t\to +\infty$ (Lemma \ref{lem.limEthetatJuncTOT}). This, in turn, will show  the existence of a limit for $\bar \theta_e(t)/t$ if $k_{e,\tilde h}<0$ and thus, by the variance estimate, the a.s. limit of $\theta_e(t)/t$ (Lemma \ref{defbarthetaTOT}). 

The proof of the superadditivity of $\bar M_{e,\tilde h}$ is intricate and requires the introduction of several additional quantities. Let $\xi^\omega:\R\times \Z\to \R$ be a measurable map which is smooth, uniformly Lipschitz continuous and increasing in the $x$ variable, with inverse also uniformly Lipschitz continuous, and such that  
$$
\xi^\omega_e(x,i)= \left\{ \begin{array}{ll}
x/e^0 & {\rm if }\; x\leq -\min(R_0,e^0),\\ 
x/e^k & {\rm if}\; x\geq0, \; T_i=k, 
\end{array}\right.
\;{\rm and}\qquad \left|\xi^\omega_e(x,i)-\frac{x}{e^0}{\bf 1}_{x\leq 0}-\frac{x}{e^k}{\bf 1}_{\{x\geq 0, \ T^\omega_i=k\}}\right| \leq 1.
$$
Since the inverse of $\xi_e^\omega$ is uniformly Lipschitz continuous, we have in particular, if $x\ge y$
\be\label{eq:xi-inv}
\xi_e^{\omega}(x,i)-\xi_e^\omega(y,i)\ge C^{-1}(x-y).
\ee
For $0<s\leq T$, let
$$
M^\omega_{e,T} (s)=M^\omega_{e,\tilde h,T} (s)= \inf_{i\in \Z\cap [-\bar CT,\bar CT/2], \; \tau\in [0,s]} \xi_e^\omega(U^\omega_{e,i}(\tau),i)-i-\tilde h \tau. 
$$
We also set, for any  $i_0\in \Z\cap  [- C_\theta T,0]$, 
$$
\tilde M^{i_0,\omega}_{e,T} (s)=\tilde M^{i_0,\omega}_{e,\tilde h,T} (s)= \inf_{\tau\in [0,s], \; i\in \Z\cap  [-\bar CT,\bar CT/2]} \xi_e^\omega(U^\omega_{e,i}(\tau),i) -\xi_e^\omega(W^{\tau_{i_0}\omega}_{i-i_0}, i-i_0)-i_0-\tilde h \tau. 
$$
Note that $M^\omega_{e,T}(0)=0$ and that  $M^\omega_{e,T}$ is nonpositive. We will prove below that $M_{e,T}$ and $\tilde M_{e,T}$ are good approximations of $\bar M_{e,\tilde h}$.

Let us introduce the event 
\be\label{deftildeEept}
\tilde E_{\ep,T} = E_{\ep,T} \cap \left\{ \sup_{s\le 2T} \left| \theta_e(s)-\bar \theta_e(s)\right| \leq \ep T, \quad \sup_{i\leq 2e_{\max}^{-1}\ep T} \ell_i\leq \bar CT/2, \quad J_{[\bar C T/(16\alpha)]+1} ([\bar CT/2])\geq 0\right\},
\ee
where $E_{\ep, T}$ is defined in \eqref{defEept}. Recalling Lemma \ref{lem.Jn}, Theorem \ref{thm:concentrationTOT} and \eqref{PEept} we have 
\be\label{boundPtildeE}
\P\left[ \tilde E_{\ep, T}^c\right] \leq C T^2 \exp\{-\ep^2 T/C\}.
\ee

\begin{Lemma}\label{lem.easy1} In $\tilde E_{\ep,T}$ and for  $i_0\in \Z\cap  [-C_\theta T,0]$, we have, for $s\in [0,T]$,  
$$
M^\omega_{e,T}(s) \leq \bar M_{e,\tilde h}(s)+ \ep T, \qquad \left| M^\omega_{e,T}(s)-\tilde M^{i_0,\omega}_{e,T}(s)\right|\leq C\ep T,
$$
where $C$ depends only on the Lipschitz constant of $\xi=\xi^\omega_e(x,i)$ with respect to $x$. 
\end{Lemma}

\begin{proof} Let $0\leq \tau\leq s\leq T$ and choose $i= -\theta^\omega_e(\tau)$ in the definition of $M^\omega_{e,T}(s)$. By \eqref{hypbarC} we know that $i= -\theta^\omega_e(\tau)\in \Z\cap  [-\bar CT,0]$. Then, as we are in $\tilde E_{\ep,T}$, we have 
$$
M^\omega_{e,T}(s)\leq 0+ \theta^\omega_e(\tau) -\tilde h \tau \leq \bar \theta_e(\tau) -\tilde h \tau +\ep T. 
$$
Taking the infimum over $\tau\in [0,s]$ gives the first inequality. For the second one, let us recall that, since we are in $E_{\ep,T}$ defined in \eqref{defEept}, we have for any $i\in \Z\cap [-\bar CT, \bar CT]$ and $i_0\in \Z\cap   [-C_\theta T,0]$
$$
|W^{\tau_{i_0}\omega}_{i-i_0} - e^0(i-i_0){\bf 1}_{i\leq i_0}-e^{T_i^{\tau_{i_0}\omega}} (i-i_0){\bf 1}_{i>i_0}|\leq \ep T.
$$
As $x\to \xi^\omega_e(x,i)$ is uniformly Lipschitz continuous with Lipschitz constant $C$ and $T_i^{\tau_{i_0}\omega}=\tau^\omega_{i-i_0}$ we have
\begin{align*}
& \bigl|\xi_e^\omega(W^{\tau_{i_0}\omega}_{i-i_0},i-i_0) - (i-i_0)\bigr|\\ 
&\qquad \leq 
\bigl|\xi_e^\omega(W^{\tau_{i_0}\omega}_{i-i_0},i-i_0) - \xi_e^\omega(e^0(i-i_0){\bf 1}_{i\leq i_0}-e^{T_i^{\tau_{i_0}\omega}} (i-i_0){\bf 1}_{i>i_0}, i-i_0)\bigr|+1\\
&\qquad \leq  C\ep T+ 1.
\end{align*}
This implies the second inequality. 
\end{proof}

We now use in a crucial way the construction of $(W_i)$ to obtain the key property of $\tilde M_{e,T}$: 
\begin{Lemma}\label{lem.central}
Let $\ep, T$ be such that $\ep T\ge \bar C$ and fix  $C_1>0$, $\omega \in \tilde E_{\ep,T}$, $\bar s\in (0,T]$ and set $i_0=-\theta_e^\omega(\bar s)$. 
Assume that $\tilde M^{i_0,\omega}_{e,T}(\tilde s)<0$ and that $\frac{d}{ds} \tilde M^{i_0,\omega}_{e,T}(\tilde s)<0$ for some $\tilde s\in (0,T)$ with $|\tilde s-\bar s| \leq C_1\ep^{1/2} T$. Then there exists $i_1$ such that $(\tilde s,i_1)$ is a minimum point in the definition of $\tilde M^{i_0,\omega}_{e,T}(\tilde s)$ and there exists large  constants $c_1$ and $C$ (depending on $C_1$ but not on $\omega$, $\ep$, $T$, $\bar s$ or $\tilde s$) with the property that, if $\tilde s\geq c_1\ep T$, then we have $|U^\omega_{i_1}(\tilde s)|\leq C\ep^{1/2}T$, $|i_1-i_0|\leq C\ep^{1/2} T$ and  
\be\label{lznqejsdfxc}
U^\omega_{e,i}(\tilde s)\geq e^0(i-i_0){\bf 1}_{i\leq i_0}+e^{T^\omega_i}(i-i_0){\bf 1}_{i>i_0} -C\ep^{1/2} T\qquad \forall i\in  \Z\cap[-\bar CT,\bar CT/2] . 
\ee
\end{Lemma}

\begin{proof} As  $\frac{d}{ds} \tilde M^{i_0,\omega}_{e,T}(\tilde s)<0$ there exists $i_1\in \Z\cap  [-\bar CT,\bar CT/2] $ such that $(\tilde s,i_1)$ is a minimum point in the definition of $\tilde M^{i_0,\omega}_{e,T}(\tilde s)$. 

By the envelope theorem, we have
\be\label{eeenvelope}
0>\frac{d}{ds}\tilde M^{i_0,\omega}_{e,T}(\tilde s) = \partial_x \xi_e^\omega(U^\omega_{e,i_1}(\tilde s),i_1) \  V_{Z_{i_1}^\omega}(U^\omega_{e,i_1+1}(\tilde s)-U^\omega_{e,i_1}(\tilde s), U^\omega_{e,\ell^\omega_{i_1}}(\tilde s)-U^\omega_{e,i_1}(\tilde s), U^\omega_{e,i_1}(\tilde s))-\tilde h. 
\ee
On the other hand, by the optimality of $i_1$, we have 
\be\label{eeopti}
\xi_e^\omega(U^\omega_{e,i}(\tilde s),i) -\xi_e^\omega(U^\omega_{e,i_1}(\tilde s),i_1) 
\geq 
\xi_e^\omega(W^{\tau_{i_0}\omega}_{i-i_0}, i-i_0)
- \xi_e^\omega(W^{\tau_{i_0}\omega}_{i_1-i_0}, i_1-i_0)
\qquad \forall i\in  \Z\cap  [-\bar CT,\bar CT/2] . 
\ee
We first claim that $i_1\ge -\bar C T/2$. Indeed, since we are in $E_{\ep,T}$, inequality \eqref{WivsUei-} in Lemma \ref{lem.Wicorrector} implies that, for any $i\in [-\bar CT, -\bar CT/2]$ (which implies that $i\leq C_\theta t\leq i_0$, $U^\omega_{e,i}(\tilde s)\leq -R_0$ and $W^{\tau_{i_0}\omega}_{i-i_0}\leq -R_0$ if $\bar C$ is big enough),  
\begin{align*}
 \xi_e^\omega(U^\omega_{e,i}(\tilde s),i) -\xi_e^\omega(W^{\tau_{i_0}\omega}_{i-i_0}, i-i_0)-i_0-\tilde h \tau & \geq 
(e^0)^{-1} (e^0i+v^0_e\tilde s) -(e^0)^{-1}(e^0(i-i_0))-i_0-\tilde h \tilde s- 4\ep (e^0)^{-1}T\\ 
& \geq ((e^0)^{-1} v^0_e-\tilde h) \tilde s- 4\ep (e^0)^{-1} T >0
 \end{align*}
since $(e^0)^{-1}v_e^0>\tilde h$ and $\tilde s> c_1\ep T$ where $c_1$ is large enough. This shows that $i_1\geq -\bar CT/2$ because $\tilde M^{i_0,\omega}_{e,T}(\tilde s)<0$. 

In the same way, we have $i_1\leq 2(\min_ke^k)^{-1}\ep T$. Indeed, for $i\in [2(\min_k e^k)^{-1}\ep T, \bar CT/2]$, by \eqref{WivsUei+} in Lemma \ref{lem.Wicorrector} and for $k=T^\omega_i$, we have
\begin{align*}
 \xi_e^\omega(U^\omega_{e,i}(\tilde s),i) -\xi_e^\omega(W^{\tau_{i_0}\omega}_{i-i_0}, i-i_0)-i_0-\tilde h \tilde s & \geq 
 (e^k)^{-1}(e^k i+v^k_e \tilde s)- (e^k)^{-1} (e^k (i-i_0))-i_0- \tilde h \tilde s - 4\ep(e^k)^{-1} T\\ 
& > ( (e^k)^{-1}v^k_e- \tilde h) \tilde s - 4\ep (e^k)^{-1} T >0.
\end{align*}
Since $\tilde M^{i_0,\omega}_{e,T}(\tilde s)<0$, this shows that $i_1\leq 2(\min_ke^k)^{-1}\ep T$. Recalling the definition  of $\tilde E_{\ep,T}$, we also have therefore that $\ell_{i_1} \leq \bar CT/2$. 

We now prove that $|U^\omega_{i_1}(\tilde s)|\leq C\ep^{1/2}T$. By contradiction, assume first that $U^\omega_{e,i_1}(\tilde s)\geq C\ep^{1/2} T$, where $C$ is to be chosen below. Then, as $|\tilde s-\bar s|\leq C_1\ep^{1/2}T$ we obtain $U^\omega_{e,i_1}(\bar s)>0$ for $C$ large enough. Since $i_0=-\theta_e^\omega(\bar s)$, we get that $i_1\geq i_0+1$. Using successively that $U^\omega_{e,i_1}(\tilde s)> 0$ (for the first equality), \eqref{eeopti} and the fact that $\ell_{i_1} \leq \bar CT/2$ and that $\ell^\omega_{i_1}-i_0= \ell^{\tau_{i_0}\omega}_{i_1-i_0}$ (for the inequality), and  the definition of the $(W_i)$ (for the last equality), we have, for $k=T^\omega_{i_1}$,  
\begin{align*}
&\partial_x \xi_e^\omega(U^\omega_{e,i_1}(\tilde s),i_1) \   V_{Z_{i_1}^\omega}(U^\omega_{e,i_1+1}(\tilde s)-U^\omega_{e,i_1}(\tilde s), U^\omega_{e,\ell^\omega_{i_1}}(\tilde s)-U^\omega_{e,i_1}(\tilde s), U^\omega_{e,i_1}(\tilde s))\\
& \qquad =(e^k)^{-1} V^k_{Z^{\tau_{i_0}\omega}_{i_1-i_0}}( U^\omega_{e,\ell^\omega_{i_1}}(\tilde s)-U^\omega_{e,i_1}(\tilde s))\\
& \qquad  \geq
 (e^k)^{-1}  V^k_{Z^{\tau_{i_0}\omega}_{i_1-i_0}}( W^{\tau_{i_0}\omega}_{\ell^{\tau_{i_0}\omega}_{i_1-i_0}}-W^{\tau_{i_0}\omega}_{i_1-i_0})= (e^k)^{-1} v_e^k> \tilde h, 
\end{align*}
which contradicts \eqref{eeenvelope}. Assume now that $U^\omega_{e,i_1}(\tilde s)\leq (- C'\ep^{1/2} T)\wedge (-R_0)$. Then as before, we have $i_1+1\leq i_0$ if $C'$ is large enough and we get
\begin{align*}
&\partial_x \xi_e^\omega(U^\omega_{e,i_1}(\tilde s),i_1) \   V_{Z_{i_1}^\omega}(U^\omega_{e,i_1+1}(\tilde s)-U^\omega_{e,i_1}(\tilde s), U^\omega_{e,\ell^\omega_{i_1}}(\tilde s)-U^\omega_{e,i_1}(\tilde s), U^\omega_{e,i_1}(\tilde s))\\
&\qquad= (e^0)^{-1}\   V^0_{Z^\omega_{i_1}}( U^\omega_{e,i_1+1}(\tilde s)-U^\omega_{e,i_1}(\tilde s))
\\
& \qquad  
\geq 
(e^0)^{-1}V^0_{Z^{\tau_{i_0}\omega}_{i_1-i_0}}( W^{\tau_{i_0}\omega}_{i_1+1-i_0}-W^{\tau_{i_0}\omega}_{i_1-i_0}) = (e^0)^{-1}v_e^0> \tilde h.
\end{align*}
This gives again a contradiction with  \eqref{eeenvelope} and show that $|U^\omega_{i_1}(\tilde s)|\leq C\ep^{1/2}T$ for $C$ large enough.

We now claim that this inequality and the fact that $i_0=-\theta_e^\omega(\bar s)$ imply that $|i_0-i_1|\leq C\ep^{1/2}T$: indeed, let $s$ be such that $U^\omega_{i_1}(s)=0$ (if it exists, otherwise, we set $s=0$). Then,  by Lemma \ref{lem.delta}, we have 
$$
|U^\omega_{e,i_1}(\tilde s)|\ge |U^\omega_{e,i_1}(\tilde s)- U^\omega_{e,i_1}(s)|\geq \delta|\tilde s-s|, 
$$
so that $|\tilde s-s|\leq \delta^{-1} C'\ep^{1/2}T$. If $i_1\le 0$, by the definition of $s$, $\theta_e^\omega(s)=-i_1$, we get, recalling Lemma \ref{lem.estitheta}:
$$
|i_0-i_1|= |\theta_e^\omega(s) -\theta_e^\omega(\bar s)| \leq C_\theta (|s-\bar s| +1)\leq C_\theta (|s-\tilde s|+|\tilde s-\bar s|+1)\leq  C\ep^{1/2}T. 
$$
If $i_1\ge 0$, we get in the same way $|i_0|\le C\ep^{1/2}T$ and so $|i_0-i_1|\le C\ep^{1/2}T$.
By the choice of $\bar C$ in \eqref{hypbarC}, we have that $i_0\in [-\bar CT/10, 0]$. Thus, for $\ep$ small enough, we obtain also $i_1\in [-\bar CT, \bar CT/2]$.  Coming back to \eqref{eeopti} we obtain therefore, using the facts that $\xi$ is uniformly Lipschitz continuous,   $|U^\omega_{i_1}(\tilde s)|\leq C\ep^{1/2}T$ $|i_0-i_1|\leq C\ep^{1/2}T$ and  the fact that we are in $ E_{\ep,T}$,
$$
\xi_e^\omega(U^\omega_{e,i}(\tilde s),i) \geq 
\xi_e^\omega(W^{\tau_{i_0}\omega}_{i-i_0}, i-i_0) -C\ep^{1/2}T
\qquad \forall i\in  \Z\cap  [-\bar CT,\bar CT/2] . 
$$
Since the inverse of $\xi$ is increasing and uniformly Lipschitz continuous, we get 
$$
U^\omega_{e,i}(\tilde s) \geq 
W^{\tau_{i_0}\omega}_{i-i_0}  -C\ep^{1/2}T
\qquad \forall i\in  \Z\cap  [-\bar CT,\bar CT/2] . 
$$
Recalling that $\omega\in E_{\ep,T}$, we find \eqref{lznqejsdfxc}.  
\end{proof}

Next we show that we can bound from below $M^\omega_{e,T}(t)$ by $\bar M_{e,\tilde h}(t)$: 

\begin{Lemma}\label{lem.kjsnedf}Let $\ep, T$ be such that $\ep T\ge \bar C$. There exists a constant $C$ such that, in $\tilde E_{\ep,T}$ and for $t\in[0,T]$, we have 
$$
M^\omega_{e,T}(t) \geq \bar M_{e,\tilde h}(t)-C\ep^{1/2}T.
$$
\end{Lemma}

\begin{proof} 
Let 
$$
\bar s=\sup\Bigl\{s\in [0,t], \; M^\omega_{e,T}(s)\geq M^\omega_{e,T}(t)+4\ep^{1/2}(2T-s)\Bigr\}.
$$
If there is no such a $s\in [0,t]$, then $M^\omega_{e,T}(t)\geq -8\ep^{1/2}T$ since $M^\omega_{e,T}(0)= 0$. So we have  $M^\omega_{e,T}(t)\geq -C\ep^{1/2}T$ while $ \bar M_{e,\tilde h}(t)\leq 0$, which proves the result in this case. In the same way, if $\bar s\leq c_1\ep T$ (for $c_1$ to be chosen below), then, since $s\to M^\omega_{e,T}(s)$ is Lipschitz continuous and $M^\omega_{e,T}(0)=0$, 
$M^\omega_{e,T}(\bar s)\geq -C\bar s \geq -C c_1\ep T$ and by the definition of $\bar s$ we have $M^\omega_{e,T}(\bar s)= M^\omega_{e,T}(t)+4\ep^{1/2}(2T-\bar s)$. So $M^\omega_{e,T}(t)\geq -C\ep^{1/2}T$ while $ \bar M_{e,\tilde h}(t)\leq 0$, which proves the result also in this case.   \\

Assume now that $\bar s$ exists and satisfies $\bar s> c_1\ep T$. We also suppose in a first step that $\bar s+ C'\ep^{1/2} T\leq t$ with $C'=C/2$ where $C$ is given by Lemma \ref{lem.easy1}. Note that $M^\omega_{e,T}(\bar s)= M^\omega_{e,T}(t)+4\ep^{1/2}(2T-\bar s)$. Let $i_0=- \theta^\omega_e(\bar s)$. Then, recalling Lemma \ref{lem.easy1} and the definition of $\bar s$, we have 
\begin{align*}
\tilde M^{i_0,\omega}_{e,T}(\bar s) & \geq M^\omega_{e,T}(\bar s) - C\ep T = M^\omega_{e,T}(t) +4\ep^{1/2}(2T-\bar s)-C\ep T  \\
& > M^\omega_{e,T}(\bar s+C'\ep^{1/2}T) -4\ep^{1/2}(2T-\bar s-C'\ep^{1/2}T)+4\ep^{1/2}(2T-\bar  s)-C\ep T \\ 
&=M^\omega_{e,T}(\bar s+\ep^{1/2}T)+C\ep T \\ 
& > \tilde M^{i_0,\omega}_{e,T}(\bar s+\ep^{1/2}T). 
\end{align*}
So there exists $\tilde s\in [\bar s, \bar s+\ep^{1/2}T]$ such that $(d/ds)\tilde M^{i_0,\omega}_{e,T}(\tilde s) <0$. Note that $|\tilde s-\bar s|\leq C\ep^{1/2}T$. Let $(\tilde s,i_1)$ (where $i_1\in \Z\cap  [-\bar CT,\bar CT/2] $) be a minimum point in the definition of $\tilde M^{i_0,\omega}_{e,T}(\tilde s)$. From Lemma \ref{lem.central} we know that $|U^\omega_{e,i_1}(\tilde s)|\leq C\ep^{1/2}T$ and $|i_1-i_0|\leq C\ep^{1/2}T$ for a large constant $C$. Then we have (since we are in $\tilde E_{\ep,T}$ and by Lemma \ref{lem.easy1}), 
\begin{align*}
 M^\omega_{e,T}(t) & \geq M^\omega_{e,T}(\bar s) -C\ep^{1/2} T  \geq \tilde M^{i_0,\omega}_{e,T}(\tilde s) -C\ep^{1/2} T \\
 & =
\xi^\omega_e(U^\omega_{e,i_1}(\tilde s),i_1) - \xi^\omega_e(W^{\tau_{i_0}\omega}_{i_1-i_0},i_1-i_0) -i_0 -\tilde h \tilde s -C\ep^{1/2} T   \\
&   \geq -C\ep^{1/2} T -i_0 -\tilde h\bar s  
\geq -C\ep^{1/2} T + \theta^\omega_e(\bar s)  -\tilde h\bar s \\
&  \geq  -C\ep^{1/2} T + \bar\theta_e(\bar s) -\tilde h\bar  s \geq -C\ep^{1/2} T +\bar M_{e,\tilde h}(t).
\end{align*}
This proves the result. \\

The case where $\bar s$ satisfies $\bar s> c_1\ep T$ and $\bar s+ C'\ep^{1/2} T> t$ can be treated in a similar way, by noticing in a first step that 
$$
\tilde M^{i_0,\omega}_{e,T}(\bar s)> \tilde M^{i_0,\omega}_{e,T}(t) 
$$
and concluding as in the previous case that there exists a minimizer $i_1$ for $\tilde M^{i_0,\omega}_{e,T}(\tilde s)$ (for some $\tilde s\in [\bar s,t]$ and thus such that $|\tilde s-\bar s|\leq C't\ep^{1/2} T$ since $0\leq t- \bar s\leq  C'\ep^{1/2}T$) such that $|U^\omega_{i_1}(\tilde s)|\leq C\ep^{1/2}T$ and $|i_0-i_1|\leq C\ep^{1/2}T$. We can then complete the proof as above.
\end{proof}

In the next step we show that $\bar M_{e,\tilde h}(t)$ is almost superadditive. 

\begin{Lemma} \label{lem:superadditivityJuncTOT}
There is a constant $C$ such that, for any $t\geq C$ and any $h\in [0,t]$,  
\be\label{superadditivityJuncTOT}
 \bar M_{e,\tilde h}(t+h)\geq \bar M_{e,\tilde h}(h)+ \bar M_{e,\tilde h}(t)- C(1+ (\ln(t))^{1/8}t^{7/8}) .
\ee
\end{Lemma} 

\begin{proof} 
Fix $0\leq h\leq t\leq T:=2t+1$ and $\ep>0$ small enough so that $t\geq c_1\ep T$, where $c_1$ is as in Lemma \ref{lem.central}. We also assume that $t$ is large enough so that $\ep T\ge \bar C$. Let consider 
$$
\bar t = \inf\{ s\in [t, t+h], \; \bar M_{e,\tilde h}(s)< \bar M_{e,\tilde h}(t)-\ep^{1/4}s\}
$$
if there is some $s\in [t,t+h]$ such that $ \bar M_{e,\tilde h}(s)< \bar M_{e,\tilde h}(t)-\ep^{1/4}s$ and set $\bar t= t+h$ otherwise. 
If $\bar t= t+h$, since $\bar M_{e,\tilde h}(h)\le 0$, we have 
$$
\bar M_{e,\tilde h}(t+h)\geq \bar M_{e,\tilde h}(t)-\ep^{1/4}(t+h)\geq \bar M_{e,\tilde h}(h)+ \bar M_{e,\tilde h}(t)-\ep^{1/4}(t+h).
$$
Our aim is to show that a similar inequality holds  if $\bar t< t+h$. \\

Let us first consider the case where $\bar t< t+h$ and $\bar t- \gamma T >  t$, where $\gamma=C'\ep^{1/4}$ for some large constant $C'$ to be chosen below.  
Then $\bar M_{e,\tilde h}(\bar t)= \bar M_{e,\tilde h}(t)-\ep^{1/4}\bar t$. In $\tilde E_{\ep,T}$, we have by Lemma \ref{lem.easy1} (for the first inequality) and Lemma \ref{lem.kjsnedf} (for the second one)
\be\label{kqudfcl}
M^\omega_{e,T}(\bar t) \leq \bar M_{e,\tilde h}(\bar t) + \ep T= \bar M_{e,\tilde h}(t)-\ep^{1/4}\bar t+ \ep T \leq M^\omega_{e,T}(t)+ C\ep^{1/2} T -\ep^{1/4}\bar t.
\ee
As $\bar t- \gamma T \in(t,t+h)$,  we have, by the definition of $\bar t$,  
$$
\bar M_{e,\tilde h}(\bar t-\gamma T) \geq  \bar M_{e,\tilde h}(t)-\ep^{1/4}(\bar t-\gamma T),
$$
so that 
\begin{align*}
M^\omega_e(\bar t-\gamma T) &  \geq \bar M_{e,\tilde h}(\bar t-\gamma T) -C\ep^{1/2} T\geq  \bar M_{e,\tilde h}(t)-\ep^{1/4}(\bar t-\gamma T)-C\ep^{1/2} T\\
&  \geq 
M^\omega_{e,T}(t)-\ep^{1/4}\bar t+\ep^{1/4}\gamma T-C\ep^{1/2} T  \\
&\geq  M_{e,T}^\omega(\bar t)- C\ep^{1/2} T+\ep^{1/4}\gamma T \qquad \mbox{\rm (by \eqref{kqudfcl})}\\
&\geq   M^\omega_{e,T}(\bar t) + C^{-1}\ep^{1/4} T, 
\end{align*}  
if we choose  $\gamma = C'\ep^{1/4}$ for $C'$ large enough and independent of $t$ and $T$.  Let us set $i_0=- \theta^\omega_e(\bar t)$. For $\ep$ small enough, the inequality above implies (by Lemma \ref{lem.easy1}) that 
\begin{align*}
\tilde M^{i_0,\omega}_{e,T}(\bar t-\gamma t)   >  \tilde M^{i_0,\omega}_{e,T}(\bar t) . 
\end{align*}  
So there exists $\tilde t\in [\bar t-\gamma t,\bar t]$ such that $\tilde M^{i_0,\omega}_{e,T}(\tilde t)<0$ and $(d/dt) \tilde M^{i_0,\omega}_{e,T}(\tilde t)<0$.  Note that $|\tilde t-\bar t|\leq \gamma t=C'\ep^{1/4}t$ and that $\tilde t\geq t\geq c_1\ep T$. According to Lemma \ref{lem.central}, there exists $i_1$ such that $(\tilde t,i_1)$ is a minimizer of $\tilde M^{i_0,\omega}_{e,T}(\tilde t)$, and $i_1$ satisfies $|U^\omega_{e,i_1}(\tilde t)|\leq C\ep^{1/4}T$,  $|i_1-i_0|\leq C\ep^{1/4}T$ and
$$
U^\omega_{e,i}(\tilde t)\geq e^0(i-i_0){\bf 1}_{i\leq i_0}+e^{T^\omega_i}(i-i_0){\bf 1}_{i>i_0} -C\ep^{1/4} T\qquad \forall i\in  \Z\cap  [-\bar CT,\bar CT/2] . 
$$
Let us set $j_0:= -\bar \theta_e(\bar t)+[C''\ep^{1/4}T]$, where $C''$ is a large constant.  Using the definition of $\xi_e$ (for the first line and the last line), \eqref{eq:xi-inv} (for the second line) and the fact that we are in $\tilde E_{\ep,T}$ (for the third line), we have, if $C''$ is large enough and for any   $i\in  \Z\cap  [-\bar CT,\bar CT/2]$,
\begin{align*}
U^\omega_{e,i}(\tilde t)\geq & \xi_e^\omega(i-i_0,i)-C\ep^{1/4}T-1\\
\ge &  \xi_e^\omega(i-j_0,i) +C^{-1}(j_0-i_0)-C\ep^{1/4}T\\
\ge &  \xi_e^\omega(i-j_0,i) +C^{-1}(\theta_e^\omega(\bar t)-\bar \theta_e(\bar t)+C''\ep^{1/4}T-1)-C\ep^{1/4}T\\
\ge & e^0(i-j_0){\bf 1}_{i\leq j_0}+e^{T^\omega_i}(i-j_0){\bf 1}_{i>j_0}. 
\end{align*}
As the solution starting from $e^0(i-j_0){\bf 1}_{i\leq j_0}+e^{T^\omega_i}(i-j_0){\bf 1}_{i>j_0}$ is $U^{\tau_{j_0}\omega}_{e,i-j_0}$, the approximate finite speed of propagation then implies for all $n\ge 1$
$$
U^\omega_{e,i}(\tilde t+s)\geq U^{\tau_{j_0}\omega}_{e,i-j_0}(s) -2^{-n} e^{\beta s} \qquad \forall s\geq 0, \; \forall i\in  \Z\cap  [-\bar CT,J_n([\bar CT/2])]. 
$$
Choosing $n= [\bar CT/(16\alpha)]+1$ (with $\alpha$  as in Lemma \ref{lem.Jn}), we obtain by the choice of $\bar C$ in \eqref{hypbarC} (which ensures that $-n\ln(2) +\beta T\leq 0$) since we are in $\tilde E_{\ep, T}$ (where $J_n([\bar CT/2])\geq0$): 
$$
U^\omega_{e,i}(\tilde t+s)\geq U^{\tau_{j_0}\omega}_{e,i-j_0}(s) -1\geq U^{\tau_{j_0}\omega}_{e,i-j_0}(s-\delta^{-1}) \qquad \forall s\in [\delta^{-1},T], \; \forall i\in  \Z\cap  [-\bar CT,0],
$$
where $\delta$ is given by Lemma \ref{lem.delta}. Hence, in $\tilde E_{\ep,T}$,  and for $s\in[\delta^{-1},h]$: 
$$
\theta^\omega_e(\tilde t+s) \geq  \theta^{\tau_{j_0}\omega}_e(s-\delta^{-1}) - j_0.
$$
As $|\tilde t-\bar t|\leq \gamma t=C'\ep^{1/4}t$ and by the definition of $j_0$ we get, in view of Lemma \ref{lem.estitheta}:
$$
\theta^\omega_e(\bar t+s) \geq  \theta^{\tau_{j_0}\omega}_e(s) +\bar \theta_e(\bar t) -C\ep^{1/4}T\qquad \forall s\in [0,h].
$$
Recalling the bounds on $\theta_e$ and on $\P[E_{\ep,T}^c]$ in \eqref{boundPtildeE} we obtain by taking expectation and for $ s\in [0,h]$: 
$$
\bar \theta_e(\bar t+s) \geq \bar \theta_e(s) +\bar\theta_e(\bar t) -CT\ep^{1/4}-CT \P[E_{\ep,T}^c]\geq  \bar \theta_e(s) +\bar\theta_e(\bar t) -C(\ep^{1/4}T +T^3\exp\{-\ep^2T/C\}).
$$
Thus, using that $t-\bar t\le 0$, 
\begin{align*}
\inf_{s\in [0,h+t-\bar t]} \bar \theta_e(\bar t+s)-\tilde h(\bar t+s) & \geq \bar M_{e,\tilde h}(h+t-\bar t)+\bar\theta_e(\bar  t) -\tilde h \bar  t-C(\ep^{1/4}T +T^3\exp\{-\ep^2T/C\})\\
& \geq \bar M_{e,\tilde h}(h)+\bar M_{e,\tilde h}(\bar t)-C(\ep^{1/4}T +T^3\exp\{-\ep^2T/C\})\\
& \geq \bar M_{e,\tilde h}(h)+\bar M_{e,\tilde h}( t)-C(\ep^{1/4}T +T^3\exp\{-\ep^2T/C\}),
\end{align*}
where the last inequality comes from the fact that $\bar M_{e,\tilde h}(\bar t)= \bar M_{e,\tilde h}(t)-\ep^{1/4}\bar t$. On the other hand, picking $\omega\in \tilde E_{\ep,T}$ and using successively Lemma \ref{lem.easy1}, inequality  $|\tilde t-\bar t|\leq C'\ep^{1/4}t$, the definition of $\tilde t$ and the fact that $| U^\omega_{e,i_1}(\tilde t)|\leq C\ep^{1/4}T$ and that $|i_1-i_0|\leq C\ep^{1/4}T$ with $i_0=- \theta^\omega_e(\bar t)$, we obtain
\begin{align*}
\bar M_{e,\tilde h}(\bar t) &\geq M^\omega_{e, T}(\bar t) -\ep T \geq \tilde M^{i_0,\omega}_{e,T}(\bar t)-C\ep T   \geq \tilde M^{i_0,\omega}_{e,T}(\tilde t)-C\ep^{1/4} T\\  
& = \xi_e^\omega(U^\omega_{e,i_1}(\tilde t),i_1) -\xi_e^\omega(W^{\tau_{i_0}\omega}_{i_1-i_0},i_1-i_0)-i_0-\tilde h \tilde t -C\ep^{1/4} T\\ 
&\geq  -i_1-C\ep T  -\tilde h \bar t- C\ep^{1/4}T \geq \theta^\omega_e(\bar t)-\tilde h \bar t-C\ep^{1/4}T \geq \bar \theta_e(\bar t)-\tilde h \bar t- C\ep^{1/4}T.  
\end{align*}
Therefore 
$$
\bar M_{e,\tilde h}(t+h)= \inf\{\bar M_{e,\tilde h}(\bar t) \ ,\  \inf_{s\in [0,h+t-\bar t]} \bar \theta_e(\bar t+s)-\tilde h(\bar t+s)\}\geq  \inf_{s\in [0,h+t-\bar t]} \bar \theta_e(\bar t+s)-\tilde h(\bar t+s)- C\ep^{1/4}T. 
$$
So we have obtained the following inequality: 
\be\label{ineqsuraddiJuncTOT}
\bar M_{e,\tilde h}(t+h) \geq \bar M_{e,\tilde h}(h)+\bar M_{e,\tilde h}(t) -C(\ep^{1/4}T +T^3\exp\{-\ep^2T/C\}).
\ee

In order to handle the case where $\bar t< t+h$ and $\bar t- \gamma T \leq  t$, we note that, by the choice of $T$ we have $\bar t\ge t\geq T/3$. Then, for $\ep>0$  small enough, we have in $\tilde E_{\ep,T}$:
\begin{align*}
M^\omega_{e,T}(\bar t) &  \leq \bar M_{e,\tilde h}(\bar t) +C\ep T=  \bar M_{e,\tilde h}(t)-\ep^{1/4}\bar t+C\ep T\\
&  \leq M^\omega_{e,T}(t)-\ep^{1/4}T/3+C\ep^{1/2} T < M^\omega_{e,T}(t). 
\end{align*}  
Thus we can find $\tilde t\in [t, \bar t]$ such that $\tilde M^{i_0,\omega}_{e,T}(\tilde t)<0$ and $(d/dt) \tilde M^{i_0,\omega}_{e,T}(\tilde t)<0$. Note also that $|\bar t-\tilde t|\leq C'\ep^{1/4}T$ because $\bar t- \gamma T \leq  t$. Then we can complete the proof as in the previous case. \\

We now know that  \eqref{ineqsuraddiJuncTOT}  always holds. If we choose $\ep = (3C\ln(T)/T)^{1/2}$ (where $C$ is the constant in \eqref{ineqsuraddiJuncTOT} and where this choice is possible since then $t\geq \bar C\ep^{-1}$ for $\ep$ small), then, as $T=2t+1$, \eqref{ineqsuraddiJuncTOT} becomes 
$$
\bar M_{e,\tilde h}(t+h) \geq \bar M_{e,\tilde h}(h)+ \bar M_{e,\tilde h}(t)- C(1+ (\ln(t))^{1/8}t^{7/8}) .
 $$
 This holds for any $0\leq h\leq t$ with $t\geq C$ for some large constant $C$. 
\end{proof}

The main consequence of the previous lemma is the following: 

\begin{Lemma}\label{lem.limEthetatJuncTOT} The limit $k_{e,\tilde h}:= \lim_{t\to +\infty} \bar M_{e,\tilde h}(t)/t$ exists and is nonpositive as $t\to+\infty$. If $k_{e,\tilde h}<0$, then $\E[\theta_e(t)]/t$ has a limit as $t\to +\infty$ given by $(k_{e,\tilde h}+\tilde h)$. If $k_{e,\tilde h}=0$, then 
$\liminf_{t\to +\infty} \frac{\E[\theta_e(t)] }{t} \geq \tilde h. $
\end{Lemma}

\begin{proof} As $\bar M_{e,\tilde h}(t)$ satisfies the almost superadditivity property \eqref{superadditivityJuncTOT}, the limit $k_{e,\tilde h}:= \lim_{t\to +\infty} \bar M_{e,\tilde h}(t)/t$ exists.  If $k_{e,\tilde h}=0$,  then, as, by the definition of $\bar M_{e,\tilde h}(t)$, we have $\bar M_{e,\tilde h}(t) \leq \bar \theta_e(t) -\tilde ht$ and we obtain 
$$
\liminf_{t\to +\infty} \bar \theta_e(t)/t\geq \tilde h.
$$
Let us now assume that $k_{e,\tilde h}<0$. By the definition of $\bar M_{e,\tilde h}(t)$  we know that
$$
\liminf_{t\to+\infty} \frac{\bar \theta_e(t)}{t} \geq \liminf_{t\to+\infty}\frac{\bar M_{e,\tilde h}(t)+\tilde ht}{t} = k_{e,\tilde h}+\tilde h.
$$
To prove the opposite inequality, let $\ep\in (0,|k_{e,\tilde h}|/2)$ and $T>0$ be such that $|\bar M_{e,\tilde h}(t)/t-k_{e,\tilde h}|<\ep$ for any $t\geq T$. 
Fix $t\geq T$ and let us define $\bar t = \sup\{s\geq t,\; \bar M_{e,\tilde h}(t)=\bar M_{e,\tilde h}(s)\}\in [t, +\infty]$. Then, for $ s\in [t, \bar t)$, we have, since $s\geq T$,  
$$
-\ep \leq  \frac{\bar M_{e,\tilde h}(t)}{t}-k_{e,\tilde h} = \frac{\bar M_{e,\tilde h}(s)}{t}-k_{e,\tilde h} = \frac{s}{t} \frac{\bar M_{e,\tilde h}(s)}{s} -k_{e,\tilde h} \leq \frac{s}{t} (k_{e,\tilde h}+\ep) -k_{e,\tilde h}.
$$
So $s\leq t(k_{e,\tilde h}-\ep)/(k_{e,\tilde h}+\ep)\leq (1+C\ep)t$. In particular, $\bar t$ is finite and satisfies $\bar t\leq t(1+C\ep)$. Note that, by the definitions of $\bar t$ and of $\bar M_{e,\tilde h}(t)$, we have $\bar \theta_e(\bar t)-\tilde h\bar t= \bar M_{e,\tilde h}(\bar t)$. Therefore, as $\bar \theta_e$ is nonnegative and nondecreasing,  
$$
\frac{\bar \theta_e(t)}{t} \leq \frac{\bar \theta_e(\bar t)}{t} \leq (1+C\ep) \frac{\bar \theta_e(\bar t)}{\bar t} =  (1+C\ep) \frac{\bar M_{e,\tilde h}(\bar t)+\tilde h\bar t}{\bar t}\leq (1+C\ep)(k_{e,\tilde h}+\ep+\tilde h).
$$
So 
$$
\limsup_{t\to+\infty} \frac{\bar \theta_e(t)}{t} \leq (1+C\ep)(k_{e,\tilde h}+\ep+\tilde h), 
$$
which proves that $\bar \theta_e(t)/t$ converges to $k_{e,\tilde h}+\tilde h$ since $\ep$ is arbitrary. 
\end{proof}

As a consequence, we have

\begin{Lemma} \label{defbarthetaTOT} Let $\tilde h<- A_0$. Assume that $k_{e,\tilde h}<0$, where $k_{e,\tilde h}$ is defined in Lemma \ref{lem.limEthetatJuncTOT}. Then the limit $\bar \vartheta_e$ of $\theta_e(t)/t$ exists a.s. as $t\to +\infty$ and satisfies $\bar \vartheta_e <  -A_0$.

If $k_{e,\tilde h}=0$ for any  $\tilde h<-A_0$, then  $\liminf_{t\to+\infty} \theta_e(t)/t$ is deterministic and is not smaller than $-A_0$.  
\end{Lemma}

\begin{proof} Assume that $k_{e,\tilde h}<0$.  Let us first check that, a.s., $\theta_e(t)/t$  converges to the limit $\bar \vartheta_e$ of $\bar \theta_e(t)/t=\E[\theta_e(t)]/t$ given by Lemma \ref{lem.limEthetatJuncTOT} and which satisfies $
\bar \vartheta_e <  -A_0$.  This is a classical consequence of the variance estimate in Theorem \ref{thm:concentrationTOT}. Fix $\ep>0$ and let $N\in \N$ be such that $|\bar \theta_e(n)/n-\bar \vartheta_e|\leq \ep$ for any $n\in \N$, $n\geq N$. By the variance estimate, we have 
$$
\P\left[ \left| \frac{\theta_e(n)}{n}-\bar \vartheta_e\right|> 2\ep\right] \leq 
\P\left[ \left| \frac{\theta_e(n)}{n}-\frac{\bar \theta_e(n)}{n}\right|> \ep\right] \leq 2\exp\{ -\ep^2n/C\},
$$
where the right-hand side is summable. So by the Borel-Cantelli Lemma we have, a.s. 
$$
\bar \vartheta_e-2\ep \leq \liminf_n \frac{\theta_e(n)}{n}\leq \limsup_n \frac{\theta_e(n)}{n}\leq \bar \vartheta_e+2\ep.
$$
As $\ep$ is arbitrary, this implies the a.s. convergence of $(\theta_e(n)/n)$ to $\bar \vartheta_e$. The convergence of $(\theta_e(t)/t)$ to $\bar \vartheta_e$ as the continuous variable $t$ tends to infinity comes directly from the  regularity  in time of $\theta_e$ (Lemma \ref{lem.estitheta}). The proof in the case $k_{e,\tilde h}=0$ goes exactly along the same lines. 
\end{proof}

%

\begin{proof}[Proof of Theorem \ref{thm.kebarvartheta}] It is a straightforward application of the previous lemmas. 
\end{proof}

\section{Definition  of the flux limiter and homogenization}\label{section.flux}
We recall that $e=(e^k)_{k=0, \dots, K}$ is such that $H^k(-1/e^k)=\min_pH^k(p)$.
The aim of this part is to define the flux limiter $\bar A$, building on the construction of $\bar \vartheta_e$ in the previous section. For this we introduce new notation. Recall that $A_0=  \max_{k\in\{0, \dots, K\}}  \min_{p\in \R} H^k(p)$.  Given $A\in [ A_0, 0)$ and $k\in \{0, \dots,K\}$, we denote by $p^{k,+}_A$ (respectively $p^{k,-}_A$) the largest (resp. the smallest) solution to $H^k(p)=A$ and set 
 \be\label{defphiJuncTOT}
 \phi_A(x,k)= \left\{\begin{array}{ll}
 p^{0,-}_Ax & {\rm if}\; x\leq 0, \; k\in\{0,\dots,K\}\\
 p^{k,+}_Ax & {\rm if }\; x\geq 0, \; k\in \{1, \dots, K\}
 \end{array}\right.
\ee
and 
 \be\label{defpsiJuncTOT}
\psi_A(y,k)= \phi^{-1}_A(-y,k)= \left\{\begin{array}{ll}
 (-p^{0,-}_A)^{-1}y & {\rm if}\; y\leq 0, \; k\in\{0,\dots,K\}\\
 (-p^{k,+}_A)^{-1}y & {\rm if }\; y\geq 0, \; k\in \{1, \dots, K\}. 
 \end{array}\right.
\ee
We note for later use that if $A_1<A_2$, then $\phi_{A_1}<\phi_{A_2}$ and $\psi_{A_1}<\psi_{A_2}$ in $\R\backslash\{0\}$. 

We define the flux limiter $\bar A$ as  
$$
\bar A := \left\{\begin{array}{ll}
- \bar \vartheta_e &  \mbox{\rm if}\; k_e<0\\
A_0&{\rm otherwise}
\end{array}
\right.
$$
with $k_e$ and $\bar \vartheta_e$ defined by Theorem \ref{thm.kebarvartheta}.  Note that if $k_e<0$, then, by Theorem \ref{thm.kebarvartheta}, 
$$\bar A=- \bar \vartheta_e> \max_{k\in\{0,\dots,K\}}H^k(-1/e^k)=A_0.$$

\subsection{The limit of $N_e$ and $U_e$}

We recall that $U_e$ is the solution of \eqref{eq.SystJuncTOT} with initial condition $U^{\omega}_{e,i}(0)=e^0i{\bf 1}_{i\leq 0}+ e^{T^\omega_i}i{\bf 1}_{i\geq 0}$. Let  us define, for $k\in \{1, \dots, K\}$ and $(x,t)\in \R\times [0,+\infty)$,  
\be\label{def.NJuncTOT}
N_e^{\omega}(x,k,t)=  \sum_{i\in \Z, \ i\leq 0, \  T_i=k}  \delta_{U^\omega_{e,i}(t)}( (x,+\infty))- \sum_{i\in \Z,\ i>0, \ T_i=k} \delta_{U^\omega_{e,i}(t)}((-\infty,x]) 
\ee
and, for  $(x,t)\in (-\infty, 0]\times [0,+\infty)$,  
$$
N_e^\omega(x,0,t)= \sum_{k=1}^K N_e^{k}(x,t)= \sum_{i\in \Z, \ i\leq 0}  \delta_{U^\omega_{e,i}(t)}( (x,+\infty)).
$$
We set 
$$
n^{\ep,\omega}_e(x,k,t)= \ep N^{\omega}_e(x/\ep,k,t/\ep)
$$
and 
$$
\nu^{\omega,\ep}_e(x,k,t)=\left\{\begin{array}{ll}
n^{\omega,\ep}_e(x,0,t) & {\rm if}\; k=0\; {\rm and} \;x\leq 0,\\
(\pi^k)^{-1}n^{\omega,\ep}_e(x,k,t) & {\rm if} \; k\in \{1, \dots, K\}\; {\rm and}\;  x\in \R. 
\end{array}\right.
$$

\begin{Lemma}\label{lem.limNt=0TOT}
There is a set $\Omega_0$ of full probability such that, for all $\omega\in \Omega_0$, 
$$
\lim_{\ep\to0} \nu^{\omega,\ep}_e(x,k,0)= \left\{\begin{array}{ll}
-x/e^0 & {\rm if}\;  (x,k)\in (-\infty,0]\times \{0, \dots, K\},\\ 
-x/e^k  & {\rm if}\;  (x,k)\in [0,+\infty)\times \{1, \dots, K\},
\end{array}\right.  
$$
and the limit is locally uniform in $x$. 
\end{Lemma}

\begin{proof} The result is obvious if $k=0$ since $U_{e,i}(0)= e^0i$. We have, for $k\in\{1, \dots, K\}$ and $x\leq 0$,  
\begin{align*}
N_e^\omega(x,k,0) & =  \sum_{i\in \Z, \ i\leq 0, \  T_i=k}  \delta_{e^0i}( (x,+\infty))- \sum_{i\in \Z,\ i>0, \ T_i=k} \delta_{e^ki}((-\infty,x]) \\
&=  \sharp\{i\in \Z, \ x/e^0< i\leq 0, \; T_i=k\}
\end{align*} 
(where $\sharp E$ is the cardinal of a set $E$).
So, by the law of large number, we have, a.s. and locally uniformly in $(-\infty, 0]$,  
$$
n^{\ep,\omega}_e(x,k,0)= \ep \sharp\{i\in \Z, \ x/(\ep e^0) < i\leq 0, \; T_i=k\} \to - \pi^kx/e^0 . 
$$
This proves the locally uniform convergence for $\nu^{\ep,\omega}_e(x,k,0)$ if $k\in\{1, \dots, K\}$ on $(-\infty, 0]$. 

The proof for $k\in \{1, \dots, K\}$ and $x> 0$ is similar: since 
\begin{align*}
N_e^\omega(x,k,0) & =  \sum_{i\in \Z, \ i\leq 0, \  T_i=k}  \delta_{e^0i}( (x,+\infty))- \sum_{i\in \Z,\ i>0, \ T_i=k} \delta_{e^ki}((-\infty,x]) \\
&=  - \sharp\{i\in \Z, \ 0< i<x/e^k, \; T_i=k\},
\end{align*}
we have as above that 
$$
n^{\ep,\omega}_e(x,k,0)= - \ep \sharp\{i\in \Z, \ 0< i<x/(\ep e^k),  \; T_i=k\} \to - \pi^kx/e^k,
$$
so that $\nu^{\ep,\omega}_e(x,k,0)$  converges a.s. and locally uniformly to $-x/e^k$.  
\end{proof}

We now want to study the convergence of $\nu^\ep_e$ as $\ep\to0$. Let us first give a regularity result for $\nu_e^\ep$. This result will be proved later in a more general setting (see Lemma \ref{lem:regularity-nu-ep}).
\begin{Lemma}\label{lem:regularity-nu-ep-e}
The function $\nu^\ep_e$ satisfies, for any $k\in \{0, \dots, K\}$,  
$$
\left| \nu^\ep_e(x,k,t)-\nu^\ep_e(y,k,s)\right| \leq C(|x-y|+|t-s|+\ep),
$$
for a constant $C$ depending on $\Delta_{\min}$, $(\pi^k)$ and $\|V\|_\infty$ only. 
\end{Lemma}

\begin{Lemma} \label{lem.limNTOT} Let $e$ be such that $H^k(-1/e^k)=\min_pH^k(p)$ and assume that $k_{e}<0$, where $k_e$ is defined in Theorem \ref{thm.kebarvartheta}. There exists a set $\Omega_0$ of full probability such that, for all $\omega\in \Omega_0$, $\nu^{\omega,\ep}_e$ converges locally uniformly to $\nu_e$ 
which satisfies 
$$
\nu_e(x,0,t)=\nu_e(x,1,t)= \dots = \nu_e(x,K,t) \qquad {\rm in}\; (-\infty, 0]\times [0,+\infty)
$$
and is given by 
\be\label{sol.tildenueTOT}
 \nu_e(x,k,t)=
 \min\left\{ \phi_{{\bar A }}(x,k)-  \bar A t\; ,\; -x/e^k- tH^{k}(-1/e^k)  \right\}\qquad \forall (x,k,t)\in \mathcal R\times [0,+\infty).
\ee
\end{Lemma}

\begin{proof}  Let us denote by $\Omega_1$  the intersection of the set of full probability measure given in Lemma \ref{lem.limNt=0TOT} and in Lemma \ref{lem.homobasic}. Let $\Omega_0$ be the set of $\omega\in \Omega_1$ such that $\theta^\omega_e(t)/t$ converges to $\bar \vartheta_e$ as $t\to +\infty$. By Lemma \ref{defbarthetaTOT}  we know that $\P[\Omega_0]=1$. Fix now $\omega\in \Omega_0$. By Lemma \ref{lem:regularity-nu-ep-e}, we can consider a locally uniform limit $w$, up to a subsequence, of $n^{\omega,\ep}_e(\cdot,0, \cdot)$ in $(-\infty, 0]\times [0,+\infty)$. Then, by standard homogenization \cite{CaFo} (see Subsection \ref{subsec:homog} in Appendix), $w$ solves 
\be\label{ununJuncTOT}
\left\{\begin{array}{l}
\ds \partial_t w +H^0(\partial_x w)= 0 \; {\rm in }\; (-\infty, 0)\times (0,+\infty)\\ 
\ds w(x,0)= -x/e^0 \; {\rm in }\; (-\infty, 0].
\end{array}\right. 
\ee
Moreover, by the definition of $\theta_e$, we have $n^{\omega,\ep}(0,0,t)= \ep\theta^\omega_e(t/\ep)$. Therefore 
\be \label{deudeuxJuncTOT} 
w(0,t)= \bar \vartheta_e t\qquad \forall t\geq 0.
\ee 
The solution to \eqref{ununJuncTOT}-\eqref{deudeuxJuncTOT} is unique and given by 
\be\label{eqwwwwTOT}
w(x,t):= \min\{ -x/e^0-H^0(-1/e^0)t\ , \ p^{0,-}_{{\bar A}} x-\bar At\} \; {\rm in }\; (-\infty, 0)\times (0,+\infty),
\ee
since $\bar A>H^0(-1/e^0)$ and $H^0$ is convex (see Lemma \ref{lem:app2}). Therefore the whole sequence $n^{\omega,\ep}_e(\cdot,0, \cdot)$ converges to  $w$ locally uniformly on $(-\infty,0]\times [0,+\infty)$ as $\ep\to 0$ for any $\omega\in \Omega_0$. We set $n_e(x,0,t)= \nu_e(x,0,t):=w(x,t)$. \\

Let us now fix $\omega\in \Omega_0$, $y<0$ and set $i_\ep= [y/(e^0\ep)]$. Our next step is to show that, if $t< y/(e^0{\bar A})$, then 
\be\label{def.uytTOT}
u(y,t):=  \lim_{\ep\to 0} \ep U^\omega_{e, i_\ep}(t/\ep)= \min\{ y-H^0(-1/e^0)e^0t\ ,\ (-p^{0,-}_{{\bar A}})^{-1}(y/e^0-{\bar A} t)\}.
\ee
Indeed, as the map $t\to \ep U^\omega_{e, i_\ep}(t/\ep)$ is uniformly Lipschitz continuous with $\ep U^\omega_{e, i_\ep}(0)\to y$ as $\ep\to 0$, we can find a subsequence which converges to some  map $t\to u(y,t)$. Assume that $U^\omega_{e, i_\ep}(t/\ep)\leq -R_0$ for any $\ep$ small enough. Then  for any $i\in \Z$ with $i\leq 0$, one has thanks to Lemma \ref{lem.CompDeBase}:
$$
U^\omega_{e, i_\ep}(t/\ep)\leq  U^\omega_{e,i}(t/\ep) \qquad \text{ if and only if } \qquad i\geq i_\ep.
$$
 Therefore
$$
N^{\omega}_e( U^\omega_{e, i_\ep}(t/\ep),0, t/\ep)= -i_\ep.
$$
Multiplying by $\ep$ and letting $\ep \to 0$, we find
$$
w(u(y,t),t)= -y/e^0 
$$
and thus, by \eqref{eqwwwwTOT},
$$
u(y,t)= w^{-1}(-y/e^0,t)=  \min\{ y-H^0(-1/e^0)e^0t\ ,\ (-p^{0,-}_{{\bar A}})^{-1}(y/e^0-{\bar A} t)\}.
$$
This holds if $U^\omega_{e, i_\ep}(t/\ep)\leq -R_0$ for any $\ep$ small enough, which is ensured by the condition $t< y/(e^0{\bar A})$ and $\ep$ small enough thanks to the equality above. Our proof of \eqref{def.uytTOT} is then complete since the limit is independent of the subsequence. \\

We now turn to the proof of the convergence of $n^{\omega,\ep}_e(\cdot, k,\cdot)$ in $(-\infty, 0]\times [0,+\infty)$. 
Fix $\omega\in \Omega_0$, $x<0$  and $t\ge0$. The map $y\mapsto \min\{ y-H^0(-1/e^0)e^0t\ ,\ (-p^{0,-}_{{\bar A}})^{-1}(y/e^0-{\bar A} t)$ being increasing in $(-\infty,t/(e^0{\bar A}))$ and being equal to $0$ for $y=t/(e^0{\bar A})$, there exists $y$ with $t< y/(e^0{\bar A})$ and such that
$$u(y,t)=x.$$
We set $i_\ep= [y/(e^0\ep)]$. Assume also that $n^{\omega,\ep}_e(\cdot, k,\cdot)$ converges up to a subsequence to $n_e(\cdot, k, \cdot)$. By the same argument as above, we have 
$$
N^{\omega}_e(U^\omega_{e, i_\ep}(t/\ep), k,t/\ep)= \sharp\{ i\in \{i_\ep, \dots, 0\}, \; T_i= k\}.  
$$
We multiply by $\ep$ and let $\ep\to 0$. Recalling Lemma \ref{lem.limNt=0TOT} and the previous step, we get
$$
n_e(u(y,t),k,t)= -\pi^ky/e^0=\pi^kw(u(y,t),t).
$$
By definition of $ y$, this shows that $n_e(x,k,t)= \pi^k w(x,t)= \pi^k n_e(x, 0,t)$ as well as the equality $\nu_e(x,k,t)=\nu_e(x,0,t)$ for any $x< 0$ and   $t\geq 0$. By continuity of $\nu_e(\cdot,k,\cdot)$ (this is a direct consequence of Lemma \ref{lem:regularity-nu-ep-e}), we  also get the result for $x=0$.\medskip

Next we turn to the limits for $x\geq 0$. Let $\omega\in\Omega_0$ and $\nu_e(\cdot, k,\cdot)$ be a locally uniform limit, up to a subsequence, of $\nu^{\omega, \ep}(\cdot, k, \cdot)$ for $k\in \{1, \dots, K\}$, which exists by Lemma \ref{lem:regularity-nu-ep-e}. By the previous arguments, we know that for $x\geq 0$
\be\label{condbordnu}
\nu_e(x, k, 0)= -x/e^k, \qquad \nu_e(0, k, t)= -{\bar A} t. 
\ee
On the other hand, on each branch $\mathcal R_k$ the dynamical system corresponds (up to relabelling) to the standard ODE 
$$
\frac{d}{dt}U_j= V^k_{Z_j}(U_{j+1}(t)-U_j(t))
$$
(where the sequence $(Z_j)_{j\in \Z, \ T_j=k}$ is defined for indices $j$ such that $T_j=k$). By homogenization (See Subsection \ref{subsec:homog} in Appendix) $n_e^k$ solves 
$$
\partial_t n_e^k +\bar H^k(\partial_x n_e^k)= 0 \qquad {\rm in}\; (0,+\infty)\times (0,+\infty),
$$
where $\bar H^k(p)= p\bar V^k(-1/p)$ for any $p< 0$ and $\bar H^k(p)=0$ if $p\geq 0$. Hence $\nu_e^k$ solves 
$$
\partial_t \nu_e^k + H^k(\partial_x \nu_e^k)= 0 \qquad {\rm in}\; (0,+\infty)\times (0,+\infty), 
$$
where $H^k$ is given by $H^k(p)= (\pi^k)^{-1}\bar H^k(\pi^k p)= p \bar V^k(-1/(\pi^kp))$. Complemented with \eqref{condbordnu} this system has a unique solution given by 
$$
\nu_e(x,k,t)= \min\left\{ -x/e^k-H^k(-1/e^k)t\ , \ p^{k,+}_{{\bar A}} x-{\bar A}t\right\}.
$$
As before this shows that the whole sequence $\nu^{\omega, \ep}(\cdot, k, \cdot)$ converges to $\nu_e^k$ given above. 
\end{proof}

In the case $k_e=0$, {we have the following result.
\begin{Lemma}\label{lem.caske=0} Assume that $e=(e^k)$ is such that $H^k(-1/e^k)= \min_p H^k(p)$ for any $k\in \{0, \dots, K\}$ and assume that $k_e=0$. Then there exists a set $\Omega_0$ of full probability such that, for all $\omega\in \Omega_0$, $\nu^{\omega,\ep}_e$ converges to $\nu_e$ 
which satisfies 
$$
\nu_e(x,0,t)=\nu_e(x,1,t)= \dots = \nu_e(x,K,t) \qquad {\rm in}\; (-\infty, 0]\times [0,+\infty)
$$
and is given by 
\be\label{sol.tildenueTOTcaske=0}
 \nu_e(x,k,t)=
 \min\left\{ \phi_{A_0}(x,k)- A_0 t\; ,\; -x/e^k- tH^{k}(-1/e^k)  \right\}\qquad \forall (x,k,t)\in \mathcal R\times [0,+\infty),
\ee
where $A_0$ is given by \eqref{defA0}. 

\end{Lemma}

\begin{proof}
Note that with our choice of $e$ we have $A_0= \max_{k\in \{0, \dots, K\}} H^k(-1/e^k)$. Let $\bar \theta_e(t)= \E[\theta_e(t)]$ and $\bar \theta^\ep_e(t)= \ep \bar \theta_e(t/\ep)$. Then, using Lemma \ref{lem.estitheta}, $\bar \theta^\ep_e$  converges, up to a subsequence, to a Lipschitz continuous map $t\to \bar\vartheta_e(t)$. From now on we argue along this subsequence and note that it does not depend on $\omega$. According to Lemma \ref{defbarthetaTOT}, we have $\bar \vartheta_e(t)\geq  -A_0t$. We also note that, by Theorem \ref{thm:concentrationTOT}, $\ep \theta_e(t/\ep)$ converges a.s. locally uniformly to $\bar\vartheta_e(t)$. Let $\Omega_0'$ be the set of $\omega\in \Omega_0$ such that this limit holds  (note that this set depends on the subsequence, we will come back to this point at the very end of the proof).  
Without loss of generality, we also assume that, for any $\omega\in \Omega_0'$, 
$$
\lim_{s\to -\infty}  (-s)^{-1} \sharp  \{ i \in \Z\cap (s,0], \; T_i^\omega=k\}= \pi^k.
$$

Recall that 
$$
N^{\omega}_e(0,k,t)=\sharp \{ i \in \Z, \; i\leq 0, \; T_i^\omega=k, \; U^\omega_{e,i}(t)>0\}=\sharp  \{ i \in \Z\cap (-\theta^\omega_e(t),0], \; T_i^\omega=k \}.
$$
Therefore 
$$
n^{\omega,\ep}(0,k,t) = \ep \sharp  \{ i \in \Z\cap (-\theta^\omega_e(t/\ep ),0], \; T_i^\omega=k\}. 
$$
If $\omega\in \Omega_0'$, then we get 
$$
\lim_{\ep\to 0} n^{\omega,\ep}(0,k,t)= \bar \vartheta_e(t)\pi^k.
$$
because  $\ep \theta^\omega_e(t/\ep )$ converges to the deterministic limit $\bar \vartheta_e(t)$. So arguing as above, for any $\omega\in \Omega'_0$, we can find a subsequence (subsequence of the previous subsequence and depending this time on $\omega$) such that $\nu^{\omega, \ep}_e$ converges locally uniformly to a continuous solution $\nu_e$ of
\be\label{eq.nubarvarthetatTOT}
\left\{\begin{array}{l}
\ds \partial_t \nu_e + H(\partial_x \nu_e)=0 \qquad {\rm in }\;  \stackrel{o}{{\mathcal R}}\times (0,+\infty) \\
\ds \nu_e(x,k, 0)= -x/e^k  \qquad {\rm in }\;  {{\mathcal R}} \\
\ds \nu_e(0,k,t)=\bar \vartheta_e(t) \qquad \forall k\in \{0, \dots, K\}, \; t\geq 0. 
\end{array}\right.
\ee
Let $\tilde \nu_e$ be the solution of the junction problem without flux limiter: 
\be\label{eq.tildenuTOT}
\left\{\begin{array}{l}
\ds \partial_t \tilde \nu_e + H(\partial_x\tilde \nu_e)=0 \qquad {\rm in }\;  \stackrel{o}{{\mathcal R}}\times (0,+\infty) \\
\ds \tilde \nu_e(x,k,0)= -x/e^k \qquad {\rm in }\; \mathcal R\\
\ds \partial_t \tilde \nu_e + \max \{ A_0, H^{0,+}(\partial_0 \tilde \nu_e), H^{1,-}(\partial_1\tilde \nu_e),\dots, H^{K,-}(\partial_K \tilde \nu_e))=0
\qquad {\rm at}\; x=0.
\end{array}\right.
\ee
The solution is given by (see Lemma \ref{lem:app1})
\be\label{sol.tildenueBISTOT}
\tilde \nu_e(x,k,t)= \min\left\{ \phi_{A_0}(x,k)-A_0t\; ,\; -x/e^k- tH^{k}(-1/e^k)  \right\} .
\ee
We  know from \cite[Theorem 2.7]{IM} that $\nu_e$  is a subsolution to \eqref{eq.tildenuTOT}, as it is continuous and satisfies the Hamilton-Jacobi equation in $ (\R\backslash\{0\})\times \{0,\dots,K\}\times(0,+\infty)$. Therefore $\nu_e\leq \tilde \nu_e$ by comparison \cite{IM}. In addition, we get by the definition of $A_0$ and since  $\bar \vartheta_e(t)\geq -A_0t$: 
$$
-A_0t \leq \bar \vartheta_e(t)= \nu_e(0,k,t) \leq \tilde \nu_e(0,k,t)=  -A_0t .
$$
This shows that $ \nu_e(0,k,t) =\bar \vartheta_e(t)=-A_0t$ for any $k$. \\

So we have proved that the whole sequence $\bar \theta_e(t)/t$ converges as $t\to +\infty$ to $-A_0t$. This shows, exactly as in the proof of Lemma \ref{defbarthetaTOT} that $\theta^\omega_e(t)/t$ converges a.s. to $-A_0t$ as $t\to +\infty$. We can then proceed as in the case $k_{e}<0$ and find a set $\Omega_0$ of full probability such that $\nu^{\omega, \ep}_e$ converges for any $\omega\in \Omega_0'$ to the unique solution of \eqref{eq.nubarvarthetatTOT} with 
$\bar \vartheta_e(t)=-A_0t$, which is nothing but $\tilde \nu_e$. This completes the proof of the lemma. 
\end{proof}

For later use it will be convenient to rewrite the previous lemma in term of the behavior of the $U_e$. Let us define, for $k\in \{1, \dots, K\}$ and $x\in \R$
$$
[x]_k=[x]_k^\omega= \sup\{i\in \Z, \; i\leq x, \; T_i=k\}.
$$
We note that, a.s.,  
$$
\lim_{\ep\to0} \ep[x/\ep]_k= x
$$
and that this convergence holds locally uniformly in $x$. 

By the definition of $N_e$, we have, for any $y<0$ and $t\geq 0$,  
$$
N_e(U_{e, [y]_k}(t),k,t)= \sharp \left\{ j\leq 0, \; T_j=k, \; U_{e,j}(t)>U_{e, [y]_k}(t)\right\} 
=
\sharp \left\{ j\leq 0, \; T_j=k, \; j> [y]_k\right\}.
$$
since the order is preserved in time among vehicle with a same type (Lemma \ref{lem.CompDeBase}). Therefore 
$$
\nu^{\ep}_e(\ep U_{e, [y/\ep]_k}(t/\ep),k,t) = \ep(\pi^k)^{-1} \sharp \left\{ j\leq 0, \; T_j=k, \; j> [y/\ep]_k\right\}
$$
from which we derive that any uniform limit $y(\cdot)$ (up to subsequences) of $t\to \ep U_{e, [y/\ep]_k}(t/\ep)$ satisfies
$$
\nu_e(y(t), k,t)=  -y. 
$$
We know by Lemma \ref{lem.limNTOT} (when $k_e<0$) or Lemma \ref{lem.caske=0} (when $k_e=0$),  that 
$$
\nu_e(x,k,t)= \min\left\{  \phi_{\bar A}(x,k)-\bar A t\; ,\; -x/e^k- tH^{k}(-1/(e^k))  \right\} ,\qquad \forall (x,k)\in \mathcal R
$$
where $\bar A=-\bar \vartheta_e$ if $k_e<0$ and $\bar A=A_0$ otherwise. 
This shows that 
$$
y(t)= \left\{\begin{array}{ll}
\min\left\{ \psi_{\bar A}(y-\bar At ,0)\; ,\; ye^0- e^0 t H^{0}(-1/e^0)  \right\} &{\rm if}\; y\le \bar A t\\
\min\left\{ \psi_{\bar A}(y-\bar At ,k)\; ,\; ye^k- e^k t H^{k}(-1/e^k)  \right\} &{\rm if}\; y\ge \bar A t.
\end{array}
\right.
$$
Since this limit is independent of the choice of the subsequence, we have proved the following (the case $y\geq 0$ being treated in the same way): 

\begin{Corollary} \label{cor.CvUepTOT} Let $e$ be such that $H^k(-1/e^k)=\min_pH^k(p)$.  For $k\in\{0,\dots, K\}$, let 
$$
u^\ep_e(y,k,t):= \ep U_{e,[y/\ep]_k}(t/\ep).
$$
Then $u^\ep_e$ converges a.s. and locally uniformly to 
\be\label{identifueJuncTOT}
u_e(y,k,t):= 
\left\{\begin{array}{ll}
\min\left\{ \psi_{{\bar A}}(y-{\bar A}t ,0)\; ,\; ye^0- e^0 t H^{0}(-1/e^0)  \right\} &{\rm if}\; y\le {\bar A} t\\
\min\left\{ \psi_{{\bar A}}(y-{\bar A}t ,k)\; ,\; ye^k- e^k t H^{k}(-1/e^k)  \right\} &{\rm if}\; y\ge {\bar A}t.
\end{array}
\right.
\ee
\end{Corollary}

\subsection{Comparison principle}

 An important point in the proof of the homogenization is to explain how the comparison for the solutions $U$ pass to the limit. This is the aim of the following lemma:

\begin{Lemma}\label{lem.CompaLimJunc} We fix a solution $U$ of 
$$
\frac{d}{dt} U_i(t)=  V_{Z_i}(U_{i+1}(t)-U_i(t), U_{\ell_i}(t)-U_i(t), U_i(t))\qquad i\in \Z
$$
and set $u^\ep(x,k,t)= \ep U_{[x/\ep]_k}(t/\ep)$. {Let $e$ be such that $H^k(-1/e^k)=\min_pH^k(p)$}. There exists a constant $C>1$ and a set $\Omega_0$ of full probability independent of $U$ such that, if $\omega\in \Omega_0$,  if $u^\omega_*$ is any half relaxed lower limit of  $u^{\omega,\ep}$ as $\ep\to 0^+$ (possibly up to a subsequence) and if there exists $\gamma>0$, a time $t_0\geq 0$ and  $a,b\in \R$ with $b> -t_0$, such that  
\be\label{cond.lemCompaLimJunc}
u^\omega_*(x,k,t_0)\geq u_e(x+a,k,t_0+b)\qquad \forall (x,k)\in [-\gamma, \gamma]\times \{1, \dots, K\}
\ee
and such that the minimum of $(x,k) \to u^\omega_*(x,k,t_0)- u_e(x+a,k,t_0+b)$ is not reached on $\{-\gamma, \gamma\}\times \{1,\dots, K\}$, then
$$
u^\omega_*(x,k,t_0+s)\geq u_e(x+a,k,t_0+b+s)\qquad \forall (x,k,s)\in [-\gamma/2, \gamma/2]\times\{1, \dots, K\} \times [0,C^{-1}\gamma]. 
$$
In the same way, if $u^{\omega,*}$ is any half relaxed upper limit of some $U^\omega_i$ (possibly up to a subsequence) and if there exists $\gamma>0$,  $t_0\geq 0$ and  $a,b\in \R$ with $b> -t_0$, such that  
$$
u^{\omega,*}(x,k,t_0)\leq u_e(x+a,k,t_0+b)\qquad \forall x\in [-\gamma, \gamma]\times \{1, \dots, K\}
$$
and such that the maximum of $(x,k) \to u^\omega_*(x,k,t_0)- u_e(x+a,k,t_0+b)$ is not reached on $\{-\gamma, \gamma\}\times \{1,\dots, K\}$,
then
$$
u^{\omega,*}(x,k,t_0+s)\leq u_e(x+a,k,t_0+b+s)\qquad \forall (x,k,t)\in [-\gamma/2, \gamma/2]\times \{1, \dots, K\} \times [0,C^{-1}\gamma]. 
$$
\end{Lemma}

\begin{proof} We only prove the first statement, the proof for the second one being symmetric. Let $\Omega_0$ be the set of $\omega$ such that $u^{\omega,\ep}_e$ converges to $u_e$ locally uniformly, such that  $(\ep J_{ [\gamma/(3\alpha\ep)]}  (\ep^{-1} \gamma))$ converges to $\gamma-\alpha \gamma/(3\alpha)=2\gamma/3$ as $\ep\to 0$ (see Lemma \ref{lem.Jn})  and for which the conclusions of Lemma \ref{lemma.bigprobaJunc} (below) hold. 

Since $u^{\omega,\ep}_e$ converges locally uniformly to $u_e$, for any $\eta\in (0,1)$ small and $M\geq 1$ large (to be chosen below), there exists $\ep_0>0$ such that the set 
$$
E_\eta:= \left\{\omega\in \Omega, \; \sup_{\ep\in (0,\ep_0)} \|u^{\omega,\ep}_e-u_e\|_{L^\infty([-M, M]\times \{1, \dots, k\}\times [0,M])} \leq \eta\right\}
$$
has a probability larger than $1-\eta$. Let us set $E_\eta(\omega):=\{n\in \Z, \; \tau_n\omega\in E_\eta\}$. Let $n_\ep:= [a/\ep]$. By Lemma \ref{lemma.bigprobaJunc} below, there exists $m_{\ep,\eta}\geq  - n_\ep $ with $m_{\ep,\eta}\in E_\eta(\omega)$ and 
\be\label{estimep}
 |m_{\ep,\eta}+n_\ep|\leq C_1(\omega,\eta)+C_2\eta |n_\ep|.
\ee
We will use below that $\ep n_\ep\to a$ as $\ep\to 0$ and therefore that $\ep m_{\ep,\eta}$ converges to $-a$ as $\ep \to 0$ and then $\eta\to 0$.

By \eqref{cond.lemCompaLimJunc}, the fact that $u^\omega_*$ is the half relaxed lower limit of  $u^{\omega,\ep}$ and by the definition of $m_{\ep,\eta}$, there exists $(x_{\ep,\eta},k_{\ep,\eta}) \in (-\gamma, \gamma)\times \{1, \dots, K\}$, minimum point of $u^{\ep,\omega}(\cdot,\cdot,t_0)- u^{\ep,\tau_{m_{\ep,\eta}}\omega}_e(\cdot+a,\cdot,t_0+b)$ such that
$$
u^{\omega,\ep}(x_{\ep,\eta},k_{\ep,\eta},t_0)- u^{\tau_{m_{\ep,\eta}}\omega}_e(x_{\ep,\eta}+a,k_{\ep,\eta}, t_0+b)  \to \min
$$ 
as $\ep$ and $\eta$ tend to $0$, where 
$$
\min:= \min_{(x,k)\in[-\delta,\delta]\times \{0,\dots,K\}}(u^\omega_*(\cdot,\cdot,t_0)- u_e(\cdot+a,\cdot,t_0+b))\geq 0.
$$
By minimality of  $(x_{\ep,\eta},k_{\ep,\eta})$, we have 
\begin{align*}
& u^{\tau_{m_{\ep,\eta}}\omega,\ep}_e(x+a,k,t_0+b)-u^{\tau_{m_{\ep,\eta}}\omega,\ep}_e(x_{\ep,\eta}+a,k_{\ep,\eta},b)\\
& \qquad \leq  u^{\omega,\ep}(x,k,t_0)
-u^{\omega,\ep}(x_{\ep,\eta},k_{\ep,\eta},t_0) \; {\rm for}\;(x,k)\in [-\gamma ,  \gamma]\times \{1,\dots, K\}.
\end{align*}
As $-\ep m_{\ep,\eta}\leq \ep n_\ep\leq a$ and $u_e^\ep$ is nondecreasing in the first variable, we obtain also
\begin{align*}
& u^{\tau_{m_{\ep,\eta}}\omega,\ep}_e(x-\ep m_{\ep,\eta},k,t_0+b)-u^{\tau_{m_{\ep,\eta}}\omega,\ep}_e(x_{\ep,\eta}+a,k_{\ep,\eta},b)\\
& \qquad \leq  u^{\omega,\ep}(x,k,t_0)
-u^{\omega,\ep}(x_{\ep,\eta},k_{\ep,\eta},t_0) \; {\rm for}\;(x,k)\in [-\gamma ,  \gamma]\times \{1,\dots, K\}.
\end{align*}
For $i\in [-\ep^{-1}\gamma , \ep^{-1} \gamma]\cap \Z$, we have, if we set $k:=T^\omega_i=T^{\tau_{m_{\ep,\eta}}\omega}_{i-m_{\ep,\eta}}$ and $x= \ep i$, that $i=[x/\ep]^\omega_k$ and $i-m_{\ep,\eta}= [x/\ep-m_{\ep,\eta}]^{\tau_{m_k}\omega}_k$ with $x\in [-\gamma, \gamma]$. Therefore we can rewrite the inequality above in terms of $U_e$ and $U$ to get 
$$
U^{\tau_{m_{\ep,\eta}}\omega}_{e,i-m_{\ep,\eta}}(\ep^{-1}(t_0+b))-r_{\ep,\eta} \leq  U^{\omega}_{i}(\ep^{-1}t_0)
\;  {\rm for}\;i\in [-\ep^{-1}\gamma , \ep^{-1} \gamma]\cap \Z.
$$
where $r_{\ep,\eta}:=\ep^{-1}(u^{\tau_{m_{\ep,\eta}}\omega,\ep}_e(x_{\ep,\eta}+a,k_{\ep,\eta},b)-u^{\omega,\ep}(x_{\ep,\eta},k_{\ep,\eta},t_0))$. Let us note for later use that, as $\ep\to 0$,  $\ep r_{\ep,\eta}$ converges to $-\min\leq 0$. 
By Lemma \ref{lem.delta} and using the fact that $t_0+b>0$ we obtain, from $\ep$ small enough, 
$$
U^{\tau_{m_{\ep,\eta}}\omega}_{e,i-m_{\ep,\eta}}(\ep^{-1}(t_0+b)-C(r_{\ep,\eta})_+) \leq  U^{\omega}_{i}(\ep^{-1}t_0)
\;  {\rm for}\;i\in [-\ep^{-1}\gamma , \ep^{-1} \gamma]\cap \Z.
$$
As $(U^{\tau_{m_{\ep,\eta}}\omega}_{e,i-m_{\ep,\eta}}(\ep^{-1}(t_0+b)-C(r_{\ep,\eta})_++\cdot))$ and $U^{\omega}_{i}(\ep^{-1}t_0+\cdot))$ solve equation \eqref{eq.SystJuncTOT} and can be compared at time $0$ for $i\in [-\ep^{-1}\gamma , \ep^{-1} \gamma-1]\cap \Z$, we obtain by  approximate finite speed of propagation (Lemma \ref{lem.finitespeedPSJunc}) that for any $n\in \N$, 
$$
U^{\tau_{m_{\ep,\eta}}\omega}_{e,i-m_{\ep,\eta}}(\ep^{-1}(t_0+b)-C(r_{\ep,\eta})_++s) \leq  U^{\omega}_{i}(\ep^{-1}t_0+s)+ C2^{-n}e^{\beta s}
\;  {\rm for}\;i\in [-\ep^{-1}\gamma , J_n(\ep^{-1} \gamma)]\cap \Z, \; s\geq 0.
$$
Coming back to the scaled problem and choosing $n= [\gamma /(3\alpha\ep) ]$ (where $\alpha$ is defined in Lemma \ref{lem.Jn}) and for $s\leq \gamma\ln(2) /(3\alpha\beta\ep)$, so that $-\ln(2)n+ \beta s\leq 0$, this implies that 
\begin{align*}
&u^{\tau_{m_{\ep,\eta}}\omega,\ep}_{e}(x-\ep m_{\ep,\eta},k,b+t_0-C\ep(r_{\ep,\eta})_++t) \leq  u^{\omega,\ep}(x ,k, t_0+ t)+ C\ep \\
 & \qquad {\rm for}\;(x,k,t)\in [-\gamma , \ep J_{ [\gamma/(3\alpha\ep)]} (\ep^{-1} \gamma)]\times \{1,\dots, K\}\times  [0, \gamma /(3\alpha\beta)]. 
\end{align*}
By  the choice of $\omega$, $(\ep J_{ [\gamma/(3\alpha\ep)]}  (\ep^{-1} \gamma))$ converges to $2\gamma/3$ as $\ep\to 0$. So, for $\ep$ small enough, we find 
\begin{align*}
&u^{\tau_{m_{\ep,\eta}}\omega,\ep}_{e}(x-\ep m_{\ep,\eta},k,b+t_0-C\ep(r_{\ep,\eta})_++t) \leq  u^{\omega,\ep}(x ,k, t_0+ t)+ C\ep \\
& \qquad \qquad \qquad \;  {\rm for}\;(x,k,t)\in [-\gamma , \gamma/2]\times \{1,\dots, K\}\times  [0,C^{-1}\gamma].
\end{align*}
So we obtain, from the definition of $E_\eta$ and for $M$ large (depending on  $a$ and $\gamma$ only) and for $\ep$ and $\eta$ small: 
$$
u_{e}(x-\ep m_{\ep,\eta},k,b+t_0-C\ep(r_{\ep,\eta})_++t) \leq  u^{\omega,\ep}(x ,k,t_0+ t)+\eta +C\ep 
\;  {\rm for}\;(x,k,t)\in [-\gamma , \gamma/2]\times \{1,\dots, K\}\times  [0,C^{-1}\gamma].
$$
Recall that $\ep m_{\ep,\eta}$ converges to $-a$ while $C\ep(r_{\ep,\eta})_+$ tends to $0$ as $\ep$ and $\eta$ tend to $0$: we can let $\ep\to 0$ (taking the half relaxed limit in the right-hand side) and then $\eta\to 0$ to find: 
$$
u_{e}(x+a,k,t_0+b+t) \leq  u^\omega_*(x , k,t_0+t)
\;  {\rm for}\;(x,k,t)\in [-\gamma , \gamma/2]\times \{1,\dots, K\}\times  [0,C^{-1}\gamma].
$$
This proves our claim. 
\end{proof}

\begin{Lemma}\label{lemma.bigprobaJunc} Let $E\in \mathcal F$ be such that $\P[E]>1-\ep$ where $\ep\in (0,1/16)$. Let $E(\omega)=\{n\in \Z, \; \tau_n\omega\in E\}$. There is a set $\Omega_0$  of full probability such that, for any $\omega\in \Omega_0$, there exists $C_1(\omega,\ep)$ and $C_2$ universal such that, for any $n\in \Z$, one can find $m^\pm\in E(\omega)$ with $|n-m^\pm|\leq C_1(\omega,\ep)+ C_2\ep|n|$, $m^+\geq n$ and $m^-\leq n$.  
\end{Lemma}

\begin{proof} By the ergodic theorem, we have 
$$
\lim_{r\to+\infty} (2r+1)^{-1}\sharp\left(E(\omega)\cap ([-r,r]\cap \Z)\right)=\P[E]\geq 1-\ep\qquad  {\rm a.s.}.
$$
(where $\sharp (A)$ is the cardinal of $A$). 
 Let $\Omega_0$ be the set of full probability for which this holds. Fix $\omega\in \Omega_0$ and let $R=R(\omega,\ep)$ be such that 
\be\label{lzkaehrfdJunc}
(2r+1)^{-1}\sharp \left(E(\omega)\cap ([-r,r]\cap\Z)\right)\geq 1-2\ep\qquad \forall r\geq R(\omega,\ep).
\ee
Fix $n\in \Z$. For simplicity we assume that $n\geq 0$ and we look for $m^+$. The other case can be treated in a symmetric way. Let us choose 
$r= 11+R+[n(1+8\ep)]$ and assume that $[n,r]\cap E(\omega)=\varnothing$. Then by \eqref{lzkaehrfdJunc} we have 
$r-n \leq 2\ep (2r+1)$, which implies that (as $\ep\leq 1/16$)
$$
10+R+8\ep n\leq r -n \leq 2\ep ( 22+2R+ 2n(1+8\ep))\leq 44/16 +  R/4+ 6\ep n. 
$$
This is impossible and therefore there exists $m^+\in E(\omega)\cap [n, r]$. Then $m^+\in E(\omega)$, $m^+\geq n$ and $m^+-n\leq r-n\leq n+C_1(\omega, \ep)+C_2\ep n$ where $C_1(\omega, \ep)= 12+R(\omega,\ep)$ while $C_2=8$. 
\end{proof}

\subsection{Proof of the homogenization}

From now on we fix $\Omega_0$ such that $\P[\Omega_0]=1$ and such that, for any $\omega\in \Omega_0$, for  $e=(e^k)$  such that $H^k(-1/e^k)=\min H^k$, $\nu_e^\ep$ converges locally uniformly to the map $\nu_e$ given in Lemma \ref{lem.limNTOT} or Lemma \ref{lem.caske=0}. Moreover, we assume that, if $\omega\in \Omega_0$ the conclusions of Lemma \ref{lem.CompaLimJunc} and of Lemma \ref{lem.homobasic} holds. 

Let $(U^\ep_{i,0})$ be a deterministic family of initial conditions satisfying the compatibility condition \eqref{compatibilitycond} and assume, up to relabel the indices, that $U^\ep_{i,0}\le0$  if and only if $i\le0$.  Let $(U^\ep_i)$ be the solution to 
$$
\frac{d}{dt} U^\ep_i(t)=  V_{Z_i}(U^\ep_{i+1}(t)-U^\ep_i(t), U^\ep_{\ell_i}(t)-U^\ep_i(t), U^\ep_i(t))\qquad i\in \Z
$$
with initial condition $U^\ep_i(0)=U^\ep_{i,0}$, $i\in \Z$. 

We set,  for $k\in \{1, \dots, K\}$ and $(x,t)\in\R \times [0,+\infty)$,
\be\label{def.NJuncGen}
N^{\omega}(x,k,t)=  \sum_{i\in \Z, \ i\leq 0, \  T^\omega_i=k}  \delta_{U^\omega_i(t)}( (x,+\infty))- \sum_{i\in \Z,\ i>0, \ T^\omega_i=k} \delta_{U^\omega_i(t)}((-\infty,x]) 
\ee
and, for $(x,t)\in (-\infty, 0]\times [0,+\infty)$,  
$$
N^{\omega}(x,0,t)= \sum_{k=1}^K N^{\omega}(x,k,t).
$$
We set $n^{\omega,\ep}(x,k,t)= \ep N^{\omega}(x/\ep,k,t/\ep)$ and 
$$
\nu^{\omega,\ep}(x,k,t)=\left\{\begin{array}{ll}
n^{\omega,\ep}(x,0,t) & {\rm if}\;  k=0\;{\rm and}\; x\leq 0,\\
\pi_k^{-1}n^{\omega,\ep}(x,k,t)& {\rm if}\; k\in \{1, \dots, K\}\; {\rm and}\; x\in \R.
\end{array}\right.
$$

 Let us first check that $\nu^\ep$ is well defined. 
\begin{Lemma}\label{lem:regularity-nu-ep} Let $(U^\ep_{i,0})$ and $U^\ep$ be as above. Then, for any $t\geq 0$, 
\be\label{jehrsdjhn}
\lim_{i\to \pm \infty} U^\ep_i(t)= \pm \infty.
\ee
Hence $\nu^\ep$ is well-defined and satisfies, for any $k\in \{0, \dots, K\}$,  
$$
\left| \nu^\ep(x,k,t)-\nu^\ep(y,k,s)\right| \leq C(|x-y|+|t-s|+\ep),
$$
for a constant $C$ depending on $\Delta_{\min}$, $\pi^k$ and $\|V\|_\infty$ only. 
\end{Lemma}

\begin{proof} The compatibility condition \eqref{compatibilitycond} implies that \eqref{jehrsdjhn} holds for $t=0$. Then it holds for any $t$ since $V$ is bounded. Fix $k\in \{1, \dots, K\}$ and let $x,y\in \R$ with $x<y$. We have  
$$
\left| \delta_{U^\omega_i(t/\ep)}( (y/\ep,+\infty))- \delta_{U^\omega_i(t/\ep)}( (x/\ep,+\infty))\right| = \left\{ \begin{array}{ll}
1 & {\rm if}\; x/\ep < U^\omega_i(t/\ep) \leq y/\ep, \\
0 & {\rm otherwise}. 
\end{array}\right. 
$$
By Lemma \ref{lem.CompDeBase} there are at most $[(y-x)/(\ep \Delta_{\min})]+1$ vehicles of the same type in $(x/\ep, y/\ep]$.  Arguing in the same way for the difference 
$$
\left|\delta_{U^\omega_i(t)}((-\infty,y]) -\delta_{U^\omega_i(t)}((-\infty,x]) \right|, 
$$
we infer that 
\begin{align*}
\left| \nu^\ep(x,k,t)-\nu^\ep(y,k,t)\right| \leq 2(\pi^{k})^{-1}( |x-y|/\Delta_{\min}+\ep).
\end{align*}
Fix now $0\leq s<t$. We have 
$$
\left| \delta_{U^\omega_i(t/\ep)}( (x/\ep,+\infty))- \delta_{U^\omega_i(s/\ep)}( (x/\ep,+\infty))\right| = \left\{ \begin{array}{ll}
1 & {\rm if}\; U^\omega_i(s/\ep) \leq x/\ep < U^\omega_i(t/\ep), \\
0 & {\rm otherwise}. 
\end{array}\right. 
$$
Let $i_0\in \Z$ be the smallest index such that $U^\omega_{i_0}(s/\ep) \leq x/\ep < U^\omega_{i_0}(t/\ep)$ and $T_{i_0}=k$ and $i_1$ be the largest one. Then 
$$
U^\omega_{i_1}(s/\ep) \leq x/\ep < U^\omega_{i_0}(t/\ep) \leq U^\omega_{i_0}(s/\ep) + \|V\|_\infty (t-s)/\ep.
$$
Still by Lemma \ref{lem.CompDeBase} we must have $i_1-i_0\leq  [\|V\|_\infty (t-s)/(\ep \Delta_{\min})]+1$, so that 
$$
\left| \nu^\ep(x,k,t)-\nu^\ep(x,k,s)\right| \leq 2(\pi^k)^{-1}(\|V\|_\infty (t-s)/( \Delta_{\min})+ \ep). 
$$
The Lipschitz continuity of $\nu^\ep(\cdot,0,\cdot)= \sum_{k=1}^K \pi^k \nu^\ep(\cdot,k,\cdot) $ is then immediate. 
\end{proof}

We assume that $\nu^{\omega,\ep}(\cdot, \cdot, 0)\to \nu_0$ locally uniformly, where $\nu_0$ is deterministic. Note that $\nu_0$ is  Lipschitz continuous  and satisfies $\nu_0(x,0)=\nu_0(x,1)=\dots=\nu_0(x,k)$ for $x\le 0$.  We fix $\omega\in \Omega_0$ and let $\nu^\omega$ be any uniform limit of $\nu^{\omega,\ep}$. We already know  (cf. Subsection \ref{subsec:homog} in Appendix) that $\nu^\omega$ satisfies
$$
\left\{\begin{array}{l}
\ds \partial_t \nu + H(\partial_x \nu)=0\qquad {\rm in}\; \stackrel{o}{{\mathcal R}}\times (0,T)\\
\ds \nu(x,k,0)= \nu_0(x,k) \qquad {\rm in}\;\mathcal R.
\end{array}\right.
$$
Our aim is to show that $\nu^\omega$ also satisfies 
$$
\partial_t \nu +\max\{ \bar A,  H^{0,+}(\partial_0 \nu), H^{1,-}(\partial_1 \nu), \dots, H^{K,-}(\partial_K \nu))\}=0\; {\rm at}\; x=0.
$$

We first show that $\nu^\omega$  is continuous in $0$ (and does not depend on $k$ for $x\le0$).
\begin{Lemma}
 Let $\nu^\omega$ be any uniform limit (up to subsequences) of $\nu^{\omega,\ep}$. Then, for all $t\ge 0$ and $x\le 0$
 $$\nu^\omega(x,0,t)=\nu^\omega(x,1,t)=\dots=\nu^\omega(x,K,t).$$
\end{Lemma}
\begin{proof}
Let $x\le 0$ and $i_0^\ep, \; i_1^\ep\in \N $ be the indices such that $U_{-i_0^\ep}(0)\le x/ \ep<U_{-i_0^\ep+1}(0)$ and $U_{-i_1^\ep}\left(t/\ep\right)\le  x/\ep<U_{-i_1^\ep+1}\left(t/\ep\right)$. 
We assume in a first step that $\nu^\omega(x,0,t)-\nu^\omega (x,0,0)\ge 2c>0$. Then, for $\ep$ small enough, $\nu^{\omega,\ep}(x,0,t)-\nu^{\omega,\ep}(x,0,0)\ge c$. As, by assumption, $U^\ep_{i,0}\leq 0$ if and only if $i\leq 0$, this implies that
$$\sum_{i\le0}\delta_{U_i^\omega(t/\ep)}((x/\ep,+\infty))-\sum_{i\le0}\delta_{U_i^\omega(0)}((x/\ep,+\infty))\ge c/\ep.$$ 
Since by Lemma \ref{lem.CompDeBase}, the cars remains ordered before $0$, we deduce that $i_1^\ep-i_0^\ep\ge c/\ep$. Moreover, since $U_{-i_1^\ep+1}(0)-x/\ep\le U_{-i_1^\ep+1}(0)-U_{-i_0^\ep}(0)\le \Delta_{\min}(i_0^\ep-i_1^\ep +1)$, we get that  $i_1^\ep- i_0^\ep\le \frac {||V||_\infty}{\Delta_{\min}}\frac t \ep +1$. We then have that $\ep(i_1^\ep- i_0^\ep)$ is bounded and converges, up to a subsequence, to a constant $z$.
Remarking that 
$$\nu^{\omega,\ep}(x,0,t)-\nu^{\omega,\ep}(x,0,0)=\ep (i_1^\ep-i_0^\ep)\to  \nu^{\omega}(x,0,t)-\nu^\omega(x,0,0),$$  
 we deduce that $z=\nu^{\omega}(x,0,t)-\nu^{\omega}(x,0,0)$.  Hence, for every $k\in\{1,\dots,K\}$,
by the law of large number,
$$\nu^{\omega,\ep}(x,k,t)-\nu^{\omega,\ep}(x,k,0)=(\pi_k)^{-1} \ep \sharp\{ i\in\{-\ep i_1^\ep/\ep,\dots, -\ep i_0^\ep/\ep\}\; {\rm s.t.}\; T_i^\omega=k \}\to (\pi_k)^{-1} \pi_k z=z.$$  
Since $\nu^\omega(x,k,0)=\nu^\omega(x,0,0)$, this implies that $\nu^\omega(x,k,t)=\nu^\omega(x,0,t).$

Assume now that $\nu^\omega(x,0,t)-\nu^\omega(x,0,0)=0$. Then 
$$0\le n^{\omega,\ep}(x,k,t)-n^{\omega,\ep}(x,k,0)\le \nu^{\omega,\ep}(x,0,t)-\nu^{\omega,\ep}(x,0,0).$$
Sending $\ep\to 0$, we deduce that $\nu^\omega(x,k,t)-\nu^\omega(x,k,0)=(\pi_k)^{-1}(n^\omega(x,k,t)-n^\omega(x,k,0))=0$ and so $\nu^\omega(x,k,t)=\nu^\omega(x,k,0)=\nu^\omega(x,0,0)=\nu^\omega(x,0,t).$
\end{proof}

It will be convenient to consider also the limit of $u^{\omega, \ep}(y,k,t):= \ep U^{\omega}_{[y/\ep]_k}(t/\ep)$ along the same subsequence as for $\nu^{\omega,\ep}$ for $k\in \{1, \dots, K\}$. Let $u^{\omega,*}$ and $u^\omega_*$ be the half-relaxed limits of  $u^{\omega, \ep}$ (along that same subsequence). 
 As $x\to \nu^\omega(x,k,t)$ is nonincreasing, it has an inverse
\begin{align*}
\tilde u^\omega(y,k,t)& :=\inf\{ x\in \R, \; \nu^\omega(x,k,t)< -y\} \qquad \text{(with $\tilde u^\omega(y,k,t)=+\infty$ if there is no such a $x$)}\\
& =  \sup\{ x\in \R, \; \nu^\omega(x,k,t)\geq -y\} 
\end{align*}
Note that $\tilde u^\omega$ is usc, while its lower semicontinuous envelope is given by 
$$
\tilde u^\omega_*(y,k,t):=\inf\{ x\in \R, \; \nu^\omega(x,k,t)\leq -y\}=\sup\{ x\in \R, \; \nu^\omega(x,k,t)> -y\} .
$$

\begin{Lemma}\label{lem.u*tildeuJunc} We have 
$$
u^{\omega,*}\leq \tilde u^\omega\; \text{and}\; u^{\omega}_*\geq \tilde u^\omega_*.
$$
\end{Lemma}

Note that, at each point where $\tilde u^\omega$ is continuous, we have $u^{\omega,*}= \tilde u^\omega= u^{\omega}_*$. 

\begin{proof} We only do the proof of the first inequality, the proof of the other one being similar. Fix $(y,k,t)$ with $k\in \{1, \dots, K\}$ and let $(y_\ep, t_\ep)\to (y,t)$ be such that $u^{\omega, \ep}(y_\ep,k,t_\ep)\to u^{\omega,*}(y,k,t)$. 
Let $i_\ep:= [y_\ep/\ep]_k$. By the definition of $N$  and the fact that the $U_i$ with $T_i=K$ remain ordered (see  Lemma \ref{lem.CompDeBase}), we have 
$$
N^{\omega}(U^\omega_{i_\ep}(t_\ep/\ep),k,t_\ep/\ep)= N^{\omega}(U^\omega_{i_\ep}(0),k,0)= \left\{\begin{array}{ll}
\sharp \{j\in \{i_\ep, \dots, -1\}, \;  T_j= T_{i_\ep}\} & {\rm if }\; i_\ep\leq 0\\
-\sharp \{j\in \{0,\dots, i_\ep\}, \;  T_j= T_{i_\ep}\} & {\rm if }\; i_\ep>0. \end{array}\right. 
$$
Hence 
$$
\nu^{\omega,\ep}(\ep U^\omega_{i_\ep}(t_\ep/\ep),k,t_\ep)= \ep (\pi^k)^{-1}  \left\{\begin{array}{ll}
\sharp \{j\in \{[y_\ep/\ep]_k, \dots, -1\}, \;  T_j= k\} & {\rm if }\; [y_\ep/\ep]_k^\omega\leq 0\\
-\sharp \{j\in \{0,\dots, [y_\ep/\ep]_k\}, \;  T_j= k\} & {\rm if }\; [y_\ep/\ep]_k^\omega>0. \end{array}\right. , 
$$
where $\ep U^\omega_{i_\ep}(t_\ep/\ep)= \ep U^\omega_{[y_\ep/\ep]_k}(t_\ep/\ep)= u^{\omega,\ep}(y_\ep, k,t_\ep)\to u^{\omega,*}(y,k,t)$ while $\ep i_\ep\to y$. So, by uniform convergence of $\nu^{\omega, \ep}$, we obtain $\nu^\omega(u^{\omega,*}(y,t),k,t)=-y$. This shows that $\tilde u^\omega(y,t) \geq u^{\omega,*}(y,t)$. 
\end{proof}

\begin{Lemma} [Supersolution at the junction] \label{lem:super}
Let $\xi:[0,T]\to \R $ be a smooth test function and $A>\bar A$ be such that $(x,k,t)\to \nu^\omega(x,k,t)-\xi(t)-\phi_{A}(x,k)$ has a local minimum on $\mathcal R\times (0,+\infty)$ at $(0,t_0)$. Then 
$$
\xi'(t_0) + A\geq 0.
$$
\end{Lemma}

\begin{proof} As $(x,k,t)\to\nu^\omega(x,k,t)-\xi(t)-\phi_{A}(x,k)$ has a local minimum in $\mathcal R\times (0,+\infty)$ at $(0,t_0)$ and $\phi_{\bar A}<\phi_{A}$ on $\stackrel{o}{{\mathcal R}}$ with an equality at $0$, modifying $\xi$ if necessary, the map $(x,k,t)\to\nu^\omega(x,k,t)-\xi(t)-\phi_{\bar A}(x,k)$ has a strict local minimum $\mathcal R\times (0,+\infty)$  at $(0,t_0)$: assuming that $\xi(t_0)=0$, there exists $\gamma>0$ such that, for any $(x,k,t)\in \mathcal R\times (t_0-\gamma, t_0+\gamma)$ with $(x,t)\neq (0,t_0)$ and $|x|\leq \gamma$, 
\be\label{azhebjnsrdm}
\nu^\omega(x,k,t)-\xi(t)-\phi_{\bar A}(x,k) > \nu^\omega(0,t_0),
\ee
with an equality at $(0,t_0)$. As $\nu^\omega(x,k,t_0)=\nu^\omega(x,0,t_0)$ for $x\leq 0$, inequality \eqref{azhebjnsrdm} actually holds for any  $(x,k,t)\in (-\gamma, \gamma)\times \{1, \dots, K\}\times (t_0-\gamma, t_0+\gamma)$ with $(x,t)\neq (0,t_0)$.  

Let $y_0=- \nu^\omega(0,t_0)$. By \eqref{azhebjnsrdm}, we have that $\nu^\omega(x,k,t_0)>-y_0$ for $x\in(-\gamma,0)$, so that
$$
\tilde u^\omega_*(y_0,k,t_0)=  \inf\{ x\in \R, \; \nu^\omega(x,k,t_0)\leq -y_0\} =0.
$$
By continuity of $\nu^\omega$, there exists $\gamma'\in (0,\gamma)$ such that, if $(y,k,t)\in (y_0-\gamma',y_0+\gamma')\times \{1, \dots, K\}\times (t_0-\gamma',t_0+\gamma')$, then 
$$
\tilde u^\omega_*(y,k,t)=  \inf\{ x\in (-\gamma, +\infty), \; \nu^\omega(x,k,t)\leq -y\}. 
$$
Therefore 
\begin{align*}
\tilde u^\omega_*(y,k,t)& \geq \min\{\gamma\ ,\  \inf\{ x\in (-\gamma, \gamma), \; \nu^\omega(x,k,t)\leq -y\} \ \}  \\ 
& \geq  \min\{ \gamma \ , \  \inf\{ x\in (-\gamma,\gamma), \; \xi(t)+\phi_{\bar A}(x,k) - y_0\leq -y\} \ \}\\
& \geq  \min\{ \gamma \ , \  \inf\{ x\in \R, \; \xi(t)+\phi_{\bar A}(x,k) - y_0\leq -y\} \ \}\; =\; \min \{\gamma\ ,\  \psi_{\bar A}\left( y-y_0+\xi(t),k\right)\}. 
\end{align*}
If $(y,k,t)=(y_0,k,t_0)$, then $\psi_{\bar A}\left( y_0-y_0+\xi(t_0),k\right)= \psi_{\bar A}\left( 0,k\right)=0<\gamma$, so that, reducing $\gamma'$ if necessary, we get 
\be\label{iuemsrd:kk}
\tilde u^\omega_*(y,k,t) \geq \psi_{\bar A}\left( y-y_0+\xi(t),k\right)\qquad \forall (y,k,t)\in (y_0-\gamma',y_0+\gamma')\times \{1,\dots, K\}\times (t_0-\gamma',t_0+\gamma'), 
\ee
In addition, as the inequality in \eqref{azhebjnsrdm} is strict, we have a strict inequality in the above inequality unless $(y,k,t)=(y_0,k,t_0)$. 
By \eqref{identifueJuncTOT} we have 
$$
u_e(y,k,t):= \left\{\begin{array}{ll}
\min\left\{ \psi_{\bar A}(y-\bar At ,k)\; ,\; ye^0- e^0 t H^{0}(-1/e^0)  \right\} &{\rm if}\; y\le \bar A t\\
\min\left\{ \psi_{\bar A}(y-\bar At ,k)\; ,\; ye^k- e^k t H^{k}(-1/e^k)  \right\} &{\rm if}\; y\ge \bar At.
\end{array}
\right.
$$
Let us fix $T>0$ and set $y_T:= \bar A T$. The equality above can be rewritten as 
\begin{align*}
& u_e(y+\bar A (t-t_0)+y_T, k, t-t_0+T) \\
& \qquad =  \left\{\begin{array}{ll}
\min\left\{ \psi_{\bar A}(y ,k)\; ,\; (y+\bar A (t-t_0)+y_T)e^0- e^0 ( t-t_0+T) H^{0}(-1/e^0)  \right\} &{\rm if}\; y\le 0\\
\min\left\{ \psi_{\bar A}(y ,k)\; ,\; (y+\bar A (t-t_0)+y_T)e^k- e^k ( t-t_0+T) H^{k}(-1/e^k)  \right\} &{\rm if}\; y\ge 0.
\end{array}
\right.
\end{align*}
By \eqref{iuemsrd:kk}, this implies that, for any $ (y, k,t)\in (y_0-\gamma',y_0+\gamma')\times \{1,\dots, K\}\times (t_0-\gamma',t_0+\gamma')$,
$$
\tilde u^\omega_*(y,k,t) \geq \psi_{\bar A}\left( y-y_0+\xi(t),k\right) \geq  u_e(y-y_0+\xi(t) +\bar A (t-t_0)+y_T, k, t-t_0+T). 
$$
with a strict inequality if $(y,k,t)\neq (y_0,k,t_0)$. 

We apply Lemma \ref{lem.CompaLimJunc} with initial time $t_0-\tau$, where $\tau>0$ is so small that the minimum of  the map
$$
(y,k)\to u^{\omega}_*(y,k,t_0-\tau) - u_e(y-y_0+\xi(t_0-\tau) -\bar A \tau +y_T, k, -\tau+T).
$$
is not reached at $y\in \{-\gamma,\gamma\}$: this is possible since this minimum point converges to $(y_0,0)$ as $\tau\to 0+$. Then by Lemma \ref{lem.CompaLimJunc} we get, if $s\geq 0$ and $|y-y_0|$ are small enough (depending on $\gamma'$ only)
$$
u^{\omega}_*(y,k,t_0-\tau+s)\geq u_e(y-y_0+\xi(t_0-\tau) -\bar A \tau +y_T, k, -\tau+T+s).
$$ 
For $y=y_0$ and $s=\tau$, we get 
\begin{align*}
0 & = u^{\omega}_*(y_0,k,t_0) \geq u_e(\xi(t_0-\tau) -\bar A \tau +y_T, k, T)\\
&= \left\{\begin{array}{ll}
\min\left\{ \psi_{\bar A}(\xi(t_0-\tau) -\bar A \tau +y_T -\bar A T ,k)\; ,\; e^0(\xi(t_0-\tau) -\bar A \tau +y_T-T H^{0}(-1/e^0) ) \right\} &{\rm if}\; \xi(t_0-\tau)\le \bar A\tau\\
\min\left\{ \psi_{\bar A}(\xi(t_0-\tau) -\bar A \tau +y_T -\bar A T ,k)\; ,\; e^k(\xi(t_0-\tau) -\bar A \tau +y_T-  T H^{k}(-1/e^k) ) \right\} &{\rm if}\; \xi(t_0-\tau)\ge \bar A\tau
\end{array}
\right.
\\
& \geq 
\left\{\begin{array}{ll}
\min\left\{ \psi_{\bar A}(\xi(t_0-\tau) -\bar A \tau  ,k)\; ,\; e^0(\xi(t_0-\tau) -\bar A \tau  ) \right\} &{\rm if}\; \xi(t_0-\tau)\le \bar A\tau\\
\min\left\{ \psi_{\bar A}(\xi(t_0-\tau) -\bar A \tau  ,k)\; ,\; e^k(\xi(t_0-\tau) -\bar A \tau  ) \right\} &{\rm if}\; \xi(t_0-\tau)\ge \bar A\tau
\end{array}
\right.
\end{align*}
because 
$y_T = \bar AT\geq T\max_k H^k(-1/e^k)$. Assume now that $\xi'(t_0) < -\bar A$. Then, since $\xi(t_0)=0$ and $\xi'(t_0)<-\bar A$, one has $\xi(t_0-\tau)- \bar A \tau>0$ and thus the right-hand side in the inequality above is positive. This leads to a contradiction and shows that 
$\xi'(t_0)\geq -\bar A\geq -A$.
\end{proof}

\begin{Lemma}[Subsolution at the junction]\label{lem:sub}  
Assume that $\bar A>A_0$ and let $A_0\le A<\bar A$ and  $\xi:[0,T]\to \R $ be a smooth test function  be such that $(x,k,t)\to \nu^\omega(x,k,t)-\xi(t)-\phi_{A}(x,k)$ has a local maximum on $\mathcal R\times (0,+\infty)$ at $(0,t_0)$. Then 
$$
\xi'(t_0) + A\leq 0.
$$
\end{Lemma} 

\begin{proof}    We argue  as in the supersolution case. As $(x,k,t)\to\nu^\omega(x,k,t)-\xi(t)-\phi_{A}(x,k)$ has a local maximum on $\mathcal R\times (0,+\infty)$ at $(0,t_0)$ and $\phi_{\bar A}>\phi_{A}$ on $\stackrel{o}{{\mathcal R}}$ with an equality at $0$, modifying $\xi$ if necessary, the map $(x,k,t)\to\nu^\omega(x,k,t)-\xi(t)-\phi_{\bar A}(x,k)$ has a strict local maximum at $(0,t_0)$: assuming that $\xi(t_0)=0$, there exists $\gamma>0$ such that, for any $(x,k,t)\in \mathcal R \times (t_0-\gamma, t_0+\gamma)$ with $(x,t)\neq (0,t_0)$ and $|x|\leq \gamma$,
\be\label{azhebjnsrdmSUB}
\nu^\omega(x,k,t)-\xi(t)-\phi_{\bar A}(x) < \nu^\omega(0,t_0),
\ee
with an equality at $(0,t_0)$. As $\nu^\omega(x,k,t)= \nu^\omega(x,0,t)$ for $x\leq 0$, inequality \eqref{azhebjnsrdmSUB} also holds for any $(x,k,t)\in (-\gamma, \gamma)\times \{1, \dots, K\}\times (t_0-\gamma, t_0+\gamma)$ with $(x,t)\neq (0,t_0)$. 

Let $y_0=- \nu^\omega(0,t_0)$. By \eqref{azhebjnsrdmSUB}, we have that $\nu^\omega(x,k,t_0)<-y_0$ for $x\in(0,\gamma)$, so that 
$$
\tilde u^\omega(y_0,k,t_0)=  \sup\{ x\in \R, \; \nu^\omega(x,k,t_0)\geq -y_0\} =0.
$$
By continuity of $\nu^\omega$, there exists $\gamma'\in (0,\gamma)$ such that, if $(y,k,t)\in (y_0-\gamma',y_0+\gamma')\times \{1, \dots, K\}\times (t_0-\gamma',t_0+\gamma')$, then 
$$
\tilde u^\omega(y,k,t)=  \sup\{ x\in (-\infty, \gamma), \; \nu^\omega(x,k,t)\geq -y\}. 
$$
Therefore 
\begin{align*}
\tilde u^\omega(y,k,t)& \leq \max\{-\gamma\ ,\  \sup\{ (-\gamma, \gamma), \;  \nu^\omega(x,k,t)\geq -y\} \ \}  \\ 
& \leq  \max\{ -\gamma \ , \  \sup\{ x\in (-\gamma,\gamma), \; \xi(t)+\phi_{\bar A}(x,k) - y_0\geq -y\} \ \}\\
& \leq  \max\{ -\gamma \ , \  \sup\{ x\in \R, \; \xi(t)+\phi_{\bar A}(x,k) - y_0\geq -y\} \ \}\; =\; \max \{- \gamma\ ,\  \psi_{\bar A}\left( y-y_0+\xi(t),k\right)\}. 
\end{align*}
If $(y,k,t)=(y_0,k,t_0)$, then $\psi_{\bar A}\left( y_0-y_0+\xi(t_0),k\right)= \psi_{\bar A}\left( 0,k\right)=0>-\gamma$, so that, reducing $\delta'$ if necessary, we get 
\be\label{iuemsrd:kkSUB}
\tilde u^\omega(y,k,t) \leq \psi_{\bar A}\left( y-y_0+\xi(t),k\right)\qquad \forall (y,k,t)\in (y_0-\gamma',y_0+\gamma')\times \{1,\dots, K\}\times (t_0-\gamma',t_0+\gamma'). 
\ee
In addition, as  inequality in \eqref{azhebjnsrdm} is strict, we have a strict inequality in the above inequality unless $(y,k,t)=(y_0,k,t_0)$. 
By \eqref{identifueJuncTOT} we have
$$
u_e(y,k,t):= \left\{\begin{array}{ll}
\min\left\{ \psi_{\bar A}(y-\bar At ,k)\; ,\; ye^0- e^0 t H^{0}(-1/e^0)  \right\} &{\rm if}\; y\le \bar A t\\
\min\left\{ \psi_{\bar A}(y-\bar At ,k)\; ,\; ye^k- e^k t H^{k}(-1/e^k)  \right\} &{\rm if}\; y\ge \bar At.
\end{array}
\right.
$$
Let us fix $T>0$ and set $y_T:= \bar A T$. Note that $u_e(y_T,k,T)= \psi_{\bar A}(y_T-\bar A T ,k) < y_Te^k- e^kT H^{k}(-1/e^k)$ because $\bar A> \max_{k\in \{0, \dots K\}}H^k(-1/e^k)$. So, reducing $\gamma'$ if necessary, the equality above can be rewritten as 
\be\label{lhkaebzsndfl}
u_e(y+\bar A t+y_T, t+T) = \psi_{\bar A}(y ,k)\qquad \forall (y, k, t)\in (-\gamma',\gamma')\times \{1,\dots, K\}\times (-\gamma',\gamma').
\ee
By \eqref{iuemsrd:kkSUB}, this implies that, for any $ (y, k,t)\in (y_0-\gamma',y_0+\gamma')\times \{1,\dots, K\}\times (t_0-\gamma',t_0+\gamma')$,
$$
\tilde u^\omega(y,k,t) \leq \psi_{\bar A}\left( y-y_0+\xi(t),k\right) =  u_e(y-y_0+\xi(t) +\bar A (t-t_0)+y_T, k, t-t_0+T). 
$$
By Lemma \ref{lem.CompaLimJunc}, applied at time $t_0-\tau$, we get, if $s\geq 0$ and $|y-y_0|$ are small enough (depending on $\gamma'$ only)
$$
u^{\omega}(y,k,t_0-\tau+s)\leq u_e(y-y_0+\xi(t_0-\tau) -\bar A \tau +y_T, k, -\tau+T+s).
$$
For $y=y_0$ and $s=\tau$, we get 
\begin{align*}
0 & = u^{\omega}(y_0,k_\tau,t_0) \leq u_e(\xi(t_0-\tau) -\bar A \tau +y_T, k_\tau, T)\\
& =\psi_{\bar A}(\xi(t_0-\tau) -\bar A \tau +y_T -\bar A T ,k_\tau)
= \psi_{\bar A}(\xi(t_0-\tau) - \bar A \tau,k_\tau),  
\end{align*}
where the second equality holds because of \eqref{lhkaebzsndfl}.
 Then, if $\xi'(t_0)>-\bar A$ and  since $\xi(t_0)=0$ , one has $\xi(t_0-\tau)- \bar A \tau<0$ and thus the right-hand side in the inequality above is negative. This leads to a contradiction and shows that $\xi'(t_0) \leq -\bar A\leq -A$. 
\end{proof}
\begin{proof}[Proof of Theorem \ref{thm.main}]
We just have to show that  $\nu^\omega$  satisfies in the viscosity sense
$$
\partial_t \nu +\max\{ \bar A,  H^{0,+}(\partial_0 \nu), H^{1,-}(\partial_1 \nu), \dots, H^{K,-}(\partial_K \nu))\}=0\; {\rm at}\; x=0.
$$
Let $A_n> \bar A$ be such that $A_n\to A$. By Lemma \ref{lem:super} and \cite[Theorem 2.11]{IM}, $\nu^\omega$ is a super-solution of 
$$
\left\{\begin{array}{l}
\ds \partial_t \nu + H(\partial_x \nu)=0\qquad {\rm in}\; \stackrel{o}{{\mathcal R}}\times (0,T)\\
\ds \nu(x,k,0)= \nu_0(x,k) \qquad {\rm in}\;\mathcal R\\
\partial_t \nu +\max\{  A_n,  H^{0,+}(\partial_0 \nu), H^{1,-}(\partial_1 \nu), \dots, H^{K,-}(\partial_K \nu))\}=0\; {\rm at}\; x=0.
\end{array}\right.
$$
By stability \cite[Proposition 2.6]{IM}, we then get that 
$\nu^\omega$  satisfies 
$$
\partial_t \nu +\max\{ \bar A,  H^{0,+}(\partial_0 \nu), H^{1,-}(\partial_1 \nu), \dots, H^{K,-}(\partial_K \nu))\}\ge0\; {\rm at}\; x=0.
$$
We now turn to the sub-solution property.
Following \cite[Theorem 2.7]{IM}, $\nu^\omega$ being continuous and a subsolution of the Hamilton-Jacobi equation in $\stackrel{o}{{\mathcal R}}$, is a subsolution at $x=0$ with $A=A_0$. So we can assume from now on that $\bar A> A_0$. Arguing as above (taking $A_n< \bar A$ with $A_n\to \bar A$), we then get that 
$$
\partial_t \nu +\max\{ \bar A,  H^{0,+}(\partial_0 \nu), H^{1,-}(\partial_1 \nu), \dots, H^{K,-}(\partial_K \nu))\}\le0\; {\rm at}\; x=0.
$$
\end{proof}

\appendix 
\section{Appendix}

\subsection{Computation for Lemma \ref{lem.Jn}} \label{subsec:lemmaJn}

Let $X:= X_0-X_1$. We have $\P\left[X>m\right] \leq Kr^m$ for $m\in \N$ and where $r:=(\pi+1)/2 \in (0,1)$. So, for $q\geq 1$, 
\begin{align*}
\E\left[ |X|^q\right] & = \sum_{m=0}^\infty m^q \P\left[X= m\right] \leq \sum_{m=1}^\infty m^q \P\left[X\geq m\right]\leq \frac{K}{r} \sum_{m=1}^\infty m^q r^m .
\end{align*}
Note that $x\to x^qr^x$ is increasing on $[0, -q/\ln(r)]$ and decreasing on $[ -q/\ln(r),+\infty]$. So we can approximate the  sum in the right-hand side by 
\begin{align*}
\frac{K}{r} \sum_{m=1}^\infty m^q r^m \leq \frac{2K}{r}\int_0^\infty x^q r^x dx = \frac{2K}{r |\ln(r)|^{q+1}} \int_0^\infty y^qe^{-y}dy = \frac{2Kq!}{r |\ln(r)|^{q+1}}. 
\end{align*}

\subsection{Flux-limited solutions}

\begin{Lemma}\label{lem:app1}
Assume that $e=(e^k)$ is such that $H^k(-1/e^k)= \min_p H^k(p)$ for any $k\in \{0, \dots, K\}$. The solution of the junction problem without flux limiter: 
\be\label{eq.tildenuTOT-app}
\left\{\begin{array}{l}
\ds \partial_t \tilde \nu_e + H(\partial_x\tilde \nu_e)=0 \qquad {\rm in }\;  \stackrel{o}{{\mathcal R}}\times(0,+\infty)\\
\ds \tilde \nu_e(x,k,0)= -x/e^k   \qquad {\rm in }\;{{\mathcal R}}\\
\ds \partial_t \tilde \nu_e + \max \{ A_0, H^{0,+}(\partial_0 \tilde \nu_e), H^{1,-}(\partial_1 \tilde \nu_e),\dots, H^{K,-}(\partial_K \tilde \nu_e)\}=0
\qquad {\rm at}\; x=0.
\end{array}\right.
\ee
is given by 
\be\label{sol.tildenueBISTOT-app}
\tilde \nu_e(x,k,t)= \min\left\{ \phi_{A_0}(x,k)-A_0t\; ,\; -x/e^k- tH^{k}(-1/e^k)  \right\} .
\ee
\end{Lemma}

\begin{proof}
By stability of super-solution, we classically have that $\tilde \nu_e$ is a super-solution. Let us prove that it is a sub-solution. First remark that for $x\ne 0$, $(x,k,t)\mapsto \phi_{A_0}(x,k)-A_0t$ and $(x,k,t)\mapsto -x/e^k- tH^{k}(-1/e^k)$ are (smooth) solutions of the equation. Then, using \cite[Theorem 9.2 ii)]{Ba2013}, $\tilde \nu_e$ is a sub-solution. We then have to consider the case $x=0$.
Let $\phi$ be a test function such that $\tilde \nu_e-\phi$ reaches a minimum at $(0,t_0)$. By \cite[Theorem 2.7]{IM}, it is sufficient to take $\phi$ such that $\phi(x,k,t)=\psi(t)+\phi_{A_0}(x,k)$. Since for $x$ close to $0$, we have $\tilde \nu_e(x,t,k)=\phi_{A_0}(x,k)-A_0t$ (because $A_0\ge H^k(-1/e^k)$), we deduce that $t\mapsto -A_0t-\psi(t)$ reaches a minimum at $t_0$ and so $\psi'(t)=-A_0$.
This implies that
$$
\ds \partial_t  \phi + \max \{ A_0, H^{0,+}(\partial_0 \phi), H^{1,-}(\partial_1 \phi),\dots, H^{K,-}(\partial_K \phi))=-A_0+\max\{A_0, \min_pH^{0}(p), \dots, \min_pH^{K}(p)\}=0
$$
and so $\tilde \nu_e$ is a sub-solution.
Finally, for $t=0$, since  $A_0\ge H^k(-1/e^k)$, we have $\tilde \nu_e(x,k,0)=-x/e^k<\phi_{A_0}(x,k)$.
\end{proof}

\begin{Lemma}\label{lem:app2}
Assume that $\bar \vartheta_e<-H^0(-1/e^0)$ and $H^0$ is convex. Then, the solution of 
\be\label{ununJuncTOT-app}
\left\{\begin{array}{l}
\ds \partial_t w +H^0(\partial_x w)= 0 \; {\rm in }\; (-\infty, 0)\times (0,+\infty)\\ 
\ds w(x,0)= -x/e^0 \; {\rm in }\; (-\infty, 0]\\
w(0,t)= \bar \vartheta_e t\; {\rm for}\;   t\geq 0.
\end{array}\right. 
\ee
 is unique and given by 
\be\label{eqwwwwTOT-app}
w(x,t):= \min\{ -x/e^0-H^0(-1/e^0)t\ , \ p^{0,-}_{-\bar \vartheta_e} x+\bar \vartheta_et\} \; {\rm in }\; (-\infty, 0)\times (0,+\infty).
\ee

\end{Lemma}
\begin{proof}
The proof is similar to the one of Lemma \ref{lem:app1}. Indeed it is sufficient to remark that, by \cite[Proposition 2.12]{IM}, $w$ is solution of \eqref{ununJuncTOT-app} iff $w$ is solution of 
\be
\left\{\begin{array}{l}
\ds \partial_t w +H^0(\partial_x w)= 0 \; {\rm in }\; (-\infty, 0)\times (0,+\infty)\\ 
\ds w(x,0)= -x/e^0 \; {\rm in }\; (-\infty, 0]\\
\ds \partial_t w + \max \{ A_e, H^{0,+}(\partial w)\}=0\qquad {\rm at}\; x=0,
\end{array}\right. 
\ee
where $A_e=-\vartheta_e$.

\end{proof}

\subsection{Homogenization outside the junction} \label{subsec:homog}

We consider a family of solutions $(U^\ep_i)$ of \eqref{eq.SystJuncTOT} and define $\nu^\ep$  from  $(U^\ep_i)$ as in Section \ref{section.flux}. Let us also fix a set  $(a,b)\times \{k\}\times [t_0,t_1]$ with $t_0<t_1$, $a<b<0$ if $k=0$ and  $b>a>0$ if $k\in \{1, \dots, K\}$. The following result is an easy adaptation of \cite{CaFo}. 

\begin{Lemma}\label{lem.homobasic} There is a set $\Omega_0$ of full probability (independent of $U^\ep$) such that, for any $\omega\in \Omega$, if $\nu^\ep$ is bounded above (respectively below) on $(a,b)\times \{k\}\times [t_0,t_1]$, then any half-relaxed upper limit (resp.  half-relaxed lower limit) of $\nu^\ep$ as $\ep\to 0$ (possibly up to a subsequence)  is a viscosity subsolution (resp. supersolution) of the Hamilton-Jacobi equation 
$$
\partial_t \nu(\cdot, k,\cdot)  +H^k(\partial_x\nu(\cdot, k,\cdot) )=0 \; {\rm in}\; (a,b)\times \{k\}\times [t_0,t_1].
$$
\end{Lemma}

%
%

\subsection{Convexity of the effective Hamiltonians}

\begin{Lemma}\label{lem:convexity} Assume that the $\tilde V^0_z$ are concave on $[\Delta_{\min}, +\infty)$ for any $z\in \mathcal Z$. Then $\bar V^0$ is also concave in $[\Delta_{\min}, +\infty)$ and $H^0$ is convex in $[-1/\Delta_{\min}, 0]$. In the same way, the $H^k$ are convex on $[-\pi^k/\Delta_{\min}, 0]$ for any $k\in \{1, \dots, K\}$. 
\end{Lemma}

\begin{proof} Recall that a one-to-one map $\phi:I\to J$ (where $I$ and $J$ are open intervals) is increasing and concave if and only if $\phi^{-1}$ is increasing and convex. Thus the maps $(\tilde V^0_z)^{-1}$ (for $z\in \mathcal Z$) are increasing and  convex from $(0,\min_{z'} h^0_{\max,z'})$ to $(\Delta_{\min}, \bar e^0)$. So $v\to \E\left[(\tilde V^0_z)^{-1}(v)\right]$ is also  increasing and  convex from $(0,\min_{z'} h^0_{\max,z'})$ to $(\Delta_{\min}, \bar e^0)$.  This shows that its inverse $\bar V^0$ is increasing and concave from $(\Delta_{\min}, \bar e^0)$  to $(0,\min_{z'} h^0_{\max,z'})$. As $\bar  V^0$ is continuous and is constant after $\bar e^0$, we infer that $\bar V^0$ is concave on $[\Delta_{\min},+\infty)$. Finally, as $H^0(p)=p\bar V^0(-1/p)$ on $(-1/\Delta_{\min},0)$, $H^0$ is convex on this interval: indeed,  if $H^0$ and $\bar V^0$ are smooth, then 
$(H^0)''(p)= p^{-3}(\bar V^0)''(-1/p)\geq 0$; the general case can be treated by approximation. 
\end{proof}

\bibliographystyle{siam}

 \end{document}